\documentclass[11pt]{article}

\usepackage[T1]{fontenc}
\usepackage{lmodern}
\usepackage[a4paper,margin=1.05in]{geometry}
\usepackage{amsmath,amssymb,amsthm,mathtools}
\usepackage{bm}
\usepackage{enumitem}
\usepackage{microtype}
\usepackage[colorlinks=true,linkcolor=blue,citecolor=blue,urlcolor=blue]{hyperref}
\hypersetup{
  pdftitle={A Uniform Pole-Subtracted Limiting Absorption Principle for High-Contrast Elastic Resonator Clusters},
  pdfauthor={Yixian Gao},
  pdfsubject={High-contrast elastic resonators and limiting absorption},
  pdfkeywords={limiting absorption principle, high-contrast elasticity, subwavelength resonances, elastic scattering}
}

\numberwithin{equation}{section}

\theoremstyle{plain}
\newtheorem{theorem}{Theorem}[section]
\newtheorem{lemma}[theorem]{Lemma}
\newtheorem{proposition}[theorem]{Proposition}
\newtheorem{corollary}[theorem]{Corollary}
\theoremstyle{definition}
\newtheorem{assumption}[theorem]{Assumption}

\theoremstyle{remark}
\newtheorem{remark}[theorem]{Remark}

\newcommand{\R}{\mathbb R}
\newcommand{\C}{\mathbb C}

\newcommand{\dd}{\,\mathrm d}
\newcommand{\ii}{\mathrm i}
\newcommand{\eps}{\varepsilon}
\newcommand{\tr}{\operatorname{tr}}
\newcommand{\supp}{\operatorname{supp}}
\newcommand{\dist}{\operatorname{dist}}
\newcommand{\rank}{\operatorname{rank}}
\newcommand{\diag}{\operatorname{diag}}
\newcommand{\Span}{\operatorname{span}}
\newcommand{\Ker}{\operatorname{ker}}

\newcommand{\cL}{\mathcal L}
\newcommand{\cH}{\mathcal H}
\newcommand{\cR}{\mathcal R}
\newcommand{\cN}{\mathcal N}
\newcommand{\cM}{\mathcal M}
\newcommand{\cE}{\mathcal E}

\newcommand{\cD}{\mathcal D}
\newcommand{\cB}{\mathcal B}
\newcommand{\cA}{\mathcal A}
\newcommand{\cF}{\mathcal F}
\newcommand{\cV}{\mathcal V}
\newcommand{\cX}{\mathcal X}

\newcommand{\cZ}{\mathcal Z}
\newcommand{\cJ}{\mathcal J}
\newcommand{\cT}{\mathcal T}
\newcommand{\cW}{\mathcal W}
\newcommand{\bfnu}{\bm\nu}

\newcommand{\bfu}{\bm u}
\newcommand{\bfv}{\bm v}
\newcommand{\bfw}{\bm w}
\newcommand{\bff}{\bm f}

\newcommand{\bfphi}{\bm \phi}

\newcommand{\inner}[2]{\left\langle #1,#2\right\rangle}
\newcommand{\norm}[1]{\left\|#1\right\|}

\title{\Large\bfseries A Uniform Pole-Subtracted Limiting Absorption Principle\\
for High-Contrast Elastic Resonator Clusters}

\author{Yixian Gao
\thanks{School of Mathematics and Statistics, Center for Mathematics and
Interdisciplinary Sciences, Northeast Normal University, Changchun,
Jilin 130024, China. 
    \texttt{gaoyx643@nenu.edu.cn}.}}
\date{}

\begin{document}

\maketitle

\begin{abstract}
We establish a uniform pole-subtracted limiting absorption principle for a fixed cluster of
\(N\) disjoint three-dimensional high-contrast elastic resonators when the
contrast tends to infinity and \(\omega=\delta^{1/2}\tau\) approaches the zero
threshold.  The exterior Dirichlet-to-Neumann map and a variational
Grushin--Feshbach reduction give an exact decomposition of the cutoff resolvent
into a uniformly bounded regular part and a finite-rank term governed by
\[
  \cM_\delta^\pm(\omega)
  =
  \delta K-\omega^2I_m
  \mp\ii\delta\omega\Gamma_0
  +\mathcal O(\delta^2+\delta\omega^2).
\]
The cutoff-resolvent norm is uniformly equivalent to
\(1+\|(\cM_\delta^\pm(\omega))^{-1}\|\); hence the finite-dimensional channel
carries every loss of uniformity.  An elastic optical identity factors the
leading radiation matrix through one total-force map into \(\C^3\).  Thus
\(\rank\Gamma_0=3\) and \(\dim\Ker\Gamma_0=6N-3\).  Compression to a static
eigenspace of dimension \(r\) leaves at most three leading radiative channels.
Simple bright poles have width \(O(\delta)\), whereas force-dark poles with a
positive second radiation form have width \(O(\delta^2)\); the corresponding
real-axis peaks have orders \(\delta^{-3/2}\) and \(\delta^{-5/2}\).  We also
treat multiple static eigenvalues, compute the spherical coefficients, and
derive a conditional two-parameter crossover for a symmetry-broken dimer.
\end{abstract}



\section{Introduction}

\subsection{The joint threshold--contrast problem}

We consider a fixed cluster of finitely many disjoint elastic inclusions in
three dimensions.  Both the stiffness and density contrasts are
\(\delta^{-1}\), where \(0<\delta\ll1\).  For a fixed uniformly elliptic
medium, the limiting absorption principle selects the outgoing or incoming
solution by taking real-frequency boundary values of the resolvent.  Here this
absorption limit interacts with the singular contrast limit.  The multi-body
rigid space has dimension \(6N\), while the associated subwavelength resonances
occur at frequencies
\[
  \omega\asymp\delta^{1/2},
\]
and therefore approach the zero threshold at the same time as the
material contrast becomes singular.  Fixed-contrast resolvent
estimates do not control this regime.

The central object is
\[
  R_\delta(z)=(\cL_\delta-z)^{-1},
  \qquad
  z=(\omega\pm\ii\eps)^2,
\]
where \(\cL_\delta\) is the self-adjoint transmission operator defined
in Section~\ref{sec:model}.  We study
\[
  \eps\downarrow0,\qquad
  \delta\downarrow0,\qquad
  \omega=\delta^{1/2}\tau,
\]
with \(\tau\) in a fixed compact subset of \((0,\infty)\).  The
absorption limit cannot be interchanged uniformly with the
high-contrast limit along a subwavelength resonant path.

\subsection{Standing assumptions and main conclusions}

The uniform statements in this paper are made for a fixed geometric and material
configuration.  We record the assumptions here so that the dependence of the
constants is unambiguous.

\begin{assumption}
\label{ass:standing}
Let \(N\ge1\).  The inclusions \(D_1,\ldots,D_N\Subset\R^3\) are pairwise disjoint connected
\(C^{2,\alpha}\) domains, \(0<\alpha<1\).  Set
\[
  D:=\bigcup_{j=1}^N D_j,
  \qquad
  \Omega:=\R^3\setminus\overline D.
\]
Assume that \(\Omega\) is connected and that the cluster is fixed as
\(\delta\downarrow0\).  In particular,
\[
  d_*:=\min_{i\ne j}\dist(\overline D_i,\overline D_j)>0
  \quad (N\ge2),
\]
with the convention \(d_*=+\infty\) when \(N=1\).
The exterior tensor is homogeneous isotropic and strongly elliptic, with density
\(\rho_0>0\).  On each component \(D_j\), the reference tensor \(C_D\) is
homogeneous isotropic and strongly elliptic, and the reference density is a
positive constant \(\rho_{D,j}\).  We write \(\rho_D=\rho_{D,j}\) on \(D_j\).
This concrete assumption guarantees the interior Cauchy unique-continuation
property used in the fixed-contrast limiting absorption argument.  The
high-contrast coefficients
are
\[
  C_\delta=\delta^{-1}C_D\mathbf1_D+C_0\mathbf1_\Omega,\qquad
  \rho_\delta=\delta^{-1}\rho_D\mathbf1_D+\rho_0\mathbf1_\Omega.
\]
The constants below may depend on this fixed cluster, on the material parameters,
on the compact rescaled-frequency set, and on the cutoff, but not on
\(\delta\), the absorption parameter, or the real frequency on the scale
\(\omega=\delta^{1/2}\tau\).
\end{assumption}

The \(C^{2,\alpha}\) regularity in Assumption~\ref{ass:standing} is used only for
the operator-norm low-frequency expansion of the elastic boundary operators and
for the explicit pointwise calculations in Section~\ref{sec:examples}.  The variational
Grushin--Feshbach reduction uses only the Korn decomposition, trace theory, and
the exterior Dirichlet-to-Neumann map; see Remark~\ref{rem:boundary-regularity}.
A simultaneous closing-gap limit is not included in the uniformity class.

Let \(\cN\) be the \(6N\)-dimensional space of fields that are rigid on each
component.  After choosing a mass-orthonormal basis of \(\cN\), put \(m=6N\),
let \(K>0\) be the static elastic capacitance matrix, and let
\(\Gamma_0\ge0\) be the leading radiation matrix.

The theorem below collects the main conclusions.  The maps and matrices used
in parts~\textup{(iii)}--\textup{(iv)} are defined explicitly in
Sections~\ref{sec:reduction} and~\ref{sec:radiation}; the full statements cited after the
theorem specify all operator topologies and nondegeneracy conditions.

\begin{theorem}
\label{thm:intro-main}
Under Assumption~\ref{ass:standing}, fix a compact interval
\(I=[\tau_-,\tau_+]\Subset(0,\infty)\), \(s>1/2\),
and \(\chi\in C_c^\infty(\R^3)\) equal to one on a neighborhood of
\(\overline D\).  There are \(\delta_0,C,c>0\) such that, for
\[
  0<\delta<\delta_0,\qquad
  \omega=\delta^{1/2}\tau,\qquad \tau\in I,
\]
the following assertions hold.

\begin{enumerate}[label=\textup{(\roman*)},leftmargin=2.2em]
\item The outgoing and incoming effective matrices satisfy
\[
  \cM_\delta^\pm(\omega)
  =
  \delta K-\omega^2I_m
  \mp\ii\delta\omega\Gamma_0
  +E_\delta^\pm(\omega),
  \qquad
  \norm{E_\delta^\pm(\omega)}
  \le C(\delta^2+\delta\omega^2),
\]
with
\[
  \norm{\partial_\omega E_\delta^\pm(\omega)}
  \le C(\delta^2+\delta\omega),
  \qquad
  \norm{\partial_\omega^2 E_\delta^\pm(\omega)}
  \le C\delta.
\]

\item The cutoff boundary values admit the exact decomposition
\[
  \chi R_\delta^\pm(\omega)\chi
  =
  \cR_{\delta,\mathrm{reg}}^\pm(\omega)
  +
  \chi\Phi_\delta^\pm(\omega)
  \bigl(\cM_\delta^\pm(\omega)\bigr)^{-1}
  \Psi_\delta^\pm(\omega)\chi,
\]
where
\[
  \cR_{\delta,\mathrm{reg}}^\pm(\omega)
  \in\mathcal B(L_s^2,H_{-s}^1),\qquad
  \chi\Phi_\delta^\pm(\omega)
  \in\mathcal B(\C^m,H_{-s}^1),\qquad
  \Psi_\delta^\pm(\omega)\chi
  \in\mathcal B(L_s^2,\C^m)
\]
are uniformly bounded.  Moreover,
\[
  c\bigl(1+\norm{(\cM_\delta^\pm(\omega))^{-1}}\bigr)
  \le
  1+\norm{\chi R_\delta^\pm(\omega)\chi}
  \le
  C\bigl(1+\norm{(\cM_\delta^\pm(\omega))^{-1}}\bigr).
\]
Thus every loss of uniformity is carried by the finite-dimensional effective
matrix.

\item The leading radiation matrix satisfies
\[
  \rank\Gamma_0=3,\qquad
  \Ker\Gamma_0=\Ker(\mathfrak qB),\qquad
  \dim\Ker\Gamma_0=6N-3.
\]
The kernel consists of leading force-dark rigid directions.  Such a direction
produces a dark resonant branch only when it is compatible with the corresponding
eigenspace of \(K\).  If
\(\cX_\lambda=\Ker(K-\lambda I)\) has dimension \(r_\lambda\) and
\(P_\lambda\) is its orthogonal projection, then
\[
  \rank(P_\lambda\Gamma_0P_\lambda)\le3,
  \qquad
  \dim(\cX_\lambda\cap\Ker\Gamma_0)\ge r_\lambda-3.
\]
Hence, counting algebraic multiplicity, at least
\(\max\{r_\lambda-3,0\}\) poles in that cluster have vanishing
order-\(\delta\) imaginary coefficient.

\item Let \(Ka_j=\lambda_ja_j\) be simple with \(|a_j|=1\) and
\(\lambda_j^{1/2}\in\operatorname{int}I\).  If
\(\inner{\Gamma_0a_j}{a_j}>0\), then the corresponding pole is bright and
\[
  \zeta_j(\delta)
  =
  \delta^{1/2}\lambda_j^{1/2}
  -
  \frac{\ii\delta}{2}
  \inner{\Gamma_0a_j}{a_j}
  +\mathcal O(\delta^{3/2}).
\]
If instead \(a_j\in\Ker\Gamma_0\) and
\(\inner{\Gamma_1a_j}{a_j}>0\), then after a real shift of order
\(\delta^{3/2}\),
\[
  \operatorname{Im}\zeta_j(\delta)
  =
  -\frac{\delta^2\lambda_j}{2}
  \inner{\Gamma_1a_j}{a_j}
  +o(\delta^2).
\]
The corresponding real-axis peaks are of orders \(\delta^{-3/2}\) and
\(\delta^{-5/2}\), respectively.
\end{enumerate}
\end{theorem}

The results cited in the theorem give the precise operator topologies,
non-cancellation estimates, multiple-eigenvalue analysis, and Lorentzian
expansions.  The rank-three law is global rather than componentwise: spectral
conclusions follow only after the radiation form is compressed to a static
eigenspace; see Corollary~\ref{cor:cluster-dark-count}.

\subsection{Contributions}

The fixed-\(\delta\) elastic limiting absorption principle is classical; the
problem here is uniformity in the joint threshold--contrast limit.  After the
homogeneous exterior is eliminated, the rigid channel is separated from its
coercive complement by an exact Grushin--Feshbach identity.  The projection
acts only on \(L^2(D)^3\), avoiding non-square-integrable static exterior
fields.  Quantitative nondegeneracy of the source and reconstruction maps then
prevents cancellation of the finite-rank singularity.

The reduction also preserves radiative flux.  For the corrected trace
\(B_\delta(\omega)=\gamma S_\delta^r(\omega)\),
\[
  -\operatorname{Im}_{\mathrm{op}}\cM_\delta^+(\omega)
  =
  \delta\omega
  B_\delta(\omega)^*\cF(\omega)^*\cF(\omega)B_\delta(\omega).
\]
The two elastic polarizations share the same leading total-force vector, giving
\[
  \Gamma_0
  =
  \gamma_{\mathrm{el}}B^*\mathfrak q^*\mathfrak qB
\]
and hence the rank-three law.  Only its compression to an eigenspace of \(K\)
determines bright and force-dark poles.  On compatible dark eigendirections,
the next far-field coefficient yields a nonnegative form \(\Gamma_1\) and the
second radiative scale.

The spherical calculation identifies both scales explicitly.  In the dimer,
exchange-odd modes are force-dark; under an explicit transversality condition,
a symmetry-breaking perturbation produces the crossover
\(O(\delta^2)+O(\delta\eta^2)\).  The pole-subtraction argument applies more
generally to analytic Fredholm families with an exact finite-dimensional
reduction and a uniformly invertible complement; the optical factorization
and radiation hierarchy are specific to elastic scattering.

\subsection{Relation to earlier work}

Foundational weighted resolvent frameworks include
\cite{Agmon1975,AgmonHormander1976}; threshold resolvent expansions are
developed in \cite{JensenNenciu2001}, with the correction
\cite{JensenNenciu2004}, and threshold forms of the Fermi golden rule in
\cite{JensenNenciu2006}.  Weighted limiting absorption
estimates for the constant-coefficient Navier operator were proved in
\cite{Barcelo2012}, while two-dimensional inhomogeneous anisotropic elastic systems were
treated, under different hypotheses, in \cite{NakamuraWang2006}.  These
works keep the operator fixed.  Here the rigid eigenvalues approach the
continuous-spectrum threshold at the same rate at which the material
contrast becomes singular.
Classical resonance results for exterior elastic obstacles include
\cite{StefanovVodev1996}; that setting is an exterior Neumann problem rather
than the high-contrast transmission problem considered here.

The mathematical theory of high-contrast phononic media includes the
spectral and layer-potential analysis of \cite{AmmariKangLee2009} and,
more recently, the hard-inclusion resonance theory of \cite{LiZou2025}.
Quasi-Minnaert elastic resonances and their relation to stress concentration
were studied in \cite{DiaoTangLiu2026}; their incident-field mechanism differs
from the resolvent poles considered below.
The three-dimensional variational formulation in \cite{ChenGaoLiRen2025}
locates subwavelength resonances and derives frequency and field asymptotics.
For multiple elastic scatterers, capacitance-type finite-dimensional systems
and Foldy--Lax interactions were developed in \cite{ChallaSini2015}; resonant
multiple inclusions with large mass density were subsequently analyzed in
\cite{ChallaGangadaraiahSini2024}.  These small-body and density-contrast
regimes provide important precedents for finite-dimensional elastic
interaction laws, but they do not address the joint real-axis
absorption--contrast limit studied here.

Uniform contrast-dependent resolvent estimates were obtained for a single
acoustic Minnaert resonator in \cite{LiSiniUniform2026}.  The present problem
instead has an elastic \(6N\)-dimensional rigid channel and two coupled
polarizations; the cluster radiation law is controlled by one global force
map into \(\C^3\).

A particularly close recent work is Li and Sini \cite{LiSini2026}, who study a
single three-dimensional Lam\'e transmission inclusion whose Lam\'e moduli and
density have the same high-contrast scaling.  Near the zero threshold they
obtain subwavelength resonances with real part of order the square root of the
inverse-contrast parameter, identify generic and exceptional radiative widths
of first and second order, derive finite-rank resolvent asymptotics, and handle
eigenvalue multiplicities.  Thus the present paper does not claim novelty for
those pole scales, finite-rank asymptotics, or multiplicity analysis by
themselves.  Its distinct objective is the uniform
pole-subtracted limiting absorption principle for a fixed cluster of
\(N\) inclusions, including a two-sided cutoff-resolvent norm law and
non-cancellation of the finite-rank channel.  The additional cluster feature is
the global total-force factorization of the leading radiation form.  Its
spectral consequences, together with the two-sided norm law, distinguish the
present result from a juxtaposition of one-body expansions; the novelty claim
is not the finite-dimensional perturbation procedure itself.

Complementary work treats time-domain elastic modal
approximations \cite{ChenGaoLiu2022,ChenGaoLiLiu2024}, periodic phononic band
gaps \cite{RenChenGaoLi2026}, two-dimensional high-contrast resonators
\cite{RenGao2025}, and the multilayer two-dimensional problem
\cite{JiangLiuSunWang2026}.  Other recent developments include resonant modes and
stress concentration for elastic dimers \cite{LiXu2026} and close-to-touching
two-dimensional configurations \cite{LiXuYang2026}.  The
Grushin--Feshbach argument follows the effective-Hamiltonian framework of
\cite{SjostrandZworski2007,DyatlovZworski2019}; the low-frequency boundary
analysis uses elastic layer potentials and DtN maps as in
\cite{GachterGrote2003,McLean2000,AmmariElasticity2015,VodickaMantic2004}.

\subsection{Organization and notation}

Section~\ref{sec:model} defines the transmission problem and eliminates the
exterior.  Section~\ref{sec:reduction} constructs the rigid channel and carries
out the uniform Grushin--Feshbach reduction.  Section~\ref{sec:lap} proves the
pole-subtracted LAP and the sharp norm equivalence.
Section~\ref{sec:radiation} derives the optical identity, radiation hierarchy,
pole asymptotics, and Lorentzian laws.  Section~\ref{sec:examples} treats the
sphere and the conditional dimer crossover.  Section~\ref{sec:conclusion}
records the scope of the result.

All Hilbert-space inner products are linear in the first argument and
anti-linear in the second.  Duality on a boundary \(\Sigma\) is denoted
by \(\inner{\cdot}{\cdot}_{\Sigma}\), extending the \(L^2(\Sigma)\)
pairing.  Dual spaces such as \(\cV^*\) and \(H^{-1/2}(\Sigma)^3\) are
understood as anti-duals, so the represented operators are linear in their
input.  Operator adjoints are taken with these conventions.  Constants
may change from line to line and are independent of \(\delta\),
\(\eps\), and \(\omega=\delta^{1/2}\tau\), unless stated otherwise.
We use
\[
  \langle x\rangle=(1+|x|^2)^{1/2},\qquad
  \operatorname{Re}_{\mathrm{op}}A=\frac{A+A^*}{2},\qquad
  \operatorname{Im}_{\mathrm{op}}A=\frac{A-A^*}{2\ii}.
\]
The symbol \(\zeta\) always denotes a complex frequency, \(\omega>0\) a
real frequency, and \(z=\zeta^2\) the spectral-energy variable.  Complex
resonance poles are therefore denoted by \(\zeta_j\), whereas a real
resonant center is denoted by \(\widehat\omega_j\).
All norms on \(\C^m\) are Euclidean.  Unless a smaller parameter set is
displayed, \(\mathcal O(\cdot)\) bounds are uniform for \(\tau\) in a fixed
compact subset of \((0,\infty)\), and \(o(\cdot)\) is understood in the
joint limit specified in the corresponding statement.  The central operator
types are
\[
\begin{array}{ccl}
\cA_\delta(\zeta)&:&\cV=H^1(D)^3\longrightarrow\cV^*=(H^1(D)^3)^*,\\
G_\delta(\zeta)&:&\cV_\perp\longrightarrow\cV_\perp^*,\\
\cM_\delta(\zeta)&:&\C^m\longrightarrow\C^m,\\
B=\gamma_\Sigma J&:&\C^m\longrightarrow H^{1/2}(\Sigma)^3.
\end{array}
\]

\section{The Concrete Transmission Model and Exterior Reduction}\label{sec:model}

\subsection{Geometry, coefficients, and the self-adjoint operator}

Let \(N\ge1\), and let \(D_1,\ldots,D_N\Subset\R^3\) be pairwise disjoint, bounded,
connected domains with \(C^{2,\alpha}\) boundaries, and assume that
\[
  \Omega:=\R^3\setminus\overline D,\qquad
  D:=\bigcup_{j=1}^ND_j,
\]
is connected.  Put \(\Sigma=\partial D\), and let \(\bfnu\) denote the
unit normal pointing from \(D\) into \(\Omega\).

The exterior tensor is homogeneous and isotropic:
\[
  C_0A=2\mu_0A+\lambda_0\tr(A)I
  \quad\text{for symmetric }A,
\]
where
\[
  \mu_0>0,\qquad 3\lambda_0+2\mu_0>0.
\]
The exterior density is the positive constant \(\rho_0\).
On each \(D_j\), the interior tensor is a fixed homogeneous isotropic
tensor
\[
  C_DA=2\mu_{D,j}A+\lambda_{D,j}\tr(A)I,
  \qquad
  \mu_{D,j}>0,\quad
  3\lambda_{D,j}+2\mu_{D,j}>0,
  \qquad x\in D_j.
\]
This concrete choice guarantees the interior unique-continuation
property used below.  The arguments extend to sufficiently regular
anisotropic tensors once that property is imposed explicitly.  The density is
constant and positive on each component:
\(\rho_D|_{D_j}=\rho_{D,j}>0\).  More general densities may be admitted if the
corresponding interior Navier system has the Cauchy unique-continuation property;
we do not use that extension below.

The high-contrast tensor and density are
\[
  C_\delta
  =
  \delta^{-1}C_D\mathbf1_D+C_0\mathbf1_\Omega,
  \qquad
  \rho_\delta
  =
  \delta^{-1}\rho_D\mathbf1_D+\rho_0\mathbf1_\Omega.
\]
Thus the stiffness and density contrasts have the same order; a fixed
order-one ratio between the interior and exterior wave speeds can be
absorbed into \(C_D\) and \(\rho_D\).

For
\[
  e(\bfu)=\frac12(\nabla\bfu+\nabla\bfu^{\mathsf T}),
\]
write
\[
  T_0^\nu\bfu=(C_0e(\bfu))\bfnu,
  \qquad
  T_D^\nu\bfu=(C_De(\bfu))\bfnu,
\]
whenever the corresponding traction trace is defined, and denote the
interior displacement trace on \(\Sigma\) by
\(\gamma=\gamma_\Sigma:H^1(D)^3\to H^{1/2}(\Sigma)^3\).  Its transpose
\(\gamma^*:H^{-1/2}(\Sigma)^3\to\cV^*=(H^1(D)^3)^*\) is defined by
\[
  \inner{\gamma^*h}{v}=\inner{h}{\gamma v}_{\Sigma}.
\]
We also use
both tractions in their standard weak
\(H^{-1/2}(\Sigma)^3\) sense.  Consider the closed nonnegative form
\[
  q_\delta[\bfu,\bfv]
  =
  \delta^{-1}
  \int_D C_De(\bfu):\overline{e(\bfv)}\,\dd x
  +
  \int_\Omega C_0e(\bfu):\overline{e(\bfv)}\,\dd x
\]
on \(H^1(\R^3)^3\), in the Hilbert space
\[
  \cH_\delta=L^2(\R^3,\rho_\delta\dd x)^3.
\]
The representation theorem defines a nonnegative self-adjoint operator
\(\cL_\delta\).  Away from \(\Sigma\),
\[
  \cL_\delta\bfu
  =
  -\rho_\delta^{-1}\nabla\cdot(C_\delta e(\bfu)),
\]
and its domain imposes continuity of displacement and physical
traction across \(\Sigma\).

\begin{remark}
\label{rem:one-model}
In this transmission model the restoring matrix is the exterior
elastic capacitance of \(\Sigma\), multiplied by the traction factor
\(\delta\).  No separate connector stiffness is present.  Coated or
spring-connected resonators lead to different static Schur complements
and require their contrast scalings to be derived separately.
\end{remark}

\begin{remark}
\label{rem:boundary-regularity}
The \(C^{2,\alpha}\) assumption is used to invoke the classical
operator-norm Taylor expansions of the elastic layer potentials and to
write the traction calculations in Section~\ref{sec:examples} pointwise.  The
variational reduction itself requires only Lipschitz boundaries, and
the static single-layer isomorphism is available under weaker
hypotheses.  We retain \(C^{2,\alpha}\) so that the low-frequency
boundary expansion and its differentiated remainder hold in the stated
Sobolev topologies without introducing a separate nonsmooth calculus.
\end{remark}

\subsection{Weighted spaces and fixed-contrast boundary values}

For \(s\in\R\), set
\[
  L_s^2(\R^3)^3
  =
  \{\bff:\langle x\rangle^s\bff\in L^2(\R^3)^3\},
  \qquad
  H_s^1(\R^3)^3
  =
  \{\bfu:\langle x\rangle^s\bfu\in H^1(\R^3)^3\}.
\]
These are Lebesgue-based weighted spaces, not spaces weighted by
\(\rho_\delta\).  For each fixed \(\delta\) they embed naturally in the
self-adjoint Hilbert space \(\cH_\delta\); the uniform estimates below
are stated after the interior equation has been normalized by the
fixed density \(\rho_D\).
For \(\omega>0\) and \(\eps>0\), write
\[
  \zeta_{\omega,\eps}^{\pm}=\omega\pm\ii\eps,
  \qquad
  R_\delta(\zeta^2)=(\cL_\delta-\zeta^2)^{-1}.
\]

\begin{proposition}
\label{prop:fixed-lap}
Fix \(\delta>0\), \(s>1/2\), and \(\omega>0\).  Then
\[
  R_\delta^\pm(\omega)
  =
  \lim_{\eps\downarrow0}
  R_\delta\bigl((\zeta_{\omega,\eps}^{\pm})^2\bigr)
\]
exists in
\[
  \mathcal B\bigl(L_s^2(\R^3)^3,H_{-s}^1(\R^3)^3\bigr).
\]
The \(+\) boundary value is outgoing and the \(-\) boundary value is
incoming.  No assertion of uniformity as \(\delta\downarrow0\) is made
in this proposition.
\end{proposition}

\begin{proof}
Let \(\chi_0\in C_c^\infty(\R^3)\) be one on a ball containing
\(\overline D\).  Outside this ball the operator coincides with the homogeneous
Navier operator
\[
  \cL_0\bfu
  =
  -\rho_0^{-1}\nabla\cdot(C_0e(\bfu)).
\]
The longitudinal and transverse Fourier projectors diagonalize its resolvent into
scalar Helmholtz resolvents with wave numbers
\[
  k_p(\zeta)=\zeta\sqrt{\frac{\rho_0}{\lambda_0+2\mu_0}},
  \qquad
  k_s(\zeta)=\zeta\sqrt{\frac{\rho_0}{\mu_0}}.
\]
For \(s>1/2\), the scalar weighted limiting absorption principle therefore gives
boundary values of \((\cL_0-(\omega\pm\ii0)^2)^{-1}\) from \(L_s^2\) to
\(H_{-s}^1\).  To make the gluing step explicit, choose nested cutoffs
\(\chi_0\prec\chi_1\prec\chi_2\), meaning that
\(\chi_{j+1}=1\) on a neighborhood of \(\supp\chi_j\), with
\(\chi_0=1\) near \(\overline D\).  Let
\(P_{\delta,\mathrm{loc}}(\zeta)\) be a standard elliptic parametrix for the
transmission operator in a ball containing \(\supp\chi_2\).  Then
\[
  \mathcal P_\delta^\pm(\omega)
  =
  (1-\chi_1)R_0^\pm(\omega)(1-\chi_0)
  +
  \chi_2P_{\delta,\mathrm{loc}}(\omega)\chi_1
\]
satisfies
\[
  (\cL_\delta-\omega^2)\mathcal P_\delta^\pm(\omega)
  =I+\mathcal K_\delta^\pm(\omega),
\]
where \(\mathcal K_\delta^\pm(\omega)\) is supported in a fixed bounded set and
factors as
\(L_s^2\to H^1(B)^3\hookrightarrow L^2(B)^3\to L_s^2\) for a fixed ball
\(B\).  Rellich compactness therefore gives
\(\mathcal K_\delta^\pm(\omega)\in\mathcal K(L_s^2,L_s^2)\).
The corresponding family for \(\operatorname{Im}\zeta\ne0\) is analytic and
has Fredholm index zero, since it is connected to the physical resolvent away
from the real axis.  Its compact error is norm-continuous down to the real
axis.  Analytic Fredholm theory therefore reduces existence and operator-norm
continuity of the boundary inverse to the absence of a nontrivial outgoing,
respectively incoming, homogeneous transmission field.
This is the Fredholm gluing construction for a compactly supported
transmission perturbation of the constant-coefficient Navier operator; the
local block is the strongly elliptic transmission parametrix and the exterior
block is the free weighted resolvent.

We verify this uniqueness statement.  Let \(\bfu\) be an outgoing solution at a
real frequency \(\omega>0\).  Applying Betti's formula to
\(B_R\setminus\overline D\), taking the imaginary part, and using continuity of
the displacement and physical traction across \(\Sigma\), the interface terms
cancel with the corresponding interior identity.  Passing to \(R\to\infty\)
and using the Kupradze radiation condition gives
\[
  \omega c_p\norm{F_p}_{L^2(\mathbb S^2)}^2
  +
  \omega c_s\norm{F_s}_{L^2(\mathbb S^2)}^2
  =0,
\]
where \(F_p\) and \(F_s\) are the pressure and shear far fields and
\(c_p,c_s>0\) are the elastic impedances introduced later in
\eqref{eq:weighted-far-field-map}.  Thus both far fields vanish.  The elastic
Rellich lemma yields \(\bfu=0\) in the connected exterior \(\Omega\).

The exterior vanishing gives both zero displacement trace and zero exterior
traction on \(\Sigma\).  The transmission conditions therefore give zero
Cauchy data for the interior Navier equation on each component.  For the
piecewise homogeneous isotropic tensors in Assumption~\ref{ass:standing},
unique continuation, equivalently continuation from vanishing Cauchy data,
implies \(\bfu=0\) in every \(D_j\).  Hence the outgoing homogeneous kernel is
trivial.  Complex conjugation gives the same conclusion for the incoming
problem.

The Fredholm inverse consequently has no singularity at the prescribed positive
real frequency.  Combining the localized parametrix with the free weighted
boundary values gives
\[
  R_\delta^\pm(\omega)
  \in\mathcal B(L_s^2,H_{-s}^1),
\]
with convergence in this operator norm as the absorption parameter tends to
zero.  This is a fixed-\(\delta\) statement; none of the constants in this
argument are asserted to remain bounded as \(\delta\downarrow0\).  The weighted
Navier estimates and the elastic Rellich uniqueness used here are standard; see
\cite{Barcelo2012,Kupradze1979,Leis1986}.
\end{proof}

\subsection{The exterior Dirichlet-to-Neumann map}

For \(\operatorname{Im}\zeta>0\) and
\(g\in H^{1/2}(\Sigma)^3\), let \(\cE(\zeta)g\) be the unique solution
of
\[
  -\nabla\cdot(C_0e(\bfu))-\rho_0\zeta^2\bfu=0
  \quad\text{in }\Omega,\qquad
  \gamma_\Sigma\bfu=g,
\]
which decays at infinity.  Its real-frequency boundary value is
outgoing.  We define
\[
  \Lambda(\zeta)g
  =
  (C_0e(\cE(\zeta)g))\bfnu\big|_\Sigma
  \in H^{-1/2}(\Sigma)^3.
\]
Notice that \(\bfnu\) points into the exterior and hence is opposite
to the outward normal of the exterior domain.  This convention gives
\begin{equation}
\label{eq:static-dtn-sign}
  -\inner{\Lambda(0)g}{g}_{\Sigma}
  =
  \int_\Omega
  C_0e(\cE(0)g):\overline{e(\cE(0)g)}\,\dd x
  \ge0.
\end{equation}

Let \(\mathbf G^\zeta\) be the outgoing Kupradze tensor.  We distinguish
the volume single-layer potential from its boundary trace:
\[
  \mathbf{SL}(\zeta)\mu(x)
  :=\int_\Sigma \mathbf G^\zeta(x-y)\mu(y)\,\dd S_y,
  \qquad
  \mathsf S(\zeta):=\gamma_\Sigma\mathbf{SL}(\zeta).
\]
Thus \(\mathbf{SL}(\zeta)\mu\) is a field on
\(\R^3\setminus\Sigma\), whereas
\(\mathsf S(\zeta):H^{-1/2}(\Sigma)^3\to H^{1/2}(\Sigma)^3\) is the
boundary single-layer operator.

\begin{proposition}
\label{prop:dtn-analytic}
There is \(r_0>0\) such that the outgoing Dirichlet-to-Neumann family
\[
  \Lambda(\zeta):
  H^{1/2}(\Sigma)^3\longrightarrow H^{-1/2}(\Sigma)^3
\]
extends holomorphically from \(\operatorname{Im}\zeta>0\) to
\(|\zeta|<r_0\).  The following expansion holds in the operator norm of
\[
  \mathcal B\bigl(H^{1/2}(\Sigma)^3,H^{-1/2}(\Sigma)^3\bigr):
\]
\begin{equation}
\label{eq:dtn-basic-expansion}
  \Lambda(\zeta)
  =
  \Lambda_0+\ii\zeta\Lambda_1+\zeta^2\Lambda_2(\zeta),
\end{equation}
where \(\Lambda_0=\Lambda(0)\), \(\Lambda_1=\Lambda_1^*\ge0\), and
\(\Lambda_2\) is holomorphic.  After decreasing \(r_0\) if necessary,
\[
  \sup_{|\zeta|<r_0}
  \left(
    \norm{\Lambda_2(\zeta)}
    +
    \norm{\partial_\zeta\Lambda_2(\zeta)}
    +
    \norm{\partial_\zeta^2\Lambda_2(\zeta)}
  \right)
  <\infty.
\]
For real \(\omega\) with \(|\omega|<r_0\),
\[
  \Lambda^-(\omega)=\Lambda^+(\omega)^*,
  \qquad
  \Lambda^+(-\omega)=\Lambda^-(\omega).
\]
Here and below \(\Lambda^+_{\mathrm{an}}(\zeta):=\Lambda(\zeta)\) denotes
the outgoing holomorphic germ, while
\(\Lambda^-_{\mathrm{an}}(\zeta):=\Lambda(-\zeta)\) denotes the incoming
germ.  Their restrictions to \(\omega>0\) are the outgoing and incoming
boundary maps \(\Lambda^\pm(\omega)\).  Values of the outgoing germ in the
lower half-plane are analytic continuations, not decaying exterior solutions.
\end{proposition}

\begin{proof}
Write the outgoing Kupradze tensor as
\[
  \mathbf G^\zeta(x)
  =
  \frac1{\mu_0}\Phi_{k_s(\zeta)}(x)I
  +
  \frac1{\rho_0\zeta^2}
  \nabla\nabla^{\mathsf T}
  \bigl(\Phi_{k_s(\zeta)}-\Phi_{k_p(\zeta)}\bigr)(x),
  \qquad
  \Phi_k(x)=\frac{e^{\ii k|x|}}{4\pi|x|}.
\]
The apparent singularity at \(\zeta=0\) is removable.  The scalar Taylor
expansion of \(\Phi_k\) shows that the terms which could produce negative
powers of \(\zeta\) disappear after the two spatial derivatives.  More
precisely, Appendix~\ref{app:dtn} proves the kernel expansion, uniformly with the
spatial derivatives required for the single-layer and traction traces.  It
follows that
\[
  \mathsf S(\zeta):
  H^{-1/2}(\Sigma)^3\to H^{1/2}(\Sigma)^3,
  \qquad
  \mathsf K^*(\zeta):
  H^{-1/2}(\Sigma)^3\to H^{-1/2}(\Sigma)^3
\]
are holomorphic operator families and have Taylor remainders which may be
differentiated twice in \(\zeta\) without loss of uniformity; see
Lemma~\ref{lem:boundary-operator-expansion}.

The static single-layer operator \(\mathsf S_0=\mathsf S(0)\) is Fredholm of
index zero in three dimensions.  If \(\mathsf S_0\mu=0\), the corresponding
single-layer potential has zero trace; the interior and decaying exterior
energy identities force the potential to vanish, and the traction jump then
gives \(\mu=0\).  Thus \(\mathsf S_0\) is invertible.  This also identifies
precisely the static decay condition which removes the rigid-motion ambiguity
of the full-space elasticity equation.  Hence \(\mathsf S(\zeta)^{-1}\) is
holomorphic for \(|\zeta|<r_0\), and the jump relation yields
\[
  T_0^\nu\mathbf{SL}(\zeta)\mu\big|_{\Omega}
  =\left(-\frac12I+\mathsf K^*(\zeta)\right)\mu,
  \qquad
  T_0^\nu\mathbf{SL}(\zeta)\mu\big|_{D}
  =\left(\frac12I+\mathsf K^*(\zeta)\right)\mu.
\]
Here the same normal \(\bfnu\), pointing from \(D\) into \(\Omega\), is used
on both sides.  Consequently,
\[
  \Lambda(\zeta)
  =
  \left(-\frac12I+\mathsf K^*(\zeta)\right)
  \mathsf S(\zeta)^{-1},
\]
and the exterior-minus-interior traction jump equals \(-\mu\).  The
operator-norm expansion and the differentiated remainder
bound follow from Proposition~\ref{prop:dtn-detailed}.

For real frequency, complex conjugation interchanges outgoing and incoming
solutions, giving the two symmetry relations.  Finally, the large-sphere Betti
identity gives, for real \(\omega>0\),
\[
  \operatorname{Im}
  \inner{\Lambda^+(\omega)g}{g}_{\Sigma}\ge0.
\]
Dividing by \(\omega\) and using \eqref{eq:dtn-basic-expansion} gives
\(\inner{\Lambda_1g}{g}_{\Sigma}\ge0\).  The same symmetry relation shows that
\(\Lambda_1\) is self-adjoint.  Section~\ref{sec:radiation} identifies this coefficient explicitly
as a far-field Gram operator.
\end{proof}

\subsection{Exact reduction to the bounded inclusion domain}

Set
\[
  \cV=H^1(D)^3,\qquad
  \mathfrak m(\bfu,\bfv)
  =
  \int_D\rho_D\bfu\cdot\overline{\bfv}\,\dd x.
\]
For \(\operatorname{Im}\zeta>0\), define
\(\cA_\delta(\zeta):\cV\to\cV^*\) by
\begin{equation}
\label{eq:interior-form}
  \inner{\cA_\delta(\zeta)\bfu}{\bfv}
  =
  \int_D C_De(\bfu):\overline{e(\bfv)}\,\dd x
  -
  \zeta^2\mathfrak m(\bfu,\bfv)
  -
  \delta
  \inner{\Lambda(\zeta)\gamma\bfu}{\gamma\bfv}_{\Sigma}.
\end{equation}
The outgoing and incoming real-frequency forms are denoted
\(\cA_\delta^\pm(\omega)\).

Let \(R_{\Omega,D}(\zeta)\) denote the solution operator for
\[
  -\nabla\cdot(C_0e(\bfw))-\rho_0\zeta^2\bfw=F
  \quad\text{in }\Omega,
  \qquad \gamma_\Sigma\bfw=0,
\]
with the decaying or outgoing condition according to \(\zeta\).  If
\(f\) is compactly supported, put
\[
  \bfw_f(\zeta)
  =
  R_{\Omega,D}(\zeta)(\rho_0f|_\Omega)
\]
and define the source functional
\begin{equation}
\label{eq:source-functional}
  \inner{\cJ_\delta(\zeta)f}{\bfv}
  =
  \int_D\rho_D f\cdot\overline{\bfv}\,\dd x
  +
  \delta
  \inner{T_0^\nu\bfw_f(\zeta)}{\gamma\bfv}_{\Sigma}.
\end{equation}

\begin{proposition}
\label{prop:bounded-reduction}
Let \(\operatorname{Im}\zeta>0\), and let \(f\in L_s^2(\R^3)^3\) be compactly
supported.  The field
\(\bfu=R_\delta(\zeta^2)f\) is determined by the solution
\(u=\bfu|_D\in\cV\) of
\begin{equation}
\label{eq:bounded-reduced-equation}
  \cA_\delta(\zeta)u=\cJ_\delta(\zeta)f.
\end{equation}
The exterior part is reconstructed by
\begin{equation}
\label{eq:reconstruction}
  \bfu|_\Omega
  =
  \cE(\zeta)\gamma u+\bfw_f(\zeta).
\end{equation}
For every cutoff \(\chi\) supported in a fixed ball and every
\(s>1/2\), the maps
\[
  \cJ_\delta(\zeta)\chi:L_s^2(\R^3)^3\to\cV^*,
  \qquad
  \chi\cE(\zeta):H^{1/2}(\Sigma)^3\to H^1_{-s}(\Omega)^3
\]
are uniformly bounded for
\[
  0<\delta<\delta_0,\qquad
  |\zeta|\le C\delta^{1/2},\qquad
  \operatorname{Im}\zeta\ge0,
\]
where the boundary values are used when \(\operatorname{Im}\zeta=0\).
\end{proposition}

\begin{proof}
In \(D\), the equation
\((\cL_\delta-\zeta^2)\bfu=f\), multiplied by \(\rho_D\), is
\[
  -\nabla\cdot(C_De(u))-\zeta^2\rho_Du=\rho_Df.
\]
The physical transmission law is
\[
  \delta^{-1}T_D^\nu u=T_0^\nu\bfu|_\Omega,
  \qquad\text{or equivalently}\qquad
  T_D^\nu u
  =
  \delta\bigl(\Lambda(\zeta)\gamma u+
  T_0^\nu\bfw_f(\zeta)\bigr).
\]
Green's formula on \(D\) therefore gives exactly
\eqref{eq:bounded-reduced-equation}.  Conversely, a solution of that
equation, reconstructed by \eqref{eq:reconstruction}, satisfies both
PDEs and both transmission conditions, and uniqueness gives the full
resolvent solution.

The zero-Dirichlet exterior problem has no static kernel.  Its
low-frequency layer-potential representation and local elliptic
regularity give uniform bounds for \(\bfw_f(\zeta)\) and its traction
when \(f\) is cut off in a fixed ball.  The same representation gives
\[
  \norm{\cE(\zeta)g}_{H^1_{-s}(\Omega)}
  \le C_s\norm{g}_{H^{1/2}(\Sigma)}
\]
for \(s>1/2\); the \(r^{-1}\) static tail is in \(L^2_{-s}\) precisely
in this range.  Trace duality then proves the bounds for
\(\cJ_\delta(\zeta)\chi\).
\end{proof}

The following local estimate will be used for the exterior Dirichlet problem.
\begin{lemma}
\label{lem:localized-free-resolvent}
Let \(\chi_0,\chi_1\in C_c^\infty(\R^3)\).  The outgoing free Navier
resolvent, initially defined for \(\operatorname{Im}\zeta>0\), has a germ at
\(\zeta=0\) such that
\begin{equation}
\label{eq:localized-free-germ}
  \chi_0R_0^{\mathrm{out}}(\zeta)\chi_1
  \in\mathcal B\bigl(L^2(\R^3)^3,H^2(\R^3)^3\bigr)
\end{equation}
holomorphically.  The family and its first \(\zeta\)-derivative are uniformly
bounded on every smaller closed disk.  If \(\chi_0=1\) near \(\Sigma\), the
corresponding displacement and conormal traces are holomorphic with values in
\(H^{3/2}(\Sigma)^3\) and \(H^{1/2}(\Sigma)^3\), respectively.
\end{lemma}

\begin{proof}
The Kupradze tensor in Lemma~\ref{lem:kernel-expansion} is the kernel of
\(R_0^{\mathrm{out}}(\zeta)\).  On compact sets, the kernel and its first
\(\zeta\)-derivative are bounded by \(C|x-y|^{-1}\).  Schur's test therefore
gives a holomorphic localized family in \(\mathcal B(L^2,L^2)\).  Choose
\(\chi_2\in C_c^\infty\) equal to one near \(\operatorname{supp}\chi_0\).
The interior estimate for the strongly elliptic Navier operator gives
\[
  \norm{\chi_0u}_{H^2}
  \le C\bigl(
    \norm{\chi_2(\cL_0-\zeta^2)u}_{L^2}
    +\norm{\chi_2u}_{L^2}
  \bigr).
\]
Apply this estimate to
\(u=R_0^{\mathrm{out}}(\zeta)\chi_1f\) and to its
\(\zeta\)-derivative.  This proves the uniform \(L^2\to H^2\) bounds;
the kernel Taylor expansion gives holomorphy in that norm.  The displacement
and conormal statements follow from the \(H^2\) trace theorem.
\end{proof}

The next lemma records the local continuation used later for cutoff resolvent
identities.
\begin{lemma}
\label{lem:cutoff-exterior-holomorphy}
Let \(\chi_0,\chi_1\in C_c^\infty(\R^3)\) and \(s>1/2\).  There is
\(r_1>0\) such that the outgoing cutoff families
\begin{align*}
  \chi_0\cE_{\mathrm{out}}(\zeta)
  &:
  H^{1/2}(\Sigma)^3\to H^1_{-s}(\Omega)^3,\\
  \chi_0R_{\Omega,D}^{\mathrm{out}}(\zeta)\chi_1
  &:
  L_s^2(\R^3)^3\to H^1_{-s}(\Omega)^3,\\
  T_0^\nu R_{\Omega,D}^{\mathrm{out}}(\zeta)\chi_1
  &:
  L_s^2(\R^3)^3\to H^{-1/2}(\Sigma)^3
\end{align*}
extend holomorphically to \(|\zeta|<r_1\).  On each smaller closed disk,
these families and their first \(\zeta\)-derivatives are uniformly bounded in
the displayed operator norms.  Define \(\cT_{\mathrm{out}}(\zeta)u\) to equal
\(u\) in \(D\) and \(\cE_{\mathrm{out}}(\zeta)\gamma u\) in \(\Omega\), and
define \(\cW_{\mathrm{out}}(\zeta)f\) to equal zero in \(D\) and
\(R_{\Omega,D}^{\mathrm{out}}(\zeta)(\rho_0f|_\Omega)\) in \(\Omega\).
Then the cutoff families
\[
  \chi_0\cT_{\mathrm{out}}(\zeta):\cV\to H^1_{-s}(\R^3)^3,
  \qquad
  \chi_0\cW_{\mathrm{out}}(\zeta)\chi_1:
  L_s^2(\R^3)^3\to H^1_{-s}(\R^3)^3
\]
have the same property.  The incoming germs are obtained by
\(\zeta\mapsto-\zeta\).  Only the cutoff families are asserted to take values
in polynomially weighted spaces after continuation into the lower half-plane.
\end{lemma}

\begin{proof}
By Proposition~\ref{prop:dtn-analytic},
\(\mathsf S(\zeta)^{-1}\) is holomorphic near zero, and the outgoing Poisson
germ has the layer-potential representation
\[
  \cE_{\mathrm{out}}(\zeta)
  =\mathbf{SL}(\zeta)\mathsf S(\zeta)^{-1}
  \quad\text{in }\Omega.
\]
Fix a ball \(B\) containing the supports of \(\chi_0,\chi_1\) and
\(\overline D\).  Lemma~\ref{lem:kernel-expansion} and the boundary-kernel
mapping theorem give, in operator norm,
\[
  \chi_0\mathbf{SL}(\zeta)
  \in\mathcal B\bigl(H^{-1/2}(\Sigma)^3,H^1(B\cap\Omega)^3\bigr),
\]
holomorphically for \(|\zeta|<r_1\).  Composition with
\(\mathsf S(\zeta)^{-1}\) proves the first assertion.  Since the output is
supported in \(B\), its local \(H^1\) norm and its \(H^1_{-s}\) norm are
uniformly equivalent.

  Let \(R_0^{\mathrm{out}}(\zeta)\) be the free outgoing Navier resolvent germ
  and choose \(\widetilde\chi_0\) equal to one near
  \(\overline D\cup\supp\chi_0\).  Lemma~\ref{lem:localized-free-resolvent}
  applies to
  \(\widetilde\chi_0R_0^{\mathrm{out}}(\zeta)\chi_1\).  In particular, its
  displacement and weak traction traces on \(\Sigma\) are holomorphic in the
  Sobolev spaces required below.

For \(f\in L_s^2(\R^3)^3\), put
\(F=(\chi_1f)|_\Omega\) and extend \(F\) by zero to \(\R^3\).  The
zero-Dirichlet exterior solution has the exact representation
\[
  R_{\Omega,D}^{\mathrm{out}}(\zeta)F
  =
  \left.
  \left[
    R_0^{\mathrm{out}}(\zeta)F
    -\cE_{\mathrm{out}}(\zeta)
     \gamma_\Sigma R_0^{\mathrm{out}}(\zeta)F
  \right]\right|_\Omega.
\]
The first term is controlled by \eqref{eq:localized-free-germ}; the second is
the composition of its trace with the Poisson family already treated.  This
proves the second displayed family in the lemma.  If \(u\) denotes the right
side, the local conormal estimate
\[
  \norm{T_0^\nu u}_{H^{-1/2}(\Sigma)}
  \le C\left(
    \norm{u}_{H^1(U\cap\Omega)}
    +\norm{\nabla\!\cdot(C_0e(u))}_{L^2(U\cap\Omega)}
  \right)
\]
on a fixed collar \(U\) of \(\Sigma\), together with the exterior equation,
gives the third family and its holomorphic dependence.

The homogeneous reconstruction has the same trace on the two sides of
\(\Sigma\), while the zero-Dirichlet field has zero trace.  Their piecewise
definitions therefore belong to global \(H^1\) after multiplication by
\(\chi_0\).  The asserted properties of \(\cT_{\mathrm{out}}\) and
\(\cW_{\mathrm{out}}\) follow by composition with the bounded trace map and
the restriction \(f\mapsto(\chi_1f)|_\Omega\).  Holomorphy gives uniform
first-derivative bounds on every smaller disk by Cauchy's estimate.  Finally,
the incoming Kupradze germ is obtained by replacing \(\zeta\) with
\(-\zeta\), which proves the last assertion.
\end{proof}

\section{Rigid Modes and the Uniform Grushin--Feshbach Reduction}
\label{sec:reduction}

\subsection{The rigid space and the uniform Korn decomposition}

For a fixed point \(x_j\in D_j\), let
\[
  \mathrm{RM}(D_j)
  =
  \{a_j+b_j\times(x-x_j):a_j,b_j\in\C^3\}.
\]
The nullspace of the interior strain energy is
\[
  \cN
  =
  \bigoplus_{j=1}^N\mathrm{RM}(D_j),
  \qquad
  m=\dim\cN=6N.
\]
Choose a basis \(\{\bfphi_1,\ldots,\bfphi_m\}\) which is orthonormal
for the mass form \(\mathfrak m\), and define
\[
  J:\C^m\longrightarrow\cN,\qquad
  Ja=\sum_{q=1}^ma_q\bfphi_q.
\]
Then
\[
  \mathfrak m(Ja,Jb)=\inner{a}{b}_{\C^m}.
\]
Let \(P=JJ^\sharp\) be the \(\mathfrak m\)-orthogonal projection onto
\(\cN\), where
\[
  (J^\sharp u)_q=\mathfrak m(u,\bfphi_q),
\]
and put
\[
  Q=I-P,\qquad
  \cV_\perp=Q\cV
  =
  \{u\in\cV:\mathfrak m(u,\bfphi_q)=0,\ 1\le q\le m\}.
\]
Because \(\cN\) is finite dimensional and consists of smooth rigid fields,
\(P,Q\in\mathcal B(\cV)\) and
\(\cV=\cN\oplus_{\mathfrak m}\cV_\perp\).

\begin{lemma}
\label{lem:korn-complement}
There is \(c_K>0\), depending only on \(D,C_D,\rho_D\), such that
\[
  \int_D C_De(w):\overline{e(w)}\,\dd x
  \ge
  c_K\norm{w}_{H^1(D)}^2,
  \qquad w\in\cV_\perp.
\]
\end{lemma}

\begin{proof}
If the assertion failed, there would be \(w_n\in\cV_\perp\) with
\(\norm{w_n}_{H^1(D)}=1\) and
\(\norm{e(w_n)}_{L^2(D)}\to0\).  Korn's inequality on each connected
component gives rigid fields \(r_{n,j}\in\mathrm{RM}(D_j)\) such that
\[
  \norm{w_n-r_n}_{H^1(D)}\to0,
  \qquad r_n=(r_{n,1},\ldots,r_{n,N})\in\cN.
\]
Passing to a subsequence, \(r_n\to r\in\cN\) in \(H^1(D)^3\).
For every \(q\), the convergence and the defining moment conditions give
\[
  \mathfrak m(r,\bfphi_q)
  =\lim_{n\to\infty}\mathfrak m(w_n,\bfphi_q)=0.
\]
Thus \(r\in\cN\cap\cV_\perp=\{0\}\), so \(w_n\to0\) in \(H^1(D)^3\).
This contradicts
\(\norm{w_n}_{H^1}=1\).
\end{proof}

\begin{remark}
\label{rem:mass-matrix}
For an arbitrary rigid basis, the mass matrix is
\[
  M=(M_{\ell q})_{\ell,q=1}^m,\qquad
  M_{\ell q}
  =
  \int_D\rho_D\bfphi_q\cdot\overline{\bfphi_\ell}\,\dd x,
  \qquad M=M^*>0.
\]
All formulas below are transformed by the congruence
\(a\mapsto M^{1/2}a\); in particular, \(I_m\) is replaced by \(M\).
We use a mass-orthonormal basis to keep the Grushin formulas readable.
\end{remark}

\subsection{Static capacitance and the resonant scale}

Let
\[
  B=\gamma_\Sigma J:
  \C^m\longrightarrow H^{1/2}(\Sigma)^3.
\]
Here and below \(B^*:H^{-1/2}(\Sigma)^3\to\C^m\) is the adjoint with
respect to boundary duality and the Euclidean inner product.
For \(a\in\C^m\), let
\[
  U_a^0=\cE(0)Ba
\]
be the decaying static exterior field with rigid boundary value \(Ba\).  It is
understood in the finite-energy class
\[
  \mathcal D^{1,2}(\Omega)^3
  :=\{u\in L^6(\Omega)^3:\nabla u\in L^2(\Omega)^{3\times3}\}.
\]
Define the elastic capacitance matrix
\begin{equation}
\label{eq:K-definition}
  \inner{Ka}{b}_{\C^m}
  =
  -\inner{\Lambda_0Ba}{Bb}_{\Sigma}
  =
  \int_\Omega
  C_0e(U_a^0):\overline{e(U_b^0)}\,\dd x.
\end{equation}

\begin{proposition}
\label{prop:K-positive}
The matrix \(K\) is Hermitian positive definite.
\end{proposition}

\begin{proof}
The second identity in \eqref{eq:K-definition} follows from the static
Green formula and the normal convention in
\eqref{eq:static-dtn-sign}.  Betti reciprocity makes \(K\) Hermitian.
If \(\inner{Ka}{a}=0\), strong ellipticity gives
\[
  e(U_a^0)=0\quad\text{in }\Omega.
\]
Thus \(U_a^0\) is a global rigid motion in the connected exterior.
A nonzero affine rigid motion does not belong to \(L^6(\Omega)^3\), so
\(U_a^0=0\).  Hence \(Ba=0\) on every
\(\partial D_j\).  A rigid motion that vanishes on the boundary of a
three-dimensional bounded domain is zero, so \(Ja=0\) in \(D\) and
therefore \(a=0\).
\end{proof}

Let
\[
  0<\lambda_1\le\cdots\le\lambda_m
\]
be the eigenvalues of \(K\), repeated with multiplicity.  The
unperturbed finite-dimensional pencil is
\[
  \cM_{\delta,0}(\omega)=\delta K-\omega^2I_m.
\]
Its positive roots are
\[
  \omega_{j,0}(\delta)=\delta^{1/2}\lambda_j^{1/2}.
\]
For a non-orthonormal basis these are the generalized eigenvalues of
\[
  Ka=\lambda Ma.
\]

\begin{proposition}
\label{prop:static-clusters}
Let \(\lambda\) be an eigenvalue of \(K\) of multiplicity \(r\), and
let \(\mathcal C\) be a positively oriented circle such that
\[
  \operatorname{int}\mathcal C\cap
  \{\pm\sqrt{\lambda_1},\ldots,\pm\sqrt{\lambda_m}\}
  =\{\sqrt\lambda\},
  \qquad
  \mathcal C\cap
  \{\pm\sqrt{\lambda_1},\ldots,\pm\sqrt{\lambda_m}\}
  =\varnothing.
\]
For all sufficiently
small \(\delta\), the full outgoing effective determinant constructed
in Section~\ref{sec:reduction} has exactly \(r\) zeros, counted algebraically, in
\[
  \{\zeta:\delta^{-1/2}\zeta\in\operatorname{int}\mathcal C\}.
\]
\end{proposition}

\begin{proof}
By Theorem~\ref{thm:effective-expansion} below,
\[
  F_\delta(\xi):=\delta^{-1}\cM_\delta(\delta^{1/2}\xi)
  =
  F_0(\xi)+\mathcal O(\delta^{1/2}),
  \qquad F_0(\xi):=K-\xi^2I_m,
\]
uniformly on \(\mathcal C\).  The matrix \(F_0(\xi)\) is invertible there,
and compactness of the circle gives
\[
  \sup_{\xi\in\mathcal C}
  \norm{F_0(\xi)^{-1}(F_\delta(\xi)-F_0(\xi))}<1
\]
for small \(\delta\).  Matrix-valued
Rouch\'e's theorem, equivalently the finite-dimensional Gohberg--Sigal
argument \cite{GohbergSigal1971} applied to \(F_\delta\), preserves the total
characteristic multiplicity inside \(\mathcal C\).  This is the number of
zeros of \(\det F_\delta\), counted with multiplicity.
\end{proof}

\subsection{The uniform complement block}

For \(|\zeta|\le C_{\mathrm{fr}}\delta^{1/2}\), define
\[
  G_\delta(\zeta)
  =
  Q^*\cA_\delta(\zeta)Q:
  \cV_\perp\longrightarrow\cV_\perp^*.
\]
Here \(Q^*:\cV^*\to\cV_\perp^*\) is restriction to
\(\cV_\perp\).

\begin{lemma}
\label{lem:complement-inverse}
There are \(C,\delta_0>0\) such that
\[
  G_\delta(\zeta)^{-1}:
  \cV_\perp^*\longrightarrow\cV_\perp
\]
exists and
\[
  \norm{G_\delta(\zeta)^{-1}}_{\cV_\perp^*\to\cV_\perp}
  \le C
\]
whenever
\[
  0<\delta<\delta_0,\qquad
  |\zeta|\le C_{\mathrm{fr}}\delta^{1/2}.
\]
The same statement holds for the incoming analytic continuation.  In
particular, both complement inverses are holomorphic in their respective full
complex disks, including the lower-half-plane portion used to locate poles.
\end{lemma}

\begin{proof}
For \(w\in\cV_\perp\), Proposition~\ref{prop:dtn-analytic} gives
\[
\begin{split}
  \operatorname{Re}\inner{G_\delta(\zeta)w}{w}
  &=
  \int_D C_De(w):\overline{e(w)}\,\dd x
  -
  \operatorname{Re}(\zeta^2)\mathfrak m(w,w)\\
  &\quad
  -
  \delta\operatorname{Re}
  \inner{\Lambda(\zeta)\gamma w}{\gamma w}_{\Sigma}.
\end{split}
\]
The static boundary term is nonnegative after the minus sign by
\eqref{eq:static-dtn-sign}.  The trace theorem and
\(\Lambda(\zeta)-\Lambda_0=\mathcal O(|\zeta|)\) give
\[
  -\delta\operatorname{Re}
  \inner{(\Lambda(\zeta)-\Lambda_0)\gamma w}{\gamma w}_{\Sigma}
  \ge
  -C\delta|\zeta|\norm{w}_{H^1(D)}^2.
\]
The mass term is bounded below by
\(-C|\zeta|^2\norm{w}_{H^1(D)}^2\).  Lemma~\ref{lem:korn-complement}
therefore yields
\[
  \operatorname{Re}\inner{G_\delta(\zeta)w}{w}
  \ge
  \bigl(c_K-C\delta-C\delta^{3/2}\bigr)
  \norm{w}_{H^1(D)}^2.
\]
For small \(\delta\), the right side is at least
\(\frac12c_K\norm{w}_{H^1}^2\).  Lax--Milgram gives the inverse and
the uniform estimate.  The proof for the incoming branch is identical.
\end{proof}

\subsection{Typed Grushin blocks and corrected resonant states}

For \(F\in\cV^*\), define
\[
  (J^*F)_q=\inner{F}{\bfphi_q},
\]
so \(J^*:\cV^*\to\C^m\).  Relative to
\(\cV=J\C^m\oplus\cV_\perp\), introduce
\[
  A_{00}=J^*\cA_\delta J,\quad
  A_{01}=J^*\cA_\delta Q,\quad
  A_{10}=Q^*\cA_\delta J,\quad
  A_{11}=G_\delta.
\]
The types are
\[
  A_{00}:\C^m\to\C^m,\quad
  A_{10}:\C^m\to\cV_\perp^*,\quad
  A_{01}:\cV_\perp\to\C^m.
\]
This eliminates the domain/codomain ambiguity of a formal full-space
block matrix.

Define the right corrected lifting and the left coordinate map by
\begin{equation}
\label{eq:right-left-lifts}
  S_\delta^r(\zeta)
  =
  J-QG_\delta(\zeta)^{-1}A_{10}(\zeta),
  \qquad
  L_\delta^\ell(\zeta)
  =
  J^*-A_{01}(\zeta)G_\delta(\zeta)^{-1}Q^*.
\end{equation}
Thus
\[
  Q^*\cA_\delta(\zeta)S_\delta^r(\zeta)=0.
\]
The effective matrix is the Feshbach map
\begin{equation}
\label{eq:feshbach-matrix}
  \cM_\delta(\zeta)
  =
  A_{00}(\zeta)
  -
  A_{01}(\zeta)G_\delta(\zeta)^{-1}A_{10}(\zeta).
\end{equation}

\begin{proposition}
\label{prop:exact-variational-inverse}
The operator
\(\cA_\delta(\zeta):\cV\to\cV^*\) is invertible if and only if
\(\cM_\delta(\zeta):\C^m\to\C^m\) is invertible.  In that case,
as an identity in \(\mathcal B(\cV^*,\cV)\),
\begin{equation}
\label{eq:variational-inverse}
  \cA_\delta(\zeta)^{-1}
  =
  QG_\delta(\zeta)^{-1}Q^*
  +
  S_\delta^r(\zeta)
  \cM_\delta(\zeta)^{-1}
  L_\delta^\ell(\zeta).
\end{equation}
Moreover,
\begin{equation}
\label{eq:block-uniform}
  \norm{S_\delta^r(\zeta)-J}_{\C^m\to\cV}
  +
  \norm{L_\delta^\ell(\zeta)-J^*}_{\cV^*\to\C^m}
  \le C\delta
\end{equation}
uniformly for \(|\zeta|\le C_{\mathrm{fr}}\delta^{1/2}\).
\end{proposition}

\begin{proof}
Let \(F\in\cV^*\), and write \(u=Ja+w\) with
\(w\in\cV_\perp\).  Relative to the decomposition
\(\cV=J\C^m\oplus\cV_\perp\), the equation
\(\cA_\delta(\zeta)u=F\) is
\[
  A_{00}a+A_{01}w=J^*F,\qquad
  A_{10}a+G_\delta w=Q^*F.
\]
Since \(G_\delta\) is invertible by Lemma~\ref{lem:complement-inverse},
\[
  w
  =
  G_\delta^{-1}Q^*F-G_\delta^{-1}A_{10}a.
\]
Substitution into the first equation gives
\[
  \cM_\delta a
  =
  \bigl(J^*-A_{01}G_\delta^{-1}Q^*\bigr)F
  =
  L_\delta^\ell F.
\]
If \(\cM_\delta\) is invertible, this determines \(a\) and \(w\) and gives
\eqref{eq:variational-inverse}.

Conversely, if \(\cM_\delta a=0\), then
\(u=S_\delta^r a\) satisfies \(Q^*\cA_\delta u=0\) by construction and
\(J^*\cA_\delta u=\cM_\delta a=0\).  Hence
\(\cA_\delta u=0\).  The identity \(PS_\delta^r=J\) shows that the map
\(S_\delta^r:\C^m\to\cV\) is injective, so a nonzero \(a\) gives a
nonzero kernel vector.  Thus invertibility of \(\cA_\delta\) implies invertibility of
\(\cM_\delta\).

For \(w\in\cV_\perp\), the identities \(e(Ja)=0\) and
\(\mathfrak m(Ja,w)=0\) give
\[
  \inner{A_{10}(\zeta)a}{w}
  =
  -\delta
  \inner{\Lambda(\zeta)Ba}{\gamma w}_{\Sigma}.
\]
The trace theorem and Proposition~\ref{prop:dtn-analytic} yield
\(\|A_{10}\|\le C\delta\).  The same argument, with the duality order
reversed, gives \(\|A_{01}\|\le C\delta\); no adjoint identification for
complex \(\zeta\) is used here.  Combining these bounds with the uniform
estimate for \(G_\delta^{-1}\) proves \eqref{eq:block-uniform}.
\end{proof}

\begin{lemma}
\label{lem:typed-block-estimates}
Uniformly for \(|\zeta|\le C_{\mathrm{fr}}\delta^{1/2}\),
\begin{align}
  &\norm{A_{01}(\zeta)}
   +\norm{A_{10}(\zeta)}
   \le C\delta,
  \label{eq:offdiag-order}\\
  &\norm{\partial_\zeta A_{01}(\zeta)}
   +\norm{\partial_\zeta A_{10}(\zeta)}
   \le C\delta,
  \label{eq:offdiag-derivative-order}\\
  &\norm{\partial_\zeta^2 A_{01}(\zeta)}
   +\norm{\partial_\zeta^2 A_{10}(\zeta)}
   \le C\delta,
  \label{eq:offdiag-second-derivative-order}\\
  &\norm{\partial_\zeta G_\delta(\zeta)}
   \le C(\delta^{1/2}+\delta),
   \qquad
   \norm{\partial_\zeta G_\delta(\zeta)^{-1}}
   \le C(\delta^{1/2}+\delta),
  \label{eq:G-derivative-order}
  \\
  &\norm{\partial_\zeta^2G_\delta(\zeta)}
   +\norm{\partial_\zeta^2G_\delta(\zeta)^{-1}}
   \le C.
  \label{eq:G-second-derivative-order}
\end{align}
All norms are taken in the typed domain--codomain spaces specified above.
\end{lemma}

\begin{proof}
The formula
\[
  \inner{A_{10}(\zeta)a}{w}
  =
  -\delta
  \inner{\Lambda(\zeta)Ba}{\gamma w}_{\Sigma}
\]
and its left-block analogue give \eqref{eq:offdiag-order}.  Differentiating
once and twice, and using Proposition~\ref{prop:dtn-analytic}, gives
\eqref{eq:offdiag-derivative-order} and
\eqref{eq:offdiag-second-derivative-order}.  On the complement,
\[
  \partial_\zeta G_\delta(\zeta)
  =
  Q^*\bigl(-2\zeta\,\mathfrak m
  -\delta\,\gamma^*\partial_\zeta\Lambda(\zeta)\gamma\bigr)Q,
\]
which proves the first estimate in \eqref{eq:G-derivative-order}.
The second follows from
\[
  \partial_\zeta G_\delta^{-1}
  =
  -G_\delta^{-1}(\partial_\zeta G_\delta)G_\delta^{-1}
\]
and Lemma~\ref{lem:complement-inverse}.
Moreover,
\[
  \partial_\zeta^2G_\delta
  =Q^*\bigl(-2\mathfrak m
    -\delta\gamma^*\partial_\zeta^2\Lambda(\zeta)\gamma\bigr)Q,
\]
so \(\partial_\zeta^2G_\delta=O(1)\).  A second differentiation of the
inverse identity gives
\[
  \partial_\zeta^2G_\delta^{-1}
  =2G_\delta^{-1}(\partial_\zeta G_\delta)G_\delta^{-1}
    (\partial_\zeta G_\delta)G_\delta^{-1}
   -G_\delta^{-1}(\partial_\zeta^2G_\delta)G_\delta^{-1},
\]
which proves \eqref{eq:G-second-derivative-order}.
\end{proof}

\subsection{Expansion of the effective matrix}

Set
\begin{equation}
\label{eq:Gamma0-definition}
  \Gamma_0=B^*\Lambda_1B.
\end{equation}
Its explicit pressure/shear Gram representation and its rank are proved
in Section~\ref{sec:radiation}.

\begin{theorem}
\label{thm:effective-expansion}
Uniformly for
\[
  0<\delta<\delta_0,\qquad
  |\zeta|\le C_{\mathrm{fr}}\delta^{1/2},
\]
one has
\begin{equation}
\label{eq:effective-expansion}
  \cM_\delta(\zeta)
  =
  \delta K-\zeta^2I_m-\ii\delta\zeta\Gamma_0
  +E_\delta(\zeta),
\end{equation}
where \(E_\delta\) is holomorphic and
\begin{equation}
\label{eq:effective-error}
  \norm{E_\delta(\zeta)}
  \le
  C(\delta^2+\delta|\zeta|^2),
  \qquad
  \norm{\partial_\zeta E_\delta(\zeta)}
  \le
  C(\delta^2+\delta|\zeta|),
  \qquad
  \norm{\partial_\zeta^2E_\delta(\zeta)}
  \le C\delta.
\end{equation}
For real \(\omega\),
\[
  \cM_\delta^-(\omega)=\cM_\delta^+(\omega)^*.
\]
\end{theorem}

\begin{proof}
Because \(e(Ja)=0\) and the rigid basis is mass orthonormal,
\[
  A_{00}(\zeta)
  =
  -\zeta^2I_m-\delta B^*\Lambda(\zeta)B.
\]
Using \eqref{eq:dtn-basic-expansion} and
\(K=-B^*\Lambda_0B\) gives
\[
  A_{00}(\zeta)
  =
  \delta K-\zeta^2I_m-\ii\delta\zeta\Gamma_0
  -
  \delta\zeta^2B^*\Lambda_2(\zeta)B.
\]
The last term is \(O(\delta|\zeta|^2)\) in matrix norm, uniformly in the
stated disk; its first and second derivatives are
\(O(\delta|\zeta|)\) and \(O(\delta)\), respectively.

For the Schur correction, Lemma~\ref{lem:typed-block-estimates} and
Lemma~\ref{lem:complement-inverse} give
\[
  \norm{A_{01}G_\delta^{-1}A_{10}}\le C\delta^2.
\]
Differentiating this product produces three terms.  The two differentiated
off-diagonal factors are \(O(\delta)\), while
\(\partial_\zeta G_\delta^{-1}=O(\delta^{1/2}+\delta)\).  Consequently
\[
  \norm{\partial_\zeta
  (A_{01}G_\delta^{-1}A_{10})}
  \le C\delta^2.
\]
The second-derivative bounds in
Lemma~\ref{lem:typed-block-estimates} give, by the same product rule,
\[
  \norm{\partial_\zeta^2
  (A_{01}G_\delta^{-1}A_{10})}
  \le C\delta^2.
\]
Combining the two contributions gives \eqref{eq:effective-error}.  For real
\(\omega\), the incoming and outgoing DtN maps are adjoints.  Since the
interior strain and mass forms are Hermitian, the typed Schur complements
satisfy
\[
  \cM_\delta^-(\omega)=\cM_\delta^+(\omega)^*.
\]
\end{proof}

\begin{proposition}
\label{prop:second-real-coefficient}
Let \(\lambda\) be an eigenvalue of \(K\), let \(P_\lambda\) be the
orthogonal projection onto \(\cX_\lambda=\Ker(K-\lambda I)\), and set
\[
  \omega_{\lambda,0}(\delta)=\delta^{1/2}\lambda^{1/2}.
\]
Then the matrix limit
\begin{equation}
\label{eq:Hlambda-definition}
  \mathsf H_\lambda
  :=
  \lim_{\delta\downarrow0}
  \delta^{-2}
  P_\lambda
  E_\delta^+(\omega_{\lambda,0}(\delta))
  P_\lambda
\end{equation}
exists in matrix norm and is Hermitian.  Moreover, uniformly for bounded
\(\sigma\in\R\),
\begin{equation}
\label{eq:second-real-cluster-expansion}
\begin{split}
  P_\lambda\operatorname{Re}_{\mathrm{op}}
  \cM_\delta^+\!
  \left(
    \delta^{1/2}\lambda^{1/2}
    +\delta^{3/2}\sigma
  \right)P_\lambda
  &=
  \delta^2
  P_\lambda
  \left(
    \mathsf H_\lambda
    -2\lambda^{1/2}\sigma I
  \right)
  P_\lambda\\
  &\quad+o(\delta^2),
\end{split}
\end{equation}
where
\(\operatorname{Re}_{\mathrm{op}}A=(A+A^*)/2\).
\end{proposition}

\begin{proof}
The remainder in Theorem~\ref{thm:effective-expansion} is the sum of the
explicit DtN contribution
\[
  -\delta\zeta^2B^*\Lambda_2(\zeta)B
\]
and the Schur correction
\[
  -A_{01}(\zeta)G_\delta(\zeta)^{-1}A_{10}(\zeta).
\]
At \(\zeta=\delta^{1/2}\lambda^{1/2}\), the first term divided by
\(\delta^2\) converges in matrix norm to
\[
  -\lambda B^*\Lambda_2(0)B.
\]
For the second term, write
\(A_{10}(\zeta)=\delta\,\widetilde A_{10}(\zeta)\) and
\(A_{01}(\zeta)=\delta\,\widetilde A_{01}(\zeta)\).
Proposition~\ref{prop:dtn-analytic} implies
\(\widetilde A_{10}(\zeta)\to\widetilde A_{10}(0)\) and
\(\widetilde A_{01}(\zeta)\to\widetilde A_{01}(0)\).
On \(\cV_\perp\),
\[
  G_\delta(\delta^{1/2}\lambda^{1/2})
  \longrightarrow
  G_0,
  \qquad
  \inner{G_0w}{v}
  =
  \int_D C_De(w):\overline{e(v)}\,\dd x,
\]
and Lemma~\ref{lem:korn-complement} gives
\(G_\delta^{-1}\to G_0^{-1}\) in operator norm.  Hence the Schur
correction divided by \(\delta^2\) also has a matrix-norm limit.
At zero frequency the symmetry
\(\Lambda^+(-\omega)=\Lambda^+(\omega)^*\) implies that
\(\Lambda_2(0)\) is self-adjoint.  Moreover,
\(\widetilde A_{01}(0)=\widetilde A_{10}(0)^*\) and
\(G_0=G_0^*\).  Hence the limiting Schur correction is Hermitian as well.
The full matrix limit therefore exists and is Hermitian, proving
\eqref{eq:Hlambda-definition}.

Finally,
\[
  \delta\lambda
  -
  \left(
    \delta^{1/2}\lambda^{1/2}
    +\delta^{3/2}\sigma
  \right)^2
  =
  -2\delta^2\lambda^{1/2}\sigma+O(\delta^3),
\]
while the Hermitian part of
\(-\ii\delta\omega\Gamma_0\) vanishes.  The derivative estimate in
\eqref{eq:effective-error} shows that replacing the center by the
\(O(\delta^{3/2})\) shifted point changes the Hermitian remainder by
\(o(\delta^2)\).  This proves \eqref{eq:second-real-cluster-expansion}.
\end{proof}

\begin{corollary}
\label{cor:rescaled-effective}
For \(\tau\) in a fixed compact subset of \((0,\infty)\),
\[
  \delta^{-1}
  \cM_\delta^\pm(\delta^{1/2}\tau)
  =
  K-\tau^2I_m
  \mp\ii\delta^{1/2}\tau\Gamma_0
  +
  \mathcal O(\delta)
\]
uniformly in \(\tau\).
\end{corollary}

\begin{proof}
Insert \(\zeta=\delta^{1/2}\tau\) into
\eqref{eq:effective-expansion}--\eqref{eq:effective-error} and divide
by \(\delta\).
\end{proof}

\section{The Pole-Subtracted Resonant Limiting Absorption Principle}
\label{sec:lap}

\subsection{Global source and reconstruction maps}

For \(\operatorname{Im}\zeta\ge0\), let
\[
  \cT(\zeta):\cV\longrightarrow H^1_{\mathrm{loc}}(\R^3)^3
\]
be the homogeneous reconstruction
\[
  (\cT(\zeta)u)|_D=u,\qquad
  (\cT(\zeta)u)|_\Omega=\cE(\zeta)\gamma u.
\]
Let \(\cW(\zeta)f\) be zero in \(D\) and equal to
\(\bfw_f(\zeta)\) in \(\Omega\).  Proposition~\ref{prop:bounded-reduction}
can then be written
\begin{equation}
\label{eq:global-from-interior}
  R_\delta(\zeta^2)
  =
  \cW(\zeta)
  +
  \cT(\zeta)\cA_\delta(\zeta)^{-1}\cJ_\delta(\zeta)
\end{equation}
for \(\operatorname{Im}\zeta>0\), with outgoing boundary values on the
positive real axis.  The incoming formula is its adjoint counterpart.

Fix \(s>1/2\) and a cutoff
\(\chi\in C_c^\infty(\R^3)\) which equals one on a neighborhood of
\(\overline D\).  Define
\begin{align}
  \cR_{\delta,\mathrm{reg}}(\zeta)
  &=
  \chi\cW(\zeta)\chi
  +
  \chi\cT(\zeta)
  QG_\delta(\zeta)^{-1}Q^*
  \cJ_\delta(\zeta)\chi,
  \label{eq:regular-block}\\
  \Phi_\delta(\zeta)
  &=
  \cT(\zeta)S_\delta^r(\zeta):
  \C^m\to H^1_{\mathrm{loc}}(\R^3)^3,
  \notag\\
  \Psi_\delta(\zeta)f
  &=
  L_\delta^\ell(\zeta)\cJ_\delta(\zeta)f,
  \qquad f\in L^2_{\mathrm{comp}}(\R^3)^3.
  \notag
\end{align}
Only \(\chi\Phi_\delta(\zeta)\) and \(\Psi_\delta(\zeta)\chi\) are assigned
the weighted operator topologies used below.  In particular, no uncut outgoing
continuation in the lower half-plane is claimed to lie in a polynomially
weighted space.  The cutoff is suppressed from the notation because it is
inserted in every resolvent identity.
The unadorned maps in this subsection use the outgoing germ and are also
written with a superscript \(+\).  Repeating the same construction with
\(\Lambda^-_{\mathrm{an}}\), the incoming exterior Poisson operator, and the
incoming zero-Dirichlet solution defines
\(\cR_{\delta,\mathrm{reg}}^-(\zeta)\),
\(\Phi_\delta^-(\zeta)\), \(\Psi_\delta^-(\zeta)\), and
\(\cM_\delta^-(\zeta)\).  Every local bound below holds for both germs.

\begin{proposition}
\label{prop:global-block-bounds}
For every fixed \(C_{\mathrm{fr}}>0\) there is \(C>0\) such that
\begin{align}
  &\norm{\cR_{\delta,\mathrm{reg}}(\zeta)}
  _{L_s^2\to H_{-s}^1}
  \le C,\label{eq:regular-bound}\\
  &\norm{\chi\Phi_\delta(\zeta)}
  _{\C^m\to H_{-s}^1}
  +
  \norm{\Psi_\delta(\zeta)\chi}
  _{L_s^2\to\C^m}
  \le C
  \label{eq:off-diagonal-bound}
\end{align}
whenever
\[
  0<\delta<\delta_0,\qquad
  |\zeta|\le C_{\mathrm{fr}}\delta^{1/2}.
\]
Here the maps in the lower half-plane are understood through the local
outgoing continuation of the layer potentials.  All three maps are
holomorphic in the open disk and continuous up to its real diameter in the
displayed operator topologies.  In addition,
\begin{equation}
\label{eq:global-block-derivative-bound}
  \norm{\partial_\zeta\cR_{\delta,\mathrm{reg}}(\zeta)}
  +\norm{\partial_\zeta(\chi\Phi_\delta(\zeta))}
  +\norm{\partial_\zeta(\Psi_\delta(\zeta)\chi)}
  \le C
\end{equation}
in the same operator topologies and parameter range.
\end{proposition}

\begin{proof}
The two terms in \eqref{eq:regular-block} contain only the exterior
zero-Dirichlet solution, the homogeneous reconstruction, the source
map, and the complement inverse.  The first three are uniformly
bounded after cutoff by Proposition~\ref{prop:bounded-reduction} and
Lemma~\ref{lem:cutoff-exterior-holomorphy}; the last is uniformly bounded by
Lemma~\ref{lem:complement-inverse}.
This proves \eqref{eq:regular-bound}.

For the right map, Proposition~\ref{prop:exact-variational-inverse}
gives
\[
  S_\delta^r(\zeta)=J+\mathcal O_{\C^m\to\cV}(\delta),
\]
and the reconstruction map is uniformly bounded.  For the left map,
\[
  L_\delta^\ell(\zeta)
  =
  J^*+\mathcal O_{\cV^*\to\C^m}(\delta),
\]
and \(\cJ_\delta(\zeta)\chi\) is uniformly bounded.  This proves
\eqref{eq:off-diagonal-bound}.  Differentiating the explicit block formulas,
using \(\partial_\zeta\cA_\delta=O(\delta^{1/2})\) on the displayed disk and
\(\partial_\zeta G_\delta^{-1}
=-G_\delta^{-1}(\partial_\zeta G_\delta)G_\delta^{-1}\), proves
\eqref{eq:global-block-derivative-bound}; the differentiated exterior factors
are controlled by Lemma~\ref{lem:cutoff-exterior-holomorphy}.  Continuity at
the real axis follows from the corresponding boundary values of the exterior
layer potentials; the complement inverse stays uniformly coercive there.
\end{proof}

\subsection{Exact resolvent identity and the uniform LAP}

\begin{proposition}
\label{prop:exact-resolvent-identity}
For \(\operatorname{Im}\zeta>0\),
\begin{equation}
\label{eq:exact-resolvent-identity}
  \chi R_\delta(\zeta^2)\chi
  =
  \cR_{\delta,\mathrm{reg}}(\zeta)
  +
  \chi\Phi_\delta(\zeta)
  \cM_\delta(\zeta)^{-1}
  \Psi_\delta(\zeta)\chi.
\end{equation}
The identity is algebraic and contains no asymptotic remainder.
\end{proposition}

\begin{proof}
Insert the exact variational inverse
\eqref{eq:variational-inverse} into
\eqref{eq:global-from-interior}.  The term containing
\(QG_\delta^{-1}Q^*\) is exactly
\(\cR_{\delta,\mathrm{reg}}\), and the remaining term factors as
\[
  \cT S_\delta^r
  \cM_\delta^{-1}
  L_\delta^\ell\cJ_\delta
  =
  \Phi_\delta\cM_\delta^{-1}\Psi_\delta.
\]
Multiplication by the two cutoffs gives
\eqref{eq:exact-resolvent-identity}.
\end{proof}

\begin{theorem}
\label{thm:resonant-lap}
Let \(I=[\tau_-,\tau_+]\Subset(0,\infty)\), \(s>1/2\), and
\(\omega=\delta^{1/2}\tau\) with \(\tau\in I\).  For the outgoing and
incoming analytic families, set
\[
  \bigl(\Phi_{\delta,\eps}^\pm(\omega),
  \cM_{\delta,\eps}^\pm(\omega),
  \Psi_{\delta,\eps}^\pm(\omega)\bigr)
  :=
  \bigl(\Phi_\delta^\pm(\zeta_{\omega,\eps}^\pm),
  \cM_\delta^\pm(\zeta_{\omega,\eps}^\pm),
  \Psi_\delta^\pm(\zeta_{\omega,\eps}^\pm)\bigr).
\]
There is \(c_I>0\) such that, uniformly for
\(0<\eps\le c_I\delta^{1/2}\),
\begin{align*}
  &\chi R_\delta\bigl((\omega\pm\ii\eps)^2\bigr)\chi
  -
  \chi\Phi_{\delta,\eps}^\pm(\omega)
  \bigl(\cM_{\delta,\eps}^\pm(\omega)\bigr)^{-1}
  \Psi_{\delta,\eps}^\pm(\omega)\chi\\
  &\hspace{5em}
  =\cR_{\delta,\mathrm{reg}}^\pm
  (\zeta_{\omega,\eps}^\pm),
\end{align*}
and the norm of this expression in
\(\mathcal B(L_s^2,H_{-s}^1)\) is bounded independently of
\((\delta,\tau,\eps)\).  Consequently, the limits
\[
  \cR_{\delta,\mathrm{reg}}^\pm(\omega)
  =
  \lim_{\eps\downarrow0}
  \left[
  \chi R_\delta\bigl((\omega\pm\ii\eps)^2\bigr)\chi
  -
  \chi\Phi_{\delta,\eps}^\pm(\omega)
  \bigl(\cM_{\delta,\eps}^\pm(\omega)\bigr)^{-1}
  \Psi_{\delta,\eps}^\pm(\omega)\chi
  \right]
\]
exist in
\(\mathcal B(L_s^2,H_{-s}^1)\) and satisfy
\[
  \sup_{\substack{0<\delta<\delta_0\\\tau\in I}}
  \norm{\cR_{\delta,\mathrm{reg}}^\pm
  (\delta^{1/2}\tau)}_{L_s^2\to H_{-s}^1}
  <\infty.
\]
For every \(0<\delta<\delta_0\) and \(\tau\in I\),
\(\cM_\delta^\pm(\delta^{1/2}\tau)\) is invertible and
\begin{equation}
\label{eq:boundary-resolvent-decomposition}
  \chi R_\delta^\pm(\omega)\chi
  =
  \cR_{\delta,\mathrm{reg}}^\pm(\omega)
  +
  \chi\Phi_\delta^\pm(\omega)
  \bigl(\cM_\delta^\pm(\omega)\bigr)^{-1}
  \Psi_\delta^\pm(\omega)\chi.
\end{equation}
\end{theorem}

\begin{proof}
For positive absorption, the first displayed identity is exactly
Proposition~\ref{prop:exact-resolvent-identity}; its uniform bound is
Proposition~\ref{prop:global-block-bounds}.  That proposition also gives
boundary values of the regular, right, and left blocks.  The finite-dimensional matrix has
a boundary value by Theorem~\ref{thm:effective-expansion}.  Hence
\eqref{eq:exact-resolvent-identity} passes to the real axis whenever
\(\cM_\delta^\pm(\omega)\) is invertible, and subtracting the
finite-rank term leaves the uniformly bounded regular block.

It remains to exclude a real kernel.  If
\(\cM_\delta^+(\omega)a=0\), Proposition~\ref{prop:exact-variational-inverse}
gives a nonzero interior solution
\(S_\delta^r(\omega)a\) of the homogeneous reduced equation.
Reconstruction produces a nonzero outgoing solution of the homogeneous
full transmission problem.  The Rellich and unique-continuation
argument in Proposition~\ref{prop:fixed-lap} forces this solution to
vanish, a contradiction.  The incoming case follows by adjunction.
\end{proof}

\begin{remark}
The full resolvent boundary value in Theorem~\ref{thm:resonant-lap}
is the usual fixed-\(\delta\) LAP.  The uniform statement concerns the
pole-subtracted quantity.  This distinction is essential: the next
subsections prove that the unsubtracted norm diverges at an unperturbed
subwavelength resonant center.
\end{remark}

\subsection{Uniform nondegeneracy of the resonant maps}

\begin{lemma}
\label{lem:resonant-map-nondegeneracy}
Fix \(C_{\mathrm{fr}}>0\).  For the holomorphic outgoing continuation, there are
\(c,C>0\) such that, uniformly for small \(\delta\) and
\[
  |\zeta|\le C_{\mathrm{fr}}\delta^{1/2},
\]
one has
\begin{equation}
\label{eq:right-lower-singular}
  \norm{\chi\Phi_\delta(\zeta)a}_{H_{-s}^1}
  \ge c|a|,
  \qquad a\in\C^m,
\end{equation}
and \(\Psi_\delta(\zeta)\chi\) has a right inverse
\(\mathcal Q_\delta(\zeta):\C^m\to L_s^2(\R^3)^3\) with
\[
  \Psi_\delta(\zeta)\chi\mathcal Q_\delta(\zeta)=I_m,
  \qquad
  \norm{\mathcal Q_\delta(\zeta)}\le C.
\]
The same assertions hold for the incoming continuation.  In particular they
hold for the boundary values \(\zeta=\omega\pm\ii0\),
\(\omega=\delta^{1/2}\tau\), \(\tau\in I\).
\end{lemma}

\begin{proof}
The proof is uniform in the whole complex disk.  Inside \(D\), reconstruction
is the identity.  Since \(PS_\delta^r=J\), the boundedness of
\(P:H^1(D)^3\to\cN\) gives the exact estimate
\[
  \norm{(\chi\Phi_\delta(\zeta)a)|_D}_{H^1(D)}
  =\norm{S_\delta^r(\zeta)a}_{H^1(D)}
  \ge \norm{P}_{H^1\to H^1}^{-1}\norm{Ja}_{H^1(D)}.
\]
The finite-dimensional map \(J:\C^m\to H^1(D)^3\) has a positive smallest
singular value, so this proves \eqref{eq:right-lower-singular} uniformly in the
whole disk.

For \(a\in\C^m\), take \(f_a=Ja\) in \(D\) and extend it by zero to
\(\R^3\).  Since \(\chi=1\) near \(D\), one has \(\chi f_a=f_a\).  The
exterior source term in \eqref{eq:source-functional} vanishes for every
\(\zeta\).  Mass orthogonality also gives
\(Q^*\cJ_\delta(\zeta)f_a=0\), while
\[
  J^*\cJ_\delta(\zeta)f_a=a.
\]
The definition of \(L_\delta^\ell\) therefore yields the exact identity
\[
  \Psi_\delta(\zeta)\chi f_a=a.
\]
Thus \(\mathcal Q_\delta a=f_a\) is a right inverse, independent of
\(\zeta\) and uniformly bounded.  The incoming statement follows from the
corresponding incoming analytic family, and the real-axis assertions follow by
boundary values.
\end{proof}

\begin{lemma}
\label{lem:finite-rank-transfer}
Let \(X,Y\) be Hilbert spaces,
\(A:\C^m\to X\), and \(B:Y\to\C^m\).  Suppose
\[
  \norm{Aa}_X\ge\alpha|a|,
\]
and suppose \(B\) has a right inverse of norm at most \(\beta^{-1}\).
Then, for every matrix \(S\),
\[
  \norm{ASB}_{Y\to X}
  \ge
  \alpha\beta\norm{S}_{\C^m\to\C^m}.
\]
\end{lemma}

\begin{proof}
Choose unit vectors \(x_n\) with
\(|Sx_n|\to\norm{S}\), and let \(y_n\in Y\) satisfy
\(By_n=x_n\) and \(\norm{y_n}\le\beta^{-1}\).  Then
\[
  \norm{ASB}
  \ge
  \frac{\norm{ASx_n}}{\norm{y_n}}
  \ge
  \alpha\beta|Sx_n|.
\]
Letting \(n\to\infty\) proves the result.
\end{proof}

\begin{theorem}
\label{thm:sharp-resolvent-law}
Under the hypotheses of Theorem~\ref{thm:resonant-lap},
\begin{equation}
\label{eq:sharp-two-sided}
  c\norm{\bigl(\cM_\delta^\pm(\omega)\bigr)^{-1}}-C
  \le
  \norm{\chi R_\delta^\pm(\omega)\chi}_{L_s^2\to H_{-s}^1}
  \le
  C\left(
  1+\norm{\bigl(\cM_\delta^\pm(\omega)\bigr)^{-1}}
  \right).
\end{equation}
Thus all loss of uniformity in the cutoff resolvent is finite
dimensional.
\end{theorem}

\begin{proof}
The upper estimate follows from
\eqref{eq:boundary-resolvent-decomposition} and
Proposition~\ref{prop:global-block-bounds}.  For the lower estimate,
apply Lemma~\ref{lem:finite-rank-transfer} with
\[
  A=\chi\Phi_\delta^\pm(\omega),\quad
  B=\Psi_\delta^\pm(\omega)\chi,\quad
  S=\bigl(\cM_\delta^\pm(\omega)\bigr)^{-1}.
\]
Lemma~\ref{lem:resonant-map-nondegeneracy} gives uniform constants
\(\alpha,\beta>0\).  Hence the norm of the finite-rank term is bounded
below by \(c\norm{\cM_\delta^\pm(\omega)^{-1}}\).  The regular term has
norm at most \(C\), and the reverse triangle inequality gives the
lower bound.
\end{proof}

\subsection{Stability away from, and blow-up at, resonance}

Set
\[
  \cM_{\delta,\mathrm{lead}}^\pm(\omega)
  =
  \delta K-\omega^2I_m
  \mp\ii\delta\omega\Gamma_0.
\]

\begin{proposition}
\label{prop:schur-stability}
If
\[
  C(\delta^2+\delta\omega^2)
  \norm{\bigl(\cM_{\delta,\mathrm{lead}}^\pm(\omega)\bigr)^{-1}}
  \le\frac12,
\]
then
\[
  \frac12\norm{\bigl(\cM_{\delta,\mathrm{lead}}^\pm(\omega)\bigr)^{-1}}
  \le
  \norm{\bigl(\cM_\delta^\pm(\omega)\bigr)^{-1}}
  \le
  2\norm{\bigl(\cM_{\delta,\mathrm{lead}}^\pm(\omega)\bigr)^{-1}},
\]
and
\[
  \norm{
  \bigl(\cM_\delta^\pm\bigr)^{-1}
  -
  \bigl(\cM_{\delta,\mathrm{lead}}^\pm\bigr)^{-1}}
  \le
  C(\delta^2+\delta\omega^2)
  \norm{\bigl(\cM_{\delta,\mathrm{lead}}^\pm\bigr)^{-1}}^2.
\]
\end{proposition}

\begin{proof}
Factor
\[
  \cM_\delta^\pm
  =
  \cM_{\delta,\mathrm{lead}}^\pm
  \left[
  I+
  \bigl(\cM_{\delta,\mathrm{lead}}^\pm\bigr)^{-1}
  E_\delta^\pm
  \right]
\]
and use the Neumann series.  The difference estimate follows from the
resolvent identity.
\end{proof}

\begin{corollary}
\label{cor:away-resonance}
Let \(I=[\tau_-,\tau_+]\Subset(0,\infty)\).  If \(\tau\in I\) and
\[
  \dist(\tau^2,\sigma(K))\ge c_0>0,
  \qquad \omega=\delta^{1/2}\tau,
\]
then
\[
  c\delta^{-1}
  \le
  \norm{\chi R_\delta^\pm(\omega)\chi}
  \le C\delta^{-1}.
\]
\end{corollary}

\begin{proof}
The rescaled matrix in Corollary~\ref{cor:rescaled-effective} is uniformly
invertible, so
\[
  c\delta^{-1}
  \le\norm{\cM_\delta^\pm(\omega)^{-1}}
  \le C\delta^{-1}.
\]
Apply both sides of Theorem~\ref{thm:sharp-resolvent-law} and absorb
the additive constant for small \(\delta\).
\end{proof}

\begin{remark}
Here ``away from resonance'' refers to separation from the static spectrum in
the rescaled variable \(\tau\).  The factor \(\delta^{-1}\) is the baseline
threshold amplification of the rigid channel for the source normalization in
\eqref{eq:source-functional}; it is not a pole enhancement.  Relative to this
baseline, a simple bright center gains an additional factor
\(\delta^{-1/2}\), while a second-order-bright dark center gains an additional
factor \(\delta^{-3/2}\).
\end{remark}

\begin{theorem}
\label{thm:noncommuting-limits}
Let \(\lambda_j\) be a simple eigenvalue of \(K\), let
\(|a_j|=1\), and assume
\[
  \gamma_j:=\inner{\Gamma_0a_j}{a_j}>0.
\]
At the unperturbed real frequency
\(\omega_{j,0}=\delta^{1/2}\lambda_j^{1/2}\),
\begin{equation}
\label{eq:bright-real-blowup}
  \norm{\chi R_\delta^+(\omega_{j,0})\chi}
  \asymp\delta^{-3/2}.
\end{equation}
More precisely, with \(P_j=a_j\otimes a_j^*\),
\begin{equation}
\label{eq:effective-scaled-limit}
  \delta^{3/2}
  \bigl(\cM_\delta^+(\omega_{j,0})\bigr)^{-1}
  \longrightarrow
  \frac{\ii}{\lambda_j^{1/2}\gamma_j}P_j
\end{equation}
in matrix norm.  Consequently convergence in the absorption parameter
is not uniform in \(\delta\) along this resonant path.  More explicitly,
fix \(\vartheta>0\) and set
\[
  \zeta_\delta=\omega_{j,0}+\ii\vartheta\delta.
\]
Then
\begin{equation}
\label{eq:absorbed-effective-scaled-limit}
  \delta^{3/2}\cM_\delta(\zeta_\delta)^{-1}
  \longrightarrow
  \frac{\ii}{\lambda_j^{1/2}(\gamma_j+2\vartheta)}P_j,
\end{equation}
and
\begin{equation}
\label{eq:absorption-nonuniform}
  \norm{
  \chi R_\delta(\zeta_\delta^2)\chi
  -\chi R_\delta^+(\omega_{j,0})\chi}
  \asymp\delta^{-3/2}.
\end{equation}
Thus \(\operatorname{Im}\zeta_\delta=\vartheta\delta\downarrow0\), but the
absorbed resolvent does not approach its boundary value uniformly in
\(\delta\).  The pole-subtracted family remains uniformly bounded.  No claim
about unscaled iterated limits at a fixed positive frequency is intended.
\end{theorem}

\begin{proof}
Let \(P_j=a_j\otimes a_j^*\) and \(Q_j=I-P_j\).
On \(Q_j\C^m\),
\[
  Q_j(\delta K-\omega_{j,0}^2I)Q_j
  =
  \delta Q_j(K-\lambda_jI)Q_j
\]
has inverse of norm \(O(\delta^{-1})\).  The off-diagonal blocks of
\(\cM_\delta^+(\omega_{j,0})\) are
\(\mathcal O(\delta^{3/2})\), and the error
\eqref{eq:effective-error} is \(\mathcal O(\delta^2)\).
Taking the Schur complement with respect to \(Q_j\) gives the scalar
block
\[
  -\ii\delta^{3/2}\lambda_j^{1/2}\gamma_j
  +
  \mathcal O(\delta^2).
\]
Its inverse is
\[
  \frac{\ii\delta^{-3/2}}
  {\lambda_j^{1/2}\gamma_j}
  +\mathcal O(\delta^{-1}),
\]
while every other block of the inverse is \(O(\delta^{-1})\).
This proves \eqref{eq:effective-scaled-limit}.
The two-sided estimate \eqref{eq:sharp-two-sided} then gives
\eqref{eq:bright-real-blowup}.

At \(\zeta_\delta=\omega_{j,0}+\ii\vartheta\delta\), the same Schur
complement calculation gives
\[
  P_j\cM_\delta(\zeta_\delta)P_j
  =-\ii\delta^{3/2}\lambda_j^{1/2}
  (\gamma_j+2\vartheta)P_j+\mathcal O(\delta^2),
\]
while all complementary inverse blocks remain \(O(\delta^{-1})\).  This
proves \eqref{eq:absorbed-effective-scaled-limit}.  Hence
\begin{equation}
\label{eq:absorbed-boundary-matrix-difference}
\begin{split}
  &\delta^{3/2}
  \left[
    \cM_\delta(\zeta_\delta)^{-1}
    -\cM_\delta^+(\omega_{j,0})^{-1}
  \right]\\
  &\qquad\longrightarrow
  \frac{\ii}{\lambda_j^{1/2}}
  \left(\frac1{\gamma_j+2\vartheta}-\frac1{\gamma_j}\right)P_j\ne0.
\end{split}
\end{equation}
Put
\[
  X_\delta(\zeta)=\chi\Phi_\delta(\zeta),
  \qquad
  Y_\delta(\zeta)=\Psi_\delta(\zeta)\chi,
  \qquad
  D_\delta=\cM_\delta(\zeta_\delta)^{-1}
  -\cM_\delta^+(\omega_{j,0})^{-1}.
\]
The two exact resolvent decompositions give
\begin{equation}
\label{eq:absorbed-resolvent-difference-expansion}
\begin{split}
  &\chi R_\delta(\zeta_\delta^2)\chi
  -\chi R_\delta^+(\omega_{j,0})\chi\\
  &={}
  \cR_{\delta,\mathrm{reg}}(\zeta_\delta)
  -\cR_{\delta,\mathrm{reg}}^+(\omega_{j,0})
  +X_\delta(\omega_{j,0})D_\delta Y_\delta(\omega_{j,0})
  +\operatorname{Err}_\delta,
\end{split}
\end{equation}
where, writing \(M_\zeta=\cM_\delta(\zeta_\delta)^{-1}\),
\[
\begin{split}
  \operatorname{Err}_\delta
  ={}&
  \bigl[X_\delta(\zeta_\delta)-X_\delta(\omega_{j,0})\bigr]
  M_\zeta Y_\delta(\zeta_\delta)\\
  &+X_\delta(\omega_{j,0})M_\zeta
  \bigl[Y_\delta(\zeta_\delta)-Y_\delta(\omega_{j,0})\bigr].
\end{split}
\]
By \eqref{eq:global-block-derivative-bound}, both bracketed differences are
\(O(\delta)\), while \(\|M_\zeta\|=O(\delta^{-3/2})\).  Hence
\(\|\operatorname{Err}_\delta\|=O(\delta^{-1/2})\); the regular-block difference is
bounded as well.  On the other hand,
\eqref{eq:absorbed-boundary-matrix-difference} gives
\(\|D_\delta\|\ge c_\vartheta\delta^{-3/2}\).  Apply
Lemma~\ref{lem:finite-rank-transfer} to the middle term in
\eqref{eq:absorbed-resolvent-difference-expansion}, using the lower bound and
exact right inverse from Lemma~\ref{lem:resonant-map-nondegeneracy}.  The
reverse triangle inequality proves the lower bound in
\eqref{eq:absorption-nonuniform}; the upper bound follows from the uniform
block bounds.  Uniform boundedness of the pole-subtracted part is
Theorem~\ref{thm:resonant-lap}.
\end{proof}

\section{Elastic Radiation, Resonance Poles, and Line Shapes}
\label{sec:radiation}

\subsection{Pressure and shear far-field maps}

Let \(g\in H^{1/2}(\Sigma)^3\) and
\(U_g^+(\omega)=\cE^+(\omega)g\).  As \(r=|x|\to\infty\),
\begin{equation}
\label{eq:far-field-expansion}
  U_g^+(x,\omega)
  =
  \frac{e^{\ii k_pr}}{r}
  F_p(\omega)g(\widehat x)
  +
  \frac{e^{\ii k_sr}}{r}
  F_s(\omega)g(\widehat x)
  +
  \mathcal O(r^{-2}),
\end{equation}
where
\[
  k_p=\omega\sqrt{\frac{\rho_0}{\lambda_0+2\mu_0}},
  \qquad
  k_s=\omega\sqrt{\frac{\rho_0}{\mu_0}},
\]
and
\[
  F_p(\omega)g(\widehat x)\parallel\widehat x,
  \qquad
  F_s(\omega)g(\widehat x)\cdot\widehat x=0.
\]
The remainder in \eqref{eq:far-field-expansion}, together with one
radial derivative, is uniform for \(g\) in bounded subsets of
\(H^{1/2}(\Sigma)^3\); see also \cite{AlvesKress2002}.

Set
\[
  c_p=(\lambda_0+2\mu_0)\frac{k_p}{\omega}
  =\sqrt{\rho_0(\lambda_0+2\mu_0)},
  \qquad
  c_s=\mu_0\frac{k_s}{\omega}
  =\sqrt{\rho_0\mu_0},
\]
and define
\[
  \cZ
  =
  L^2(\mathbb S^2)^3\oplus L^2(\mathbb S^2)^3,
\]
\begin{equation}
\label{eq:weighted-far-field-map}
  \cF(\omega)g
  =
  \bigl(\sqrt{c_p}F_p(\omega)g,
  \sqrt{c_s}F_s(\omega)g\bigr)\in\cZ.
\end{equation}
The operator imaginary part, rather than
\(\operatorname{Im}\inner{Ag}{h}\) for unrelated \(g,h\), is the
correct polarized object.

\begin{theorem}
\label{thm:optical-identity}
For every real \(\omega>0\),
\begin{equation}
\label{eq:operator-optical-identity}
  \operatorname{Im}_{\mathrm{op}}\Lambda^+(\omega)
  =
  \omega\cF(\omega)^*\cF(\omega)
\end{equation}
as an operator from \(H^{1/2}(\Sigma)^3\) to
\(H^{-1/2}(\Sigma)^3\).  Equivalently, for all \(g,h\),
\begin{align}
  &\frac1{2\ii}
  \left(
  \inner{\Lambda^+(\omega)g}{h}_{\Sigma}
  -
  \overline{\inner{\Lambda^+(\omega)h}{g}_{\Sigma}}
  \right)
  \notag\\
  &\qquad
  =
  \omega c_p
  \inner{F_p(\omega)g}{F_p(\omega)h}_{L^2(\mathbb S^2)}
  +
  \omega c_s
  \inner{F_s(\omega)g}{F_s(\omega)h}_{L^2(\mathbb S^2)}.
  \label{eq:polarized-optical}
\end{align}
\end{theorem}

\begin{proof}
Let \(U=\cE^+(\omega)g\) and \(V=\cE^+(\omega)h\).
Apply Betti's formula in
\(\Omega_R=B_R\setminus\overline D\).  The outward normal to
\(\Omega_R\) is \(-\bfnu\) on \(\Sigma\) and \(\widehat x\) on
\(\partial B_R\).  Therefore
\begin{align}
  &\inner{\Lambda^+(\omega)g}{h}_{\Sigma}
  -
  \overline{\inner{\Lambda^+(\omega)h}{g}_{\Sigma}}
  \notag\\
  &\qquad
  =
  \lim_{R\to\infty}
  \int_{\partial B_R}
  \left(
  T_rU\cdot\overline V
  -
  U\cdot\overline{T_rV}
  \right)\,\dd S.
  \label{eq:betti-large-sphere}
\end{align}

Write \(U=U_p+U_s\) and \(V=V_p+V_s\).  Direct differentiation of
\eqref{eq:far-field-expansion}, or the Kupradze condition, gives
\[
  T_rU_p
  =
  \ii k_p(\lambda_0+2\mu_0)U_p+\mathcal O(r^{-2}),
  \qquad
  T_rU_s
  =
  \ii k_s\mu_0U_s+\mathcal O(r^{-2}).
\]
The leading pressure field is radial and the leading shear field is
tangential.  Hence the mixed terms in
\eqref{eq:betti-large-sphere} are \(o(1)\) after integration.
The pressure contribution converges to
\[
  2\ii k_p(\lambda_0+2\mu_0)
  \inner{F_p(\omega)g}{F_p(\omega)h}_{L^2(\mathbb S^2)}
  =
  2\ii\omega c_p
  \inner{F_p(\omega)g}{F_p(\omega)h},
\]
and the shear contribution converges to
\[
  2\ii\omega c_s
  \inner{F_s(\omega)g}{F_s(\omega)h}.
\]
Division by \(2\ii\) proves \eqref{eq:polarized-optical} and hence
\eqref{eq:operator-optical-identity}.
\end{proof}

\subsection{Low-frequency force map and explicit constants}

Let \(\mathsf S_0\) be the static elastic single-layer boundary
operator from the proof of Proposition~\ref{prop:dtn-analytic}.  Define
the total-force map
\[
  \mathfrak q:
  H^{1/2}(\Sigma)^3\longrightarrow\C^3,
  \qquad
  \mathfrak q g
  =
  \int_\Sigma\mathsf S_0^{-1}g\,\dd S.
\]
Here the integral of an \(H^{-1/2}(\Sigma)^3\) density is understood
componentwise by duality against the constant function on \(\Sigma\).
The terminology follows from the traction jump relation: for static
Dirichlet data, \(\mathsf S_0^{-1}g\) is the jump of traction and its
integral is the resultant force, up to the fixed normal convention.

\begin{proposition}
\label{prop:far-field-low-frequency}
The map
\[
  \cF(\omega):
  H^{1/2}(\Sigma)^3\to\cZ
\]
is analytic at zero and
\begin{equation}
\label{eq:F-Taylor}
  \cF(\omega)
  =
  \cF_0+\omega\cF_1+\mathcal O(\omega^2).
\end{equation}
The first two coefficients satisfy the operator orthogonality relation
\begin{equation}
\label{eq:F0F1-orthogonality}
  \cF_0^*\cF_1+\cF_1^*\cF_0=0.
\end{equation}
Its leading pressure and shear components are
\begin{align}
  F_{p,0}g(\widehat x)
  &=
  \frac1{4\pi(\lambda_0+2\mu_0)}
  \widehat x\widehat x^{\mathsf T}\mathfrak qg,
  \label{eq:Fp0}\\
  F_{s,0}g(\widehat x)
  &=
  \frac1{4\pi\mu_0}
  (I-\widehat x\widehat x^{\mathsf T})\mathfrak qg.
  \label{eq:Fs0}
\end{align}
Consequently,
\begin{equation}
\label{eq:F0-force-gram}
  \cF_0^*\cF_0
  =
  \gamma_{\mathrm{el}}\mathfrak q^*\mathfrak q,
  \qquad
  \gamma_{\mathrm{el}}
  =
  \frac{\sqrt{\rho_0}}{12\pi}
  \left(
  \frac1{(\lambda_0+2\mu_0)^{3/2}}
  +
  \frac2{\mu_0^{3/2}}
  \right).
\end{equation}
\end{proposition}

\begin{proof}
The exterior solution is
\[
  U_g^+(\omega)
  =
  \mathbf{SL}(\omega)\mu_g(\omega),
  \qquad
  \mu_g(\omega)
  =
  \mathsf S(\omega)^{-1}g.
\]
Both factors are analytic at zero by Proposition~\ref{prop:dtn-analytic}.
The large-\(r\) expansion of the Kupradze
tensor gives
\begin{align*}
  F_p(\omega)g(\widehat x)
  &=
  \frac{\widehat x\widehat x^{\mathsf T}}
  {4\pi(\lambda_0+2\mu_0)}
  \int_\Sigma
  e^{-\ii k_p\widehat x\cdot y}
  \mu_g(\omega,y)\,\dd S_y,\\
  F_s(\omega)g(\widehat x)
  &=
  \frac{I-\widehat x\widehat x^{\mathsf T}}
  {4\pi\mu_0}
  \int_\Sigma
  e^{-\ii k_s\widehat x\cdot y}
  \mu_g(\omega,y)\,\dd S_y.
\end{align*}
Taylor expansion of the exponential and of \(\mu_g(\omega)\) proves
\eqref{eq:F-Taylor}--\eqref{eq:Fs0}.

To prove \eqref{eq:F0F1-orthogonality}, use
\(\Lambda^+(-\omega)=\Lambda^+(\omega)^*\).  The coefficient of
\(\omega^2\) in the Taylor expansion of \(\Lambda^+\) is therefore
self-adjoint.  Expanding the exact identity
\[
  \operatorname{Im}_{\mathrm{op}}\Lambda^+(\omega)
  =\omega\cF(\omega)^*\cF(\omega)
\]
shows that its left side has no term of order \(\omega^2\).  Comparing
that coefficient on the right gives
\eqref{eq:F0F1-orthogonality}.

For \(q\in\C^3\),
\[
  \int_{\mathbb S^2}
  |\widehat x\widehat x^{\mathsf T}q|^2\,\dd\widehat x
  =
  \frac{4\pi}{3}|q|^2,
  \qquad
  \int_{\mathbb S^2}
  |(I-\widehat x\widehat x^{\mathsf T})q|^2\,\dd\widehat x
  =
  \frac{8\pi}{3}|q|^2.
\]
Multiplying these identities by \(c_p\) and \(c_s\), respectively,
and using \eqref{eq:weighted-far-field-map} gives
\eqref{eq:F0-force-gram}.
\end{proof}

\begin{corollary}
\label{cor:Lambda1-explicit}
The first DtN coefficient is
\[
  \Lambda_1
  =
  \cF_0^*\cF_0
  =
  \gamma_{\mathrm{el}}\mathfrak q^*\mathfrak q.
\]
In particular, \(\rank\Lambda_1\le3\).
\end{corollary}

\begin{proof}
Divide \eqref{eq:operator-optical-identity} by \(\omega\) and let
\(\omega\downarrow0\).  The left side converges to \(\Lambda_1\) by
\eqref{eq:dtn-basic-expansion}, and the right side converges to
\(\cF_0^*\cF_0\).  Proposition~\ref{prop:far-field-low-frequency} gives the
last identity.
\end{proof}

\subsection{The rank-three theorem}

\begin{theorem}
\label{thm:rank-three}
The leading radiation matrix is
\begin{equation}
\label{eq:Gamma0-force}
  \Gamma_0
  =
  \gamma_{\mathrm{el}}
  B^*\mathfrak q^*\mathfrak qB.
\end{equation}
Moreover,
\begin{equation}
\label{eq:rank-three}
  \rank\Gamma_0=3,\qquad
  \Ker\Gamma_0=\Ker(\mathfrak qB),\qquad
  \dim\Ker\Gamma_0=6N-3.
\end{equation}
\end{theorem}

\begin{proof}
Equation \eqref{eq:Gamma0-force} follows from
\eqref{eq:Gamma0-definition} and Corollary~\ref{cor:Lambda1-explicit}.  It gives
\[
  \inner{\Gamma_0a}{a}
  =
  \gamma_{\mathrm{el}}|\mathfrak qBa|^2,
\]
and therefore
\(\Ker\Gamma_0=\Ker(\mathfrak qB)\) and
\(\rank\Gamma_0\le3\).

It remains to prove surjectivity of \(\mathfrak qB\).
Let \(\cN_{\mathrm{tr}}\subset\cN\) be the three-dimensional space of
global translations, so that, for \(t\in\C^3\), there is
\(a(t)\in\C^m\) with
\[
  Ja(t)=t\quad\text{on every }D_j.
\]
Let \(\mu_t=\mathsf S_0^{-1}Ba(t)\).  With the convention used in
Proposition~\ref{prop:dtn-analytic}, the traction jump is
\[
  T_0^\nu\mathbf{SL}(0)\mu\big|_+
  -T_0^\nu\mathbf{SL}(0)\mu\big|_-=-\mu,
\]
where \(+\) and \(-\) denote the exterior and interior traces, respectively.
The static single-layer potential generated by \(\mu_t\) equals the
constant field \(t\) in every \(D_j\), whose interior traction is zero.
Hence \(\mu_t=-\Lambda_0Ba(t)\), and Green's identity gives
\[
  -\inner{\Lambda_0Ba(t)}{Ba(t)}_{\Sigma}
  =
  \inner{\mathfrak qBa(t)}{t}_{\C^3}
  >0
  \qquad (t\ne0),
\]
with no undetermined sign.  Hence the linear map
\[
  t\longmapsto\mathfrak qBa(t)
\]
has trivial kernel.  It is an endomorphism of \(\C^3\), so it is
invertible.  Thus \(\rank(\mathfrak qB)=3\), which proves
\eqref{eq:rank-three}.
\end{proof}

\begin{corollary}
\label{cor:polarization-ratio}
If \(\mathfrak qg\ne0\), the leading radiated pressure and shear powers
satisfy
\[
  \frac{
  c_p\norm{F_{p,0}g}_{L^2(\mathbb S^2)}^2}
  {
  c_s\norm{F_{s,0}g}_{L^2(\mathbb S^2)}^2}
  =
  \frac12
  \left(\frac{\mu_0}{\lambda_0+2\mu_0}\right)^{3/2}.
\]
Thus a nonzero leading force channel emits both polarizations in a
geometry-independent ratio; pressure-only and shear-only behavior can
first occur at higher multipole order.
\end{corollary}

\begin{proof}
Insert \eqref{eq:Fp0}--\eqref{eq:Fs0} and the two spherical integral
identities from Proposition~\ref{prop:far-field-low-frequency}.
\end{proof}

\subsection{Exact effective damping and the second radiation form}

Define the corrected boundary trace
\[
  B_\delta(\omega)
  =
  \gamma S_\delta^r(\omega):
  \C^m\to H^{1/2}(\Sigma)^3.
\]
By \eqref{eq:block-uniform},
\begin{equation}
\label{eq:Bdelta-expansion}
  B_\delta(\omega)=B+\mathcal O(\delta)
\end{equation}
uniformly for \(\omega=\delta^{1/2}\tau\), \(\tau\in I\).

\begin{theorem}
\label{thm:effective-optical}
For every real \(\omega>0\) and \(a,b\in\C^m\),
\begin{align}
  &-\frac1{2\ii}
  \left(
  \inner{\cM_\delta^+(\omega)a}{b}
  -
  \overline{\inner{\cM_\delta^+(\omega)b}{a}}
  \right)
  \notag\\
  &\qquad
  =
  \delta\omega
  \inner{
  \cF(\omega)B_\delta(\omega)a
  }{
  \cF(\omega)B_\delta(\omega)b
  }_{\cZ}.
  \label{eq:effective-optical-polarized}
\end{align}
Equivalently,
\begin{equation}
\label{eq:effective-optical-operator}
  -\operatorname{Im}_{\mathrm{op}}\cM_\delta^+(\omega)
  =
  \delta\omega
  B_\delta(\omega)^*
  \cF(\omega)^*\cF(\omega)
  B_\delta(\omega)
  \ge0.
\end{equation}
\end{theorem}

\begin{proof}
Let \(u_a=S_\delta^r(\omega)a\).  Since
\(Q^*\cA_\delta^+(\omega)u_a=0\) and \(u_a-Ja\in\cV_\perp\),
\[
  \inner{\cM_\delta^+(\omega)a}{a}
  =
  \inner{\cA_\delta^+(\omega)u_a}{Ja}
  =
  \inner{\cA_\delta^+(\omega)u_a}{u_a}.
\]
The interior strain and mass terms in
\eqref{eq:interior-form} are real on the diagonal.  Hence
\[
  -\operatorname{Im}
  \inner{\cM_\delta^+(\omega)a}{a}
  =
  \delta
  \inner{
  \operatorname{Im}_{\mathrm{op}}\Lambda^+(\omega)
  B_\delta(\omega)a
  }{
  B_\delta(\omega)a
  }_{\Sigma}.
\]
Theorem~\ref{thm:optical-identity} gives the diagonal version of
\eqref{eq:effective-optical-polarized}.  Polarization proves the
identity for \(a,b\), and hence the operator formula.
\end{proof}

Let
\[
  \cD_0=\Ker\Gamma_0=\Ker(\mathfrak qB),
  \qquad
  P_{\cD_0}:\C^m\to\cD_0
\]
be the orthogonal projection.  Define the second radiation form on
\(\cD_0\) by
\begin{equation}
\label{eq:Gamma1-definition}
  \Gamma_1
  =
  P_{\cD_0}
  B^*\cF_1^*\cF_1B
  P_{\cD_0}.
\end{equation}
It is Hermitian and nonnegative.
Although \(\cF_1\) depends on the choice of spatial origin, the compressed
form \(\Gamma_1\) does not.  A change of origin modifies \(\cF_1\) by an
operator factoring through \(\cF_0\); this additional term vanishes on
\(\cD_0=\Ker(\cF_0B)\).

\begin{theorem}
\label{thm:second-radiation-law}
Let \(\omega=\delta^{1/2}\tau\) with
\(\tau\in I=[\tau_-,\tau_+]\Subset(0,\infty)\).  Uniformly for \(a\in\cD_0\),
\begin{equation}
\label{eq:second-radiation-asymptotic}
  -\operatorname{Im}
  \inner{\cM_\delta^+(\omega)a}{a}
  =
  \delta\omega^3
  \inner{\Gamma_1a}{a}
  +
  \mathcal O(\delta^3)|a|^2.
\end{equation}
In particular, a vector in
\(\cD_0\setminus\Ker\Gamma_1\) is dark at the force level but bright at
the elastic dipole level.
\end{theorem}

\begin{proof}
For \(a\in\cD_0\), one has \(\cF_0Ba=0\).  Combining
\eqref{eq:F-Taylor} and \eqref{eq:Bdelta-expansion} yields
\[
  \cF(\omega)B_\delta(\omega)a
  =
  \omega\cF_1Ba
  +
  \mathcal O(\omega^2+\delta)|a|
\]
in \(\cZ\).  Therefore
\[
\begin{split}
  \norm{\cF(\omega)B_\delta(\omega)a}_{\cZ}^2
  &=
  \omega^2\norm{\cF_1Ba}_{\cZ}^2\\
  &\quad
  +
  \mathcal O\bigl(
  \omega^3+\delta\omega+\omega^4+\delta^2
  \bigr)|a|^2.
\end{split}
\]
Multiply by \(\delta\omega\) and use
\(\omega\asymp\delta^{1/2}\).  Every displayed error is
\(\mathcal O(\delta^3)|a|^2\).  The exact identity
\eqref{eq:effective-optical-operator} and the definition
\eqref{eq:Gamma1-definition} now give
\eqref{eq:second-radiation-asymptotic}.
\end{proof}

\subsection{Pole continuation and radiative widths}

\subsubsection{Holomorphic continuation and pole multiplicity}

For the remainder of the paper, \(\cM_\delta(\zeta)\) denotes the
outgoing continuation from \(\operatorname{Im}\zeta>0\).

\begin{proposition}
\label{prop:meromorphic-continuation}
There are \(r_0,\delta_0>0\) such that, for every
\(0<\delta<\delta_0\), the maps
\[
  G_\delta(\zeta)^{-1},\quad
  S_\delta^r(\zeta),\quad
  L_\delta^\ell(\zeta),\quad
  \cM_\delta(\zeta)
\]
are holomorphic for
\[
  |\zeta|\le C_{\mathrm{fr}}\delta^{1/2}<r_0.
\]
The cutoff resolvent has a meromorphic continuation to this disk in
the usual outgoing sense \cite{Vainberg1989}, and
its poles coincide with the zeros of
\[
  \det\cM_\delta(\zeta).
\]
The algebraic multiplicity of a pole \(\zeta_*\) is
\begin{equation}
\label{eq:algebraic-multiplicity}
  \frac1{2\pi\ii}
  \tr\int_{\partial B(\zeta_*,r)}
  \cM_\delta(\zeta)^{-1}
  \partial_\zeta\cM_\delta(\zeta)\,\dd\zeta
\end{equation}
when the circle contains no other zero.
\end{proposition}

\begin{proof}
Proposition~\ref{prop:dtn-analytic} gives a holomorphic outgoing
continuation of \(\Lambda(\zeta)\) near zero.  The complement
coercivity proof in Lemma~\ref{lem:complement-inverse} is stable under
small complex perturbations, so \(G_\delta(\zeta)^{-1}\) is
holomorphic throughout the indicated disk.  The definitions
\eqref{eq:right-left-lifts} and \eqref{eq:feshbach-matrix} then give
holomorphy of all finite-dimensional blocks.

The characteristic multiplicity is preserved by the analytic Feshbach
equivalence.  Indeed, in the typed decomposition
\(\cV=J\C^m\oplus\cV_\perp\),
\begin{equation*}
\begin{split}
&\begin{pmatrix}
I&-A_{01}G_\delta^{-1}\\0&I
\end{pmatrix}
\begin{pmatrix}
A_{00}&A_{01}\\A_{10}&G_\delta
\end{pmatrix}
\begin{pmatrix}
I&0\\-G_\delta^{-1}A_{10}&I
\end{pmatrix}\\
&\hspace{8em}=
\begin{pmatrix}
\cM_\delta&0\\0&G_\delta
\end{pmatrix}.
\end{split}
\end{equation*}
Both triangular factors are holomorphic and invertible.  Since
\(G_\delta\) is holomorphically invertible, the Gohberg--Sigal
characteristic multiplicity of the variational family equals that of
\(\cM_\delta\).

The exact formula \eqref{eq:exact-resolvent-identity}, continued
through the layer-potential maps, shows that the only possible poles
are those of \(\cM_\delta(\zeta)^{-1}\).  Conversely, the complex-disk nondegeneracy in
Lemma~\ref{lem:resonant-map-nondegeneracy} prevents a zero of the matrix from
being cancelled by either reconstruction factor.  More precisely, the right
map has a local analytic left inverse on its finite-dimensional range, the
left map has the exact analytic right inverse constructed in that lemma, and
the cutoff equals one on \(D\); hence generalized root chains are preserved.
Thus the pole sets and algebraic multiplicities agree.  Formula
\eqref{eq:algebraic-multiplicity} is the argument principle for the
holomorphic matrix pencil.
\end{proof}

\begin{proposition}
\label{prop:passivity}
The continued cutoff resolvent has no pole in
\(\operatorname{Im}\zeta>0\) and no pole on
\((0,r_0)\).  Hence every pole converging to a positive subwavelength
frequency lies in the open lower half-plane.
\end{proposition}

\begin{proof}
For \(\operatorname{Im}\zeta>0\), the outgoing field decays and is the
physical resolvent solution.  If
\(\operatorname{Re}\zeta\ne0\), then \(\zeta^2\notin[0,\infty)\), so
self-adjointness of \(\cL_\delta\) excludes a pole.  If
\(\operatorname{Re}\zeta=0\), then \(\zeta^2<0\), which lies below the
nonnegative spectrum.  Real positive poles are excluded by the
Rellich/unique-continuation proof of Theorem~\ref{thm:resonant-lap}.  The
remaining local poles are therefore in
the lower half-plane.
\end{proof}

\subsubsection{Simple bright branches}

\begin{theorem}
\label{thm:bright-poles}
Let \(Ka_j=\lambda_ja_j\), \(|a_j|=1\), with \(\lambda_j\) simple, and
assume
\[
  \gamma_j=\inner{\Gamma_0a_j}{a_j}>0.
\]
There exist \(C_j>\gamma_j/2\) and \(\delta_j>0\) such that, for
\(0<\delta<\delta_j\), there is exactly one pole in
\[
  \left|
  \zeta-\delta^{1/2}\lambda_j^{1/2}
  \right|\le C_j\delta,
\]
and it is simple.  It satisfies
\begin{equation}
\label{eq:bright-pole-expansion}
  \zeta_j(\delta)
  =
  \delta^{1/2}\lambda_j^{1/2}
  -
  \frac{\ii\delta}{2}\gamma_j
  +
  \mathcal O(\delta^{3/2}).
\end{equation}
\end{theorem}

\begin{proof}
Let \(P_j=a_j\otimes a_j^*\) and \(Q_j=I-P_j\).  For
\(|\zeta-\delta^{1/2}\lambda_j^{1/2}|\le C_j\delta\), the restriction
\[
  Q_j(\delta K-\zeta^2I)Q_j
\]
is invertible with norm \(O(\delta^{-1})\).  The remaining terms in
\eqref{eq:effective-expansion} have norm \(O(\delta^{3/2})\), so the
\(Q_j\)-equation can be solved by a Neumann series.  The resulting
right vector is
\[
  a_j+b_j(\zeta,\delta),
  \qquad
  b_j\in Q_j\C^m,\qquad
  |b_j|\le C\delta^{1/2}.
\]

The scalar Lyapunov--Schmidt function is
\[
  m_{j,\delta}(\zeta)
  =
  \inner{
  \cM_\delta(\zeta)
  (a_j+b_j(\zeta,\delta))
  }{a_j}.
\]
The off-diagonal correction is quadratic and has size
\[
  O(\delta^{3/2})\,O(\delta^{-1})\,O(\delta^{3/2})
  =
  O(\delta^2).
\]
Using \eqref{eq:effective-error}, one obtains
\begin{equation}
\label{eq:bright-scalar}
  m_{j,\delta}(\zeta)
  =
  \delta\lambda_j-\zeta^2
  -
  \ii\delta\zeta\gamma_j
  +
  O(\delta^2)
\end{equation}
uniformly in the disk, with derivative of the error
\(O(\delta^{3/2})\).

The leading scalar function has one zero in the disk.  Rouch\'e's
theorem gives one zero of \(m_{j,\delta}\), and
\[
  \partial_\zeta m_{j,\delta}
  =
  -2\delta^{1/2}\lambda_j^{1/2}+O(\delta)
\]
there, so the zero is simple.  Substituting
\(\zeta=\delta^{1/2}\lambda_j^{1/2}+h\) into
\eqref{eq:bright-scalar} gives
\[
  -2\delta^{1/2}\lambda_j^{1/2}h
  -
  \ii\delta^{3/2}\lambda_j^{1/2}\gamma_j
  +
  O(\delta^2)=0.
\]
Division by \(2\delta^{1/2}\lambda_j^{1/2}\) proves
\eqref{eq:bright-pole-expansion}.
\end{proof}

\begin{corollary}
\label{cor:bright-Q}
The radiation quality factor of a simple bright branch satisfies
\[
  \frac{\operatorname{Re}\zeta_j(\delta)}
  {-2\operatorname{Im}\zeta_j(\delta)}
  =
  \frac{\lambda_j^{1/2}}{\gamma_j}\delta^{-1/2}
  \bigl(1+O(\delta^{1/2})\bigr).
\]
\end{corollary}

\begin{proof}
Insert \eqref{eq:bright-pole-expansion} into the definition of the
quality factor.
\end{proof}

\subsubsection{Symmetry-protected leading dark branches}

The rank theorem produces many dark directions, but a dark direction
is a resonant mode only when it is invariant under the leading
capacitance matrix.  The following hypothesis is typically enforced by
symmetry:
\begin{equation}
\label{eq:dark-eigenvector-hypothesis}
  Ka_j=\lambda_ja_j,\qquad
  a_j\in\cD_0,\qquad
  |a_j|=1,\qquad
  \lambda_j\ \text{simple}.
\end{equation}

\begin{lemma}
\label{lem:dark-real-center}
Under \eqref{eq:dark-eigenvector-hypothesis}, there are a real number
\(\widehat\omega_j(\delta)\) and a Lyapunov--Schmidt vector
\(\widehat a_j(\delta)\) such that
\[
  \widehat\omega_j(\delta)
  =
  \delta^{1/2}\lambda_j^{1/2}
  +
  O(\delta^{3/2}),
  \qquad
  \widehat a_j(\delta)=a_j+O(\delta),
\]
and the real part of the scalar Lyapunov--Schmidt function has a unique
zero at \(\widehat\omega_j(\delta)\) in an
\(O(\delta^{3/2})\) neighborhood of
\(\delta^{1/2}\lambda_j^{1/2}\).  Moreover,
\[
  \partial_\omega\operatorname{Re}m_{j,\delta}
  (\widehat\omega_j)
  =-2\delta^{1/2}\lambda_j^{1/2}+O(\delta),
  \qquad
  \Gamma_0\widehat a_j=O(\delta).
\]
\end{lemma}

\begin{proof}
Let \(P_j=a_j\otimes a_j^*\) and \(Q_j=I-P_j\).  On the real axis and
in an \(O(\delta^{3/2})\) neighborhood of the unperturbed frequency,
the \(Q_j\)-block is invertible with norm \(O(\delta^{-1})\).  Since
\(\Gamma_0a_j=0\), the whole \(Q_j\)--\(P_j\) block of the radiative
term vanishes.  The remaining off-diagonal block is
\(O(\delta^2)\) by \eqref{eq:effective-error}.  Solving the complement
equation therefore gives
\[
  a_j+b_j(\omega,\delta),
  \qquad b_j\in Q_j\C^m,
  \qquad |b_j|\le C\delta.
\]
Define
\[
  m_{j,\delta}(\omega)
  =
  \inner{
    \cM_\delta^+(\omega)(a_j+b_j(\omega,\delta))
  }{a_j}.
\]
The Schur correction is \(O(\delta^3)\), and hence
\[
  \operatorname{Re}m_{j,\delta}(\omega)
  =\delta\lambda_j-\omega^2+O(\delta^2),
  \qquad
  \partial_\omega\operatorname{Re}m_{j,\delta}(\omega)
  =-2\omega+O(\delta).
\]
The implicit-function theorem gives the asserted unique real zero and
its location.  Put
\(\widehat a_j=a_j+b_j(\widehat\omega_j,\delta)\).
The last estimate follows from \(\Gamma_0a_j=0\) and
\(|b_j|=O(\delta)\).
\end{proof}

\begin{lemma}
\label{lem:dark-leakage}
Under \eqref{eq:dark-eigenvector-hypothesis}, let
\(\widehat\omega_j\) and \(\widehat a_j\) be given by
Lemma~\ref{lem:dark-real-center}.  Then
\begin{equation}
\label{eq:dark-leakage-amplitude}
  \cF(\widehat\omega_j)
  B_\delta(\widehat\omega_j)\widehat a_j
  =
  \widehat\omega_j\cF_1Ba_j
  +O(\delta)
\end{equation}
in \(\cZ\), uniformly as \(\delta\downarrow0\).  Consequently,
\begin{equation}
\label{eq:dark-leakage-flux}
  -\operatorname{Im}
  \inner{
    \cM_\delta^+(\widehat\omega_j)\widehat a_j
  }{
    \widehat a_j
  }
  =
  \delta\widehat\omega_j^3
  \inner{\Gamma_1a_j}{a_j}
  +
  O(\delta^3).
\end{equation}
\end{lemma}

\begin{proof}
Since \(a_j\in\cD_0=\Ker\Gamma_0\),
Theorem~\ref{thm:rank-three} and the Gram representation give
\(\cF_0Ba_j=0\).  Moreover,
\[
  \widehat a_j-a_j=O(\delta),\qquad
  B_\delta(\widehat\omega_j)-B=O(\delta),\qquad
  \widehat\omega_j\asymp\delta^{1/2}.
\]
Using
\(\cF(\omega)=\cF_0+\omega\cF_1+O(\omega^2)\), we obtain
\[
\begin{split}
  \cF(\widehat\omega_j)
  B_\delta(\widehat\omega_j)\widehat a_j
  &=
  \cF_0Ba_j
  +
  \widehat\omega_j\cF_1Ba_j\\
  &\quad
  +O(\widehat\omega_j^2)
  +O(\delta)
  =
  \widehat\omega_j\cF_1Ba_j+O(\delta),
\end{split}
\]
which proves \eqref{eq:dark-leakage-amplitude}.  Squaring the norm gives
\[
  \norm{
    \cF(\widehat\omega_j)
    B_\delta(\widehat\omega_j)\widehat a_j
  }_{\cZ}^2
  =
  \widehat\omega_j^2
  \norm{\cF_1Ba_j}_{\cZ}^2
  +
  O(\delta^{3/2}).
\]
The exact effective optical identity then yields
\[
  -\operatorname{Im}
  \inner{\cM_\delta^+(\widehat\omega_j)\widehat a_j}
  {\widehat a_j}
  =
  \delta\widehat\omega_j^3
  \norm{\cF_1Ba_j}_{\cZ}^2
  +
  O(\delta^3).
\]
Because \(a_j\in\cD_0\), the first term equals
\(\delta\widehat\omega_j^3
\inner{\Gamma_1a_j}{a_j}\), proving
\eqref{eq:dark-leakage-flux}.
\end{proof}

\begin{theorem}
\label{thm:dark-poles}
Assume \eqref{eq:dark-eigenvector-hypothesis} and
\[
  \gamma_{j,1}
  :=
  \inner{\Gamma_1a_j}{a_j}>0.
\]
There exist \(C_j,\delta_j>0\) such that, for \(0<\delta<\delta_j\),
exactly one pole lies in
\[
  |\zeta-\widehat\omega_j(\delta)|\le C_j\delta^2.
\]
It is simple and satisfies
\begin{equation}
\label{eq:dark-pole-expansion}
  \zeta_j^{\mathrm{dark}}(\delta)
  =
  \widehat\omega_j(\delta)
  -
  \frac{\ii\delta^2\lambda_j}{2}\gamma_{j,1}
  +
  o(\delta^2).
\end{equation}
\end{theorem}

\begin{proof}
Let \(m_{j,\delta}\) be the scalar Lyapunov--Schmidt function from
Lemma~\ref{lem:dark-real-center}.  The complement equation implies that
\(\cM_\delta^+(\widehat\omega_j)\widehat a_j\) is a multiple of \(a_j\).
Since \(\widehat a_j-a_j\perp a_j\),
\[
  m_{j,\delta}(\widehat\omega_j)
  =
  \inner{
    \cM_\delta^+(\widehat\omega_j)\widehat a_j
  }{
    \widehat a_j
  }.
\]
By the definition of \(\widehat\omega_j\),
\(\operatorname{Re}m_{j,\delta}(\widehat\omega_j)=0\), while
\[
  \partial_\zeta m_{j,\delta}(\widehat\omega_j)
  =
  -2\widehat\omega_j+O(\delta)
  =
  -2\delta^{1/2}\lambda_j^{1/2}+O(\delta).
\]
Lemma~\ref{lem:dark-leakage} gives
\[
  -\operatorname{Im}
  m_{j,\delta}(\widehat\omega_j)
  =
  \delta\widehat\omega_j^3\gamma_{j,1}
  +
  O(\delta^3)
  =
  \delta^{5/2}\lambda_j^{3/2}\gamma_{j,1}
  +
  O(\delta^3).
\]
The derivative is bounded away from zero on the natural
\(O(\delta^2)\) disk after division by \(\delta^{1/2}\), and the
holomorphic error has the derivative control from
\eqref{eq:effective-error}.  Rouch\'e's theorem, equivalently the analytic
implicit-function theorem for the scalar Schur complement, gives exactly one
simple zero in that disk.  Taylor expansion at \(\widehat\omega_j\) gives the
full Newton quotient
\[
  \zeta_j^{\mathrm{dark}}-\widehat\omega_j
  =-
  \frac{m_{j,\delta}(\widehat\omega_j)}
  {\partial_\zeta m_{j,\delta}(\widehat\omega_j)}
  +o(\delta^2),
\]
and hence
\[
  \operatorname{Im}\zeta_j^{\mathrm{dark}}
  =
  -
  \frac{
    \delta^{5/2}\lambda_j^{3/2}\gamma_{j,1}
  }{
    2\delta^{1/2}\lambda_j^{1/2}
  }
  +
  o(\delta^2),
\]
which is \eqref{eq:dark-pole-expansion}.
\end{proof}

\begin{corollary}
\label{cor:dark-real-blowup}
Under the hypotheses of Theorem~\ref{thm:dark-poles},
\[
  \norm{
  \chi R_\delta^+(\widehat\omega_j(\delta))\chi
  }
  \asymp
  \delta^{-5/2}.
\]
\end{corollary}

\begin{proof}
At the real center, the scalar Schur block has imaginary part
\[
  -\delta\widehat\omega_j^3\gamma_{j,1}
  (1+o(1))
  =
  -\delta^{5/2}\lambda_j^{3/2}\gamma_{j,1}
  (1+o(1)).
\]
All complementary inverse blocks are \(O(\delta^{-1})\).  Hence
\(\norm{\cM_\delta^+(\widehat\omega_j)^{-1}}
\asymp\delta^{-5/2}\), and Theorem~\ref{thm:sharp-resolvent-law} transfers
this estimate to the cutoff
resolvent.
\end{proof}

\subsubsection{Multiple static eigenvalues}

The first radiation matrix may be nontrivial on a degenerate eigenspace of
\(K\).  Two reductions are then necessary.  One first eliminates the static
spectral complement \(\cX_\lambda^\perp\); inside \(\cX_\lambda\), one then
separates the force-bright and force-dark channels.  A point which is important
at second order is that the first elimination does not modify the dark real
center at order \(\delta^2\).

Let \(\lambda\) be an eigenvalue of \(K\), let
\(\cX_\lambda=\Ker(K-\lambda I)\), and let \(P_\lambda\) be the orthogonal
projection onto \(\cX_\lambda\).  Set
\[
  \Gamma_\lambda
  =
  P_\lambda\Gamma_0P_\lambda\big|_{\cX_\lambda},
  \qquad
  \cD_\lambda=\Ker\Gamma_\lambda.
\]
Since \(\Gamma_0\ge0\),
\[
  \cD_\lambda
  =
  \cX_\lambda\cap\Ker\Gamma_0.
\]
Let \(P_D\) be the projection of \(\cX_\lambda\) onto \(\cD_\lambda\) and
\(P_B=P_\lambda-P_D\).  On
\(\cB_\lambda=P_B\cX_\lambda\), the compression
\[
  \Gamma_{\lambda,B}
  =
  P_B\Gamma_\lambda P_B\big|_{\cB_\lambda}
\]
is positive definite whenever \(\cB_\lambda\ne\{0\}\).

For complex \(\zeta\) in a disk
\[
  |\zeta-\delta^{1/2}\lambda^{1/2}|\le c_\lambda\delta,
\]
put
\[
  \cM_{\lambda,\delta}^{\perp}(\zeta)
  =
  P_\lambda^\perp\cM_\delta(\zeta)P_\lambda^\perp.
\]
For small \(\delta\), this matrix is holomorphic and invertible with
\(\|\cM_{\lambda,\delta}^{\perp}(\zeta)^{-1}\|\le C\delta^{-1}\),
uniformly in the disk.  Define the holomorphic cluster Schur complement
\begin{equation}
\label{eq:cluster-schur}
  \cM_{\lambda,\delta}^{\mathrm{cl}}(\zeta)
  =
  P_\lambda\cM_\delta(\zeta)P_\lambda
  -
  P_\lambda\cM_\delta(\zeta)P_\lambda^\perp
  \cM_{\lambda,\delta}^{\perp}(\zeta)^{-1}
  P_\lambda^\perp\cM_\delta(\zeta)P_\lambda.
\end{equation}

\begin{lemma}
\label{lem:dark-cluster-schur}
Uniformly for complex \(\zeta\) satisfying
\[
  |\zeta-\delta^{1/2}\lambda^{1/2}|\le C\delta^{3/2},
\]
one has
\begin{align}
  &\|P_\lambda^\perp\cM_\delta(\zeta)P_D\|
   +
   \|P_D\cM_\delta(\zeta)P_\lambda^\perp\|
   \le C\delta^2,
   \label{eq:dark-perp-coupling}\\
  &P_D\cM_{\lambda,\delta}^{\mathrm{cl}}(\zeta)P_D
   =
   P_D\cM_\delta(\zeta)P_D+O(\delta^3),
   \label{eq:dark-cluster-dd}\\
  &P_D\cM_{\lambda,\delta}^{\mathrm{cl}}(\zeta)P_B
   =
   P_D\cM_\delta(\zeta)P_B+O(\delta^{5/2}),
   \label{eq:dark-cluster-db}
\end{align}
and also
\begin{equation}
\label{eq:dark-cluster-bd}
  P_B\cM_{\lambda,\delta}^{\mathrm{cl}}(\zeta)P_D
  =
  P_B\cM_\delta(\zeta)P_D+O(\delta^{5/2}).
\end{equation}
All remainders are holomorphic and uniform in the indicated disk.
\end{lemma}

\begin{proof}
Use
\[
  \cM_\delta(\zeta)
  =
  \delta K-\zeta^2I
  -\ii\delta\zeta\Gamma_0
  +E_\delta(\zeta).
\]
The first two terms are block diagonal with respect to the spectral decomposition
of \(K\).  If \(a\in\cD_\lambda\), positivity of \(\Gamma_0\) and
\(\inner{\Gamma_0a}{a}=0\) imply \(\Gamma_0a=0\), and hence
\(P_D\Gamma_0=\Gamma_0P_D=0\).  The only coupling from \(\cD_\lambda\) to
\(\cX_\lambda^\perp\) is therefore \(E_\delta^+\), which is
\(O(\delta^2)\) on the stated scale.  This proves
\eqref{eq:dark-perp-coupling}.

On \(\cX_\lambda^\perp\), the static gap of \(K\) gives
\(\|\cM_{\lambda,\delta}^{\perp}(\zeta)^{-1}\|=O(\delta^{-1})\).
The \(D\)--\(D\) Schur
correction is therefore
\[
  O(\delta^2)\,O(\delta^{-1})\,O(\delta^2)=O(\delta^3),
\]
which gives \eqref{eq:dark-cluster-dd}.  From \(\cB_\lambda\) to
\(\cX_\lambda^\perp\), the radiative block may be \(O(\delta^{3/2})\), while
the \(D\)-side factor remains \(O(\delta^2)\).  Hence the mixed cluster
correction is \(O(\delta^{5/2})\), proving
\eqref{eq:dark-cluster-db}.  Interchanging the two off-diagonal factors gives
the same estimate directly for the opposite mixed block and proves
\eqref{eq:dark-cluster-bd}; no real-axis adjoint argument is used.
\end{proof}

Suppose now \(\cD_\lambda\ne\{0\}\), and set
\[
  \mathsf H_{\lambda,D}
  =
  P_D\mathsf H_\lambda P_D\big|_{\cD_\lambda}.
\]
If \(\cB_\lambda=\{0\}\), define
\(\Theta_\lambda=P_D\Gamma_1P_D\).  Otherwise set
\begin{equation}
\label{eq:Theta-lambda}
  \Theta_\lambda
  =
  P_D\Gamma_1P_D
  +
  \lambda^{-2}
  P_D\mathsf H_\lambda P_B
  \Gamma_{\lambda,B}^{-1}
  P_B\mathsf H_\lambda P_D.
\end{equation}
Both terms are Hermitian nonnegative, since \(\Gamma_1\ge0\),
\(\Gamma_{\lambda,B}^{-1}>0\), and \(\mathsf H_\lambda\) is Hermitian;
hence \(\Theta_\lambda\ge0\).

\begin{lemma}
\label{lem:multiple-corrected-dark-state}
Let \(r_*\) be a simple eigenvalue of \(\mathsf H_{\lambda,D}\), with normalized
eigenvector \(a_*\in\cD_\lambda\), and put \(P_*=a_*\otimes a_*^*\).  For
small \(\delta\), Lyapunov--Schmidt reduction gives a vector
\(a_{*,\delta}(\zeta)\), holomorphic in a complex neighborhood of
\[
  \delta^{1/2}\lambda^{1/2}
  +\frac{\delta^{3/2}}{2\lambda^{1/2}}r_*,
\]
and uniquely normalized by
\begin{equation}
\label{eq:multiple-graph-equation}
  (I-P_*)\cM_\delta(\zeta)a_{*,\delta}(\zeta)=0,
  \qquad
  a_*^*a_{*,\delta}(\zeta)=1.
\end{equation}
Define the holomorphic scalar Schur function
\begin{equation*}
  m_{*,\delta}(\zeta)
  =a_*^*\cM_\delta(\zeta)a_{*,\delta}(\zeta).
\end{equation*}
There is a unique real center in the natural window, characterized exactly by
\[
  \operatorname{Re}m_{*,\delta}(\widehat\omega_*(\delta))=0,
\]
and it satisfies
\begin{equation}
\label{eq:multiple-dark-center}
  \widehat\omega_*(\delta)
  =
  \delta^{1/2}\lambda^{1/2}
  +
  \frac{\delta^{3/2}}{2\lambda^{1/2}}r_*
  +
  o(\delta^{3/2}).
\end{equation}
At this center, writing
\(a_{*,\delta}=a_{*,\delta}(\widehat\omega_*)\), one has the decomposition
\[
  a_{*,\delta}
  =
  a_*+d_D+b_B+b_\perp,
\]
where
\[
  d_D\in\cD_\lambda\cap a_*^\perp,
  \qquad d_D=o(1),
  \qquad
  b_\perp\in\cX_\lambda^\perp,
  \quad \|b_\perp\|=O(\delta),
\]
and, when \(\cB_\lambda\ne\{0\}\),
\begin{equation}
\label{eq:multiple-bright-admixture}
  b_B
  =
  -\frac{\ii\delta^{1/2}}{\lambda^{1/2}}
  \Gamma_{\lambda,B}^{-1}
  P_B\mathsf H_\lambda a_*
  +
  o(\delta^{1/2});
\end{equation}
if \(\cB_\lambda=\{0\}\), set \(b_B=0\).  Moreover,
\begin{equation}
\label{eq:multiple-corrected-farfield}
\begin{split}
  \cF(\widehat\omega_*)B_\delta(\widehat\omega_*)a_{*,\delta}
  &=
  \widehat\omega_*\cF_1Ba_*
  -
  \frac{\ii\delta^{1/2}}{\lambda^{1/2}}
  \cF_0B\Gamma_{\lambda,B}^{-1}P_B\mathsf H_\lambda a_*\\
  &\quad+o(\delta^{1/2}),
\end{split}
\end{equation}
with the second term omitted when \(\cB_\lambda=\{0\}\).  Consequently,
\begin{equation}
\label{eq:multiple-dark-power-corrected}
  \norm{
    \cF(\widehat\omega_*)B_\delta(\widehat\omega_*)a_{*,\delta}
  }_{\cZ}^2
  =
  \delta\lambda
  \inner{\Theta_\lambda a_*}{a_*}
  +
  o(\delta).
\end{equation}
In particular, \(\Theta_\lambda\ge0\).
\end{lemma}

\begin{proof}
By Lemma~\ref{lem:dark-cluster-schur}, eliminating
\(\cX_\lambda^\perp\) changes the \(D\)--\(D\) block only by
\(O(\delta^3)\).  For real
\[
  \omega
  =
  \delta^{1/2}\lambda^{1/2}
  +
  \delta^{3/2}\sigma,
  \qquad \sigma\in\R,
\]
Proposition~\ref{prop:second-real-coefficient} therefore gives
\[
  \delta^{-2}
  P_D\operatorname{Re}_{\mathrm{op}}
  \cM_{\lambda,\delta}^{\mathrm{cl}}(\omega)P_D
  =
  \mathsf H_{\lambda,D}
  -
  2\lambda^{1/2}\sigma I
  +
  o(1).
\]
For complex \(\sigma\) in a fixed neighborhood of
\(r_*/(2\lambda^{1/2})\), the full holomorphic expansion and
Lemma~\ref{lem:dark-cluster-schur} give, uniformly in that neighborhood,
\begin{align*}
  P_D\cM_{\lambda,\delta}^{\mathrm{cl}}
  (\delta^{1/2}\lambda^{1/2}+\delta^{3/2}\sigma)P_D
  &=\delta^2
  (\mathsf H_{\lambda,D}-2\lambda^{1/2}\sigma I)+o(\delta^2),\\
  P_B\cM_{\lambda,\delta}^{\mathrm{cl}}
  (\delta^{1/2}\lambda^{1/2}+\delta^{3/2}\sigma)P_B
  &=-\ii\delta^{3/2}\lambda^{1/2}\Gamma_{\lambda,B}
  +O(\delta^2),\\
  \cM_{\lambda,\delta}^{\perp}
  (\delta^{1/2}\lambda^{1/2}+\delta^{3/2}\sigma)
  &=\delta P_\lambda^\perp(K-\lambda I)P_\lambda^\perp+o(\delta).
\end{align*}
The middle line is omitted when \(\cB_\lambda=\{0\}\).
The simple gap of \(r_*\), the positivity of \(\Gamma_{\lambda,B}\), and the
static spectral gap on \(\cX_\lambda^\perp\) therefore make the complementary
block in \eqref{eq:multiple-graph-equation} invertible throughout a smaller
fixed rescaled complex neighborhood.  Analytic block inversion constructs the
unique holomorphic graph vector and gives the complex scalar expansion
\[
  \delta^{-2}m_{*,\delta}
  \bigl(\delta^{1/2}\lambda^{1/2}+\delta^{3/2}\sigma\bigr)
  =r_*-2\lambda^{1/2}\sigma+o(1),
\]
with the remainder uniform in that complex neighborhood.  Restricting to
bounded real \(\sigma\) gives
\begin{equation*}
  \delta^{-2}\operatorname{Re}
  m_{*,\delta}
  \bigl(\delta^{1/2}\lambda^{1/2}+\delta^{3/2}\sigma\bigr)
  =r_*-2\lambda^{1/2}\sigma+o(1),
\end{equation*}
and Cauchy's estimate applied on a smaller complex disk gives
\[
  \partial_\omega\operatorname{Re}m_{*,\delta}(\omega)
  =-2\delta^{1/2}\lambda^{1/2}+o(\delta^{1/2})
\]
on the same real window.  The real implicit-function theorem now gives the
unique exact center and the expansion \eqref{eq:multiple-dark-center}.

At that center, solve the complementary equations successively in
\[
  \C^m
  =
  \Span\{a_*\}
  \oplus
  (\cD_\lambda\cap a_*^\perp)
  \oplus
  \cB_\lambda
  \oplus
  \cX_\lambda^\perp.
\]
On \(\cB_\lambda\), the leading block at the real center is
\[
  -\ii\delta^{3/2}\lambda^{1/2}\Gamma_{\lambda,B}
  +O(\delta^2),
\]
while Lemma~\ref{lem:dark-cluster-schur} and
Proposition~\ref{prop:second-real-coefficient} give the coupling from \(a_*\)
as
\[
  \delta^2P_B\mathsf H_\lambda a_*+o(\delta^2).
\]
Thus the bright equation gives \eqref{eq:multiple-bright-admixture}.  The
simple gap of \(\mathsf H_{\lambda,D}\) gives \(d_D=o(1)\).  The forcing of
\(\cX_\lambda^\perp\) from \(a_*\) is \(O(\delta^2)\) by
\eqref{eq:dark-perp-coupling}; the forcing generated by
\(b_B=O(\delta^{1/2})\) is also \(O(\delta^2)\), because the unrestricted
bright-to-complement coupling is \(O(\delta^{3/2})\).  Since the inverse on
\(\cX_\lambda^\perp\) is \(O(\delta^{-1})\), one obtains
\(b_\perp=O(\delta)\).  These estimates give existence and uniqueness of the
normalized graph vector.

Since \(a_*,d_D\in\Ker\Gamma_0\), the Gram representation implies
\(\cF_0Ba_*=\cF_0Bd_D=0\).  Also
\(B_\delta-B=O(\delta)\), \(b_\perp=O(\delta)\), and
\(\widehat\omega_*=O(\delta^{1/2})\).  Expanding
\(\cF(\omega)=\cF_0+\omega\cF_1+O(\omega^2)\) gives
\eqref{eq:multiple-corrected-farfield}.  The two leading terms there are
orthogonal in \(\cZ\):
\(\cF_0^*\cF_1+\cF_1^*\cF_0=0\), together with
\(\cF_0Ba_*=0\), implies
\[
  \inner{\cF_1Ba_*}{\cF_0Bb}_{\cZ}=0
  \qquad\text{for every rigid vector }b.
\]
Furthermore,
\[
  \norm{\cF_0Bb}_{\cZ}^2
  =
  \inner{\Gamma_0b}{b}.
\]
Using \(\widehat\omega_*^2=\delta\lambda+o(\delta)\) and
\eqref{eq:Theta-lambda} proves \eqref{eq:multiple-dark-power-corrected}.
\end{proof}

\begin{theorem}
\label{thm:multiple-splitting}
Let \(\lambda\) have multiplicity \(r\).

\begin{enumerate}[label=\textup{(\roman*)},leftmargin=2.2em]
\item If \(\gamma_1,\ldots,\gamma_r\) are the eigenvalues of
\(\Gamma_\lambda\), repeated with multiplicity, then the \(r\) poles in the
cluster near \(\delta^{1/2}\lambda^{1/2}\) satisfy, as a multiset,
\begin{equation}
\label{eq:multiple-first-splitting}
  \zeta_\ell(\delta)
  =
  \delta^{1/2}\lambda^{1/2}
  -
  \frac{\ii\delta}{2}\gamma_\ell
  +
  O(\delta^{3/2}).
\end{equation}

\item Suppose \(\cD_\lambda\ne\{0\}\), let \(r_*\) and \(a_*\) be as in
Lemma~\ref{lem:multiple-corrected-dark-state}, and assume
\[
  \inner{\Theta_\lambda a_*}{a_*}>0.
\]
Then there are \(C,\delta_0>0\) such that, for
\(0<\delta<\delta_0\), exactly one pole lies in
\[
  |\zeta-\widehat\omega_*(\delta)|\le C\delta^2.
\]
It is simple and
\begin{equation}
\label{eq:multiple-dark-width}
  \zeta_*(\delta)
  =
  \widehat\omega_*(\delta)
  -
  \frac{\ii\delta^2\lambda}{2}
  \inner{\Theta_\lambda a_*}{a_*}
  +
  o(\delta^2).
\end{equation}
\end{enumerate}
\end{theorem}

\begin{proof}
For part \textup{(i)}, eliminate \(\cX_\lambda^\perp\) by
\eqref{eq:cluster-schur}.  In an \(O(\delta)\) neighborhood of
\(\delta^{1/2}\lambda^{1/2}\), the unrestricted off-diagonal coupling is
\(O(\delta^{3/2})\), so the cluster Schur correction is \(O(\delta^2)\).  With
\[
  \zeta
  =
  \delta^{1/2}\lambda^{1/2}
  +
  \delta h,
\]
division by \(\delta^{3/2}\) gives on \(\cX_\lambda\)
\[
  -2\lambda^{1/2}h I_{\cX_\lambda}
  -
  \ii\lambda^{1/2}\Gamma_\lambda
  +
  O(\delta^{1/2}).
\]
The argument principle and finite-dimensional degenerate perturbation theory
\cite{Kato1995} then give \eqref{eq:multiple-first-splitting} as a multiset statement.

For part \textup{(ii)}, let \(m_{*,\delta}\) be the scalar
Lyapunov--Schmidt function associated with the graph vector in
Lemma~\ref{lem:multiple-corrected-dark-state}.  The exact effective optical
identity and \eqref{eq:multiple-dark-power-corrected} give
\[
  -\operatorname{Im}m_{*,\delta}(\widehat\omega_*)
  =
  \delta\widehat\omega_*
  \left[
    \delta\lambda\inner{\Theta_\lambda a_*}{a_*}
    +o(\delta)
  \right].
\]
Moreover,
\[
  \partial_\zeta m_{*,\delta}(\widehat\omega_*)
  =
  -2\widehat\omega_*+o(\delta^{1/2}).
\]
The second-derivative estimate in Theorem~\ref{thm:effective-expansion},
applied to the scalar Schur complement, bounds the Taylor remainder by
\(C|\zeta-\widehat\omega_*|^2\) on this disk.  Rouch\'e's theorem on a
sufficiently large \(C\delta^2\) circle therefore gives exactly one simple
zero.  Dividing the preceding two expressions and using
\(\widehat\omega_*^2=\delta\lambda+o(\delta)\) proves
\eqref{eq:multiple-dark-width}.
\end{proof}

\begin{corollary}
\label{cor:cluster-dark-count}
Let \(\lambda\) be a static eigenvalue of multiplicity
\(r=\dim\cX_\lambda\).  Then
\begin{equation}
\label{eq:cluster-dark-count}
  \rank\Gamma_\lambda\le\min\{3,r\},
  \qquad
  \dim\cD_\lambda
  =r-\rank\Gamma_\lambda
  \ge\max\{r-3,0\}.
\end{equation}
Exactly \(\dim\cD_\lambda\) of the first-order coefficients
\(\gamma_1,\ldots,\gamma_r\) in
\eqref{eq:multiple-first-splitting} vanish.  Consequently that many poles in
the cluster satisfy
\[
  \zeta_\ell(\delta)
  =\delta^{1/2}\lambda^{1/2}+O(\delta^{3/2}),
  \qquad
  \operatorname{Im}\zeta_\ell(\delta)=o(\delta),
\]
as a multiset.  In particular, a static eigenspace of multiplicity greater
than three necessarily contains force-dark resonant directions.
\end{corollary}

\begin{proof}
Compression cannot increase rank, so Theorem~\ref{thm:rank-three} gives
\(\rank\Gamma_\lambda\le3\).  Positivity of \(\Gamma_0\) yields
\(\Ker\Gamma_\lambda=\cX_\lambda\cap\Ker\Gamma_0=\cD_\lambda\), and
rank--nullity proves \eqref{eq:cluster-dark-count}.  The remaining assertions
follow from the multiset expansion in Theorem~\ref{thm:multiple-splitting}\textup{(i)}.
\end{proof}

\begin{remark}
Part~\textup{(ii)} of Theorem~\ref{thm:multiple-splitting} deliberately assumes
that the relevant eigenvalue of \(\mathsf H_{\lambda,D}\) is simple.  If it is
multiple, the available \(o(\delta^2)\) control of the Hermitian dark block may
produce real splittings larger than the \(\delta^{5/2}\) damping term in the
effective matrix.  A universal second-order width formula then requires either
one further real expansion or an exact symmetry that preserves the degenerate
subspace.  No scalar \(\delta^2\)-width formula is claimed in that case.
\end{remark}

\subsubsection{Real-frequency Lorentzian laws}

\begin{theorem}
\label{thm:rank-one-expansion}
Let \(\lambda_j\) be simple.  Assume either the bright hypotheses of
Theorem~\ref{thm:bright-poles}, or the dark hypotheses of
Theorem~\ref{thm:dark-poles}, including the stated positivity of the first
nonzero radiation coefficient.  In the corresponding complex neighborhood
there exist a holomorphic right vector \(r_{j,\delta}(\zeta)\), a holomorphic
left covector
\(\ell_{j,\delta}(\zeta)\in(\C^m)^*\), and a holomorphic scalar
\(m_{j,\delta}(\zeta)\), normalized so that
\[
  r_{j,\delta}(\zeta)=a_j+o(1),
  \qquad
  \ell_{j,\delta}(\zeta)(x)=\inner{x}{a_j}+o(1)\norm{x},
  \qquad
  \ell_{j,\delta}(\zeta)\bigl(r_{j,\delta}(\zeta)\bigr)=1.
\]
Writing \((r\otimes\ell)x=\ell(x)r\), one has
the meromorphic identity
\begin{equation}
\label{eq:matrix-rank-one}
  \cM_\delta(\zeta)^{-1}
  =
  \frac{r_{j,\delta}(\zeta)\otimes\ell_{j,\delta}(\zeta)}
       {m_{j,\delta}(\zeta)}
  +H_{j,\delta}(\zeta),
  \qquad
  \norm{H_{j,\delta}(\zeta)}\le C\delta^{-1}.
\end{equation}
Here \(H_{j,\delta}\) is holomorphic throughout the neighborhood.
On the outgoing real boundary, write
\(m_{j,\delta}^+(\omega)=m_{j,\delta}(\omega)\).  For a bright branch and
\(|\omega-\omega_{j,0}|\le c\delta\),
\begin{equation}
\label{eq:bright-Lorentzian-denominator}
  m_{j,\delta}^+(\omega)
  =
  -2\omega_{j,0}(\omega-\omega_{j,0})
  -
  \ii\delta\omega_{j,0}\gamma_j
  +
  O(\delta^2).
\end{equation}
For a dark branch and
\(|\omega-\widehat\omega_j|\le c\delta^2\),
\begin{equation}
\label{eq:dark-Lorentzian-denominator}
  m_{j,\delta}^+(\omega)
  =
  -2\widehat\omega_j(\omega-\widehat\omega_j)
  -
  \ii\delta\widehat\omega_j^3\gamma_{j,1}
  +
  o(\delta^{5/2}).
\end{equation}
The errors in both formulas, as well as the bound in
\eqref{eq:matrix-rank-one}, are uniform on the displayed shrinking windows.
\end{theorem}

\begin{proof}
Put \(P=P_j\), \(Q=Q_j\), and identify \(P\C^m\) with \(\C a_j\).  In the
decomposition \(P\C^m\oplus Q\C^m\), write
\[
  \cM_\delta(\zeta)
  =
  \begin{pmatrix}
    \mathsf a(\zeta)&\mathsf b(\zeta)\\
    \mathsf c(\zeta)&\mathsf D(\zeta)
  \end{pmatrix}.
\]
On both shrinking neighborhoods in the theorem, the static spectral gap gives
\[
  \mathsf D(\zeta)^{-1}
  \in\mathcal B(Q\C^m),
  \qquad
  \norm{\mathsf D(\zeta)^{-1}}\le C\delta^{-1}.
\]
All four blocks and \(\mathsf D^{-1}\) are holomorphic.  Define the scalar
Schur complement
\[
  \mathsf s(\zeta)
  =\mathsf a(\zeta)
  -\mathsf b(\zeta)\mathsf D(\zeta)^{-1}\mathsf c(\zeta),
\]
where the right side is identified with its coefficient on \(a_j\).  Set
\begin{align*}
  r_0(\zeta)
  &=a_j-\mathsf D(\zeta)^{-1}\mathsf c(\zeta)a_j,\\
  \widetilde\ell_0(\zeta)x
  &=a_j^*Px
    -a_j^*\mathsf b(\zeta)\mathsf D(\zeta)^{-1}Qx.
\end{align*}
These are respectively a holomorphic vector and a holomorphic covector.  In
particular, \(\widetilde\ell_0\) is not obtained by taking the Hermitian
adjoint of \(r_0\), which would be anti-holomorphic.  Direct block inversion
gives the exact identity
\begin{equation}
\label{eq:explicit-rank-one-block-inverse}
  \cM_\delta(\zeta)^{-1}
  =
  \frac{r_0(\zeta)\otimes\widetilde\ell_0(\zeta)}
       {\mathsf s(\zeta)}
  +Q\mathsf D(\zeta)^{-1}Q.
\end{equation}

Let \(d(\zeta)=\widetilde\ell_0(\zeta)r_0(\zeta)\).  On a bright window,
\(\mathsf b,\mathsf c=O(\delta^{3/2})\), and hence
\[
  r_0=a_j+O(\delta^{1/2}),\qquad
  \widetilde\ell_0=a_j^*+O(\delta^{1/2}),\qquad
  d=1+O(\delta).
\]
On a dark window, \(\Gamma_0a_j=0\) annihilates both leading radiative mixed
blocks, so \(\mathsf b,\mathsf c=O(\delta^2)\); consequently
\[
  r_0=a_j+O(\delta),\qquad
  \widetilde\ell_0=a_j^*+O(\delta),\qquad
  d=1+O(\delta^2).
\]
After decreasing the neighborhoods, \(d\) has no zero.  Define
\[
  r_{j,\delta}=r_0,
  \qquad
  \ell_{j,\delta}=d^{-1}\widetilde\ell_0,
  \qquad
  m_{j,\delta}=d^{-1}\mathsf s,
  \qquad
  H_{j,\delta}=Q\mathsf D^{-1}Q.
\]
Then \(\ell_{j,\delta}(r_{j,\delta})=1\), and
\eqref{eq:explicit-rank-one-block-inverse} becomes
\eqref{eq:matrix-rank-one}.  It also gives directly
\(\norm{H_{j,\delta}}\le C\delta^{-1}\).

It remains to record the scalar expansions.  In the bright case,
\eqref{eq:bright-scalar} and
\[
  \delta\lambda_j-\omega^2
  =-2\omega_{j,0}(\omega-\omega_{j,0})
   -(\omega-\omega_{j,0})^2
\]
give \eqref{eq:bright-Lorentzian-denominator}; replacing \(\mathsf s\) by
\(m_{j,\delta}=d^{-1}\mathsf s\) changes only the stated
\(O(\delta^2)\) remainder.  In the dark case, the exact definition of
\(\widehat\omega_j\), Lemma~\ref{lem:dark-leakage}, and
\(\partial_\omega\mathsf s(\widehat\omega_j)
=-2\widehat\omega_j+O(\delta)\) yield
\eqref{eq:dark-Lorentzian-denominator}.  Here the factor
\(d^{-1}=1+O(\delta^2)\) changes the denominator by
\(o(\delta^{5/2})\).  All estimates are uniform on the displayed windows.
\end{proof}

\begin{corollary}
\label{cor:field-enhancement}
Let \(f\in L_s^2(\R^3)^3\) be supported where \(\chi=1\), and define
\[
  \mathfrak c_{j,\delta}(f,\omega)
  =
  \ell_{j,\delta}(\omega)
  \bigl(\Psi_\delta^+(\omega)f\bigr).
\]
Then
\begin{align}
  \chi R_\delta^+(\omega)f
  &=
  \frac{\mathfrak c_{j,\delta}(f,\omega)}
  {m_{j,\delta}^+(\omega)}
  \chi\Phi_\delta^+(\omega)
  r_{j,\delta}(\omega)
  +
  \mathcal H_{j,\delta}^{\mathrm{fld}}(\omega)f,
  \label{eq:field-rank-one}\\
  \norm{\mathcal H_{j,\delta}^{\mathrm{fld}}(\omega)f}_{H_{-s}^1}
  &\le C\delta^{-1}\norm{f}_{L_s^2}.
  \label{eq:field-remainder}
\end{align}
If the coupling coefficient stays bounded away from zero, the peak
sizes are \(\delta^{-3/2}\) for a bright branch and
\(\delta^{-5/2}\) for a second-order-bright dark branch.
\end{corollary}

\begin{proof}
Insert \eqref{eq:matrix-rank-one} into the exact boundary-value
decomposition \eqref{eq:boundary-resolvent-decomposition}.  The regular
block is \(O(1)\), the left and right maps are \(O(1)\), and the
complementary matrix inverse is \(O(\delta^{-1})\).  This proves
\eqref{eq:field-rank-one}--\eqref{eq:field-remainder}.  At the real
center, the two denominators have sizes
\(\delta^{3/2}\) and \(\delta^{5/2}\), respectively; the right
resonant map is bounded below by Lemma~\ref{lem:resonant-map-nondegeneracy}.
\end{proof}

\section{Symmetric Models and a Conditional Shape Crossover}
\label{sec:examples}

\subsection{A spherical inclusion}

Let \(D=B_a(0)\), assume that \(\rho_D\) is the constant
\(\rho_D>0\), and use the natural rigid coordinates
\[
  u(x)=t+b\times x,\qquad t,b\in\R^3.
\]
The formulas are first computed for real \(t,b\) and then extended
sesquilinearly to \(t,b\in\C^3\).
The corresponding mass matrix is
\begin{equation}
\label{eq:sphere-mass}
  M_{\mathrm{nat}}
  =
  \diag(m_T I_3,m_R I_3),
  \qquad
  m_T=\frac{4\pi}{3}\rho_Da^3,
  \qquad
  m_R=\frac{8\pi}{15}\rho_Da^5.
\end{equation}

\begin{proposition}
\label{prop:sphere}
For the sphere,
\begin{equation}
\label{eq:sphere-K}
  K_{\mathrm{nat}}
  =
  \diag(\kappa_T I_3,\kappa_R I_3),
\end{equation}
where
\begin{equation}
\label{eq:sphere-kappas}
  \kappa_T
  =
  \frac{12\pi a\mu_0(\lambda_0+2\mu_0)}
  {2\lambda_0+5\mu_0},
  \qquad
  \kappa_R=8\pi\mu_0a^3.
\end{equation}
The leading radiation matrix in the same coordinates is
\begin{equation}
\label{eq:sphere-Gamma0}
  \Gamma_{0,\mathrm{nat}}
  =
  \diag(
  \gamma_{\mathrm{el}}\kappa_T^2I_3,0_{3\times3}).
\end{equation}
Thus the translational branches are bright, while all three rotational
branches are force-dark.  In natural rotational coordinates,
\[
  \Gamma_{1,\mathrm{nat}}
  =
  \gamma_R^{(1)}I_3,
  \qquad
  \gamma_R^{(1)}
  =
  \frac{8\pi}{3}a^6
  \frac{\rho_0^{3/2}}{\mu_0^{1/2}}>0.
\]
After mass normalization, the corresponding blocks are
\begin{equation}
\label{eq:sphere-normalized-blocks}
  K
  =
  \diag\left(
  \frac{\kappa_T}{m_T}I_3,
  \frac{\kappa_R}{m_R}I_3
  \right),
  \quad
  \Gamma_0
  =
  \diag\left(
  \frac{\gamma_{\mathrm{el}}\kappa_T^2}{m_T}I_3,
  0
  \right),
  \quad
  \Gamma_1\big|_{\mathrm{rot}}
  =\frac{\gamma_R^{(1)}}{m_R}I_3.
\end{equation}
Hence the rotational poles have width of order \(\delta^2\).
\end{proposition}

\begin{proof}
The sphere and the isotropic exterior equations are invariant under the
full orthogonal group \(O(3)\).  Translations are polar vectors whereas
rotations are axial vectors, so they have opposite parity under
reflections.  They are therefore inequivalent as \(O(3)\)
representations, although they would be equivalent if one used only
\(SO(3)\).  Schur's lemma and parity give the block form
\eqref{eq:sphere-K} and exclude translation--rotation cross blocks.

For a rotational datum \(b\times x\), the static exterior solution is
\[
  U_b^0(x)=\frac{a^3}{|x|^3}\,b\times x.
\]
It is divergence free, harmonic componentwise, equals \(b\times x\)
at \(|x|=a\), and decays at infinity.  Its traction at the sphere is
\[
  T_0^\nu U_b^0=-3\mu_0\,b\times\widehat x.
\]
Consequently,
\[
  -\int_{\partial B_a}
  T_0^\nu U_b^0\cdot\overline{b\times x}\,\dd S
  =
  8\pi\mu_0a^3|b|^2,
\]
which proves the formula for \(\kappa_R\).

For a unit translation \(t\), seek a static solution in the form
\[
  U_t^0(x)=f(r)t+h(r)(t\cdot x)x.
\]
The decaying Navier solutions satisfying \(U_t^0=t\) at \(r=a\) are
\begin{align*}
  f(r)
  &=
  \frac{3a(\lambda_0+3\mu_0)}
  {2(2\lambda_0+5\mu_0)}\,r^{-1}
  +
  \frac{a^3(\lambda_0+\mu_0)}
  {2(2\lambda_0+5\mu_0)}\,r^{-3},\\
  h(r)
  &=
  \frac{\lambda_0+\mu_0}{\lambda_0+3\mu_0}
  \frac{3a(\lambda_0+3\mu_0)}
  {2(2\lambda_0+5\mu_0)}\,r^{-3}
  -
  \frac{3a^3(\lambda_0+\mu_0)}
  {2(2\lambda_0+5\mu_0)}\,r^{-5}.
\end{align*}
Substitution verifies the Navier equation and the boundary condition.
Integrating its traction over \(\partial B_a\) gives
\[
  \int_{\partial B_a}T_0^\nu U_t^0\,\dd S
  =
  -\kappa_Tt,
\]
with \(\kappa_T\) as in \eqref{eq:sphere-kappas}.  This proves the
translation block.  Formula \eqref{eq:sphere-mass} follows by direct
integration.

The total-force map is \(\kappa_T I_3\) on translations and zero on
rotations.  Equation \eqref{eq:Gamma0-force} gives
\eqref{eq:sphere-Gamma0}.

For the dynamic rotational field, put
\[
  k_s=\omega\sqrt{\rho_0/\mu_0}.
\]
Let \(g_b(x)=b\times x\) on \(\partial B_a\).  The exact outgoing toroidal
solution with boundary value \(g_b\) is
\[
  U_b^+(r,\omega)
  =
  \frac{a\,h_1^{(1)}(k_sr)}
  {h_1^{(1)}(k_sa)}
  \,b\times\widehat x,
\]
where \(h_1^{(1)}\) is the spherical Hankel function.  Since
\[
  h_1^{(1)}(w)=-\frac{\ii}{w^2}+O(1)
  \quad(w\to0),
  \qquad
  h_1^{(1)}(w)=-\frac{e^{\ii w}}{w}
  +O(w^{-2})
  \quad(|w|\to\infty),
\]
its shear far-field amplitude is
\[
  F_s(\omega)g_b
  =
  -\ii\omega a^3
  \sqrt{\frac{\rho_0}{\mu_0}}\,
  b\times\widehat x
  +
  O(\omega^2),
  \qquad
  F_p(\omega)g_b=0.
\]
Using
\[
  \int_{\mathbb S^2}|b\times\widehat x|^2\,\dd\widehat x
  =
  \frac{8\pi}{3}|b|^2
\]
and \(c_s=\sqrt{\rho_0\mu_0}\) gives the displayed value of
\(\gamma_R^{(1)}\).  Finally, conjugating the three natural-coordinate
matrices by \(M_{\mathrm{nat}}^{-1/2}\) gives
\eqref{eq:sphere-normalized-blocks}.
\end{proof}

\begin{corollary}
\label{cor:sphere-frequencies}
In natural coordinates,
\[
  \zeta_T(\delta)
  =
  \delta^{1/2}
  \left(
  \frac{
  9\mu_0(\lambda_0+2\mu_0)
  }{
  \rho_Da^2(2\lambda_0+5\mu_0)
  }
  \right)^{1/2}
  -\frac{\ii\delta}{2}
  \frac{\gamma_{\mathrm{el}}\kappa_T^2}{m_T}
  +O(\delta^{3/2}),
\]
and
\[
  \zeta_R(\delta)
  =
  \widehat\omega_R(\delta)
  -\frac{\ii\delta^2}{2}
  \frac{\kappa_R}{m_R}
  \frac{\gamma_R^{(1)}}{m_R}
  +o(\delta^2),
  \qquad
  \widehat\omega_R(\delta)
  =\delta^{1/2}
  \left(\frac{15\mu_0}{\rho_Da^2}\right)^{1/2}
  +O(\delta^{3/2}).
\]
Each displayed pole has algebraic multiplicity three.  In particular,
the first imaginary correction to the translational triplet is of order
\(\delta\), whereas that of the rotational triplet is of order
\(\delta^2\).
\end{corollary}

\begin{proof}
The generalized eigenvalues of
\((K_{\mathrm{nat}},M_{\mathrm{nat}})\) give the two leading real
frequencies.  They are distinct because
\[
  \frac{\kappa_R}{m_R}-\frac{\kappa_T}{m_T}
  =
  \frac{3\mu_0(7\lambda_0+19\mu_0)}
  {\rho_Da^2(2\lambda_0+5\mu_0)}>0.
\]
The full effective pencil commutes with \(O(3)\); parity
separates its translational and rotational sectors, and Schur's lemma
makes each sector a scalar pencil tensored with \(I_3\).  The scalar
arguments in Theorems~\ref{thm:bright-poles} and
\ref{thm:dark-poles}, together with
\eqref{eq:sphere-normalized-blocks}, therefore apply in each sector and
give a zero of multiplicity three.  This symmetry reduction is needed
because the corresponding eigenvalues of \(K\) are not simple on the
full six-dimensional rigid space.
\end{proof}

\subsection{A symmetric dimer and a conditional shape crossover}

Let \(D_1\) and \(D_2\) be congruent and exchanged by an isometry
\(\mathsf E\) of the reference configuration.  Assume that the two components
carry identical reference Lam\'e parameters and density, so both the exterior
medium and the interior mass form are invariant under \(\mathsf E\).  The
induced unitary action on \(\cN\) then commutes with \(K\).

\begin{proposition}
\label{prop:dimer}
Suppose an eigenspace of \(K\) contains an exchange-odd vector
\(a_-\), normalized by \(|a_-|=1\), whose static single-layer densities on the two components are mapped
to one another with opposite resultant forces.  Then
\[
  \mathfrak qBa_-=0,
  \qquad
  \Gamma_0a_-=0.
\]
If the corresponding eigenvalue is simple in its symmetry class and
\(\inner{\Gamma_1a_-}{a_-}>0\), then its odd-sector pole satisfies
\[
  \zeta_-(\delta)
  =
  \widehat\omega_-(\delta)
  -
  \frac{\ii\delta^2\lambda_-}{2}
  \inner{\Gamma_1a_-}{a_-}
  +
  o(\delta^2).
\]
\end{proposition}

\begin{proof}
The static single-layer operator commutes with the exchange isometry.  For an
exchange-odd datum, the transformed densities have opposite resultants, so
their integrals cancel.  Hence \(\mathfrak qBa_-=0\), and
Theorem~\ref{thm:rank-three} gives \(\Gamma_0a_-=0\).  The pole formula is
Theorem~\ref{thm:dark-poles} applied in the odd symmetry sector.
\end{proof}

We next break the exchange symmetry.  The full-space simplicity assumption in
the next proposition is intentional: once the symmetry is broken, the odd and
even subspaces are no longer invariant, so simplicity only inside the odd
sector would not by itself produce a differentiable eigenvector branch.

Fix
\[
  \Xi_\eta(x)=x+\eta V(x),
  \qquad
  V\in C_c^{2,\alpha}(\R^3;\R^3),
  \qquad
  D_j(\eta)=\Xi_\eta(D_j).
\]
On each moved component we keep the same homogeneous isotropic Lam\'e
parameters and density as on its reference component; thus the material is
not transformed into an anisotropic push-forward medium.  We identify the
moving rigid spaces with a fixed \(\C^m\) by using translations and rotations
about the moving component centers, forming their \(\rho_D\)-mass Gram matrix,
and applying its positive inverse square root.  Boundary operators are pulled
back by \(\Xi_\eta\).  More precisely, if
\(\Sigma_\eta=\partial D(\eta)\), let
\[
  \mathcal P_\eta:H^{1/2}(\Sigma_\eta)^3\to H^{1/2}(\Sigma)^3,
  \qquad
  \mathcal P_\eta g=g\circ\Xi_\eta,
\]
and let \(\mathcal P_{\eta,*}:H^{-1/2}(\Sigma_\eta)^3
\to H^{-1/2}(\Sigma)^3\) be its dual transport, including the surface
Jacobian, characterized by
\[
  \inner{\mathcal P_{\eta,*}\mu}{\varphi}_{\Sigma}
  =\inner{\mu}{\mathcal P_\eta^{-1}\varphi}_{\Sigma_\eta}.
\]
With this convention, define the fixed-space boundary operators by
\[
  \widetilde{\mathsf S}_{\eta}(\zeta)
  =\mathcal P_\eta\mathsf S_{\eta}(\zeta)\mathcal P_{\eta,*}^{-1},
  \qquad
  \widetilde{\mathsf K}_{\eta}^*(\zeta)
  =\mathcal P_{\eta,*}\mathsf K_{\eta}^*(\zeta)
   \mathcal P_{\eta,*}^{-1},
\]
and
\[
  \widetilde\Lambda_\eta(\zeta)
  =\mathcal P_{\eta,*}\Lambda_\eta(\zeta)\mathcal P_\eta^{-1}
  =\left(-\frac12I+\widetilde{\mathsf K}_{\eta}^*(\zeta)\right)
   \widetilde{\mathsf S}_{\eta}(\zeta)^{-1}.
\]
We write \(\widetilde{\mathsf S}_{0,\eta}\) for
\(\widetilde{\mathsf S}_{\eta}(0)\).  The two transports play different
roles: \(\mathcal P_\eta\) acts on Dirichlet traces, whereas
\(\mathcal P_{\eta,*}\) acts on densities and tractions.

We first record the operator regularity required for this moving-geometry
argument.
\begin{lemma}
\label{lem:dimer-shape-regularity}
For \(|\eta|\) small enough that the components retain a uniform positive
separation, the pulled-back static single-layer operator and its inverse, the
static DtN map, \(\mathfrak q_\eta\), \(B_\eta\), \(K_\eta\),
\(\Gamma_{0,\eta}\), and \(\Gamma_{1,\eta}\) are \(C^1\) in \(\eta\) in
the fixed operator spaces
\[
  \mathcal B(H^{-1/2}(\Sigma)^3,H^{1/2}(\Sigma)^3),\quad
  \mathcal B(H^{1/2}(\Sigma)^3,H^{-1/2}(\Sigma)^3),
\]
\[
  \mathcal B(H^{1/2}(\Sigma)^3,\C^3),\quad
  \mathcal B(\C^m,H^{1/2}(\Sigma)^3),
\]
and the corresponding finite-dimensional matrix spaces.  The low-frequency
expansions and their first two \(\zeta\)-derivatives are \(C^1\) in \(\eta\), uniformly for
\(\eta\) in a sufficiently small compact interval.  In particular,
\(\cM_{\delta,\eta}(\zeta)\) is
holomorphic in \(\zeta\) and \(C^1\) in \(\eta\) on the joint neighborhood
used below.
\end{lemma}

\begin{proof}
In this proof, a dot denotes \(\partial_\eta|_{\eta=0}\).
For small \(|\eta|\), \(\Xi_\eta\) is a uniformly \(C^{2,\alpha}\)
diffeomorphic family and
\[
  c|x-y|\le
  |\Xi_\eta(x)-\Xi_\eta(y)|
  \le C|x-y|,
  \qquad x,y\in\Sigma.
\]
The composition maps \(\mathcal P_\eta\), their inverses, and their dual
transports are uniformly bounded isomorphisms on the indicated Sobolev
spaces.  We do not differentiate these composition operators on arbitrary
Sobolev data; the required operator-norm differentiability is proved directly
for the pulled-back kernels.  For a smooth fixed density, the kernel of
\(\widetilde{\mathsf S}_{\eta}(\zeta)\) is exactly
\[
  \mathbf G^\zeta(\Xi_\eta(x)-\Xi_\eta(y)),
\]
because the input surface Jacobian is already built into
\(\mathcal P_{\eta,*}\).  Its derivative contains
\[
  \nabla\mathbf G^\zeta(\Xi_\eta(x)-\Xi_\eta(y))
  \bigl(V(x)-V(y)\bigr).
\]
At \(\eta=0\), this gives
\begin{equation}
\label{eq:shape-derivative-single-layer}
  (\dot{\widetilde{\mathsf S}}_{0,0}\mu)(x)
  =\int_\Sigma
  \nabla\mathbf G^0(x-y)\bigl(V(x)-V(y)\bigr)
  \mu(y)\,\dd S_y.
\end{equation}
In the pulled-back traction kernel, the output dual transport contributes
\(J_{\Sigma,\eta}(x)\), together with the transformed normal and first
derivatives of \(\Xi_\eta\).  More precisely,
\[
  \dot J_{\Sigma,0}=\operatorname{div}_\Sigma V,
  \qquad
  \dot\bfnu_0
  =-(I-\bfnu\otimes\bfnu)(DV)^{\mathsf T}\bfnu,
\]
and, initially for smooth \(\mu\),
\begin{equation}
\label{eq:shape-derivative-traction-kernel}
  (\dot{\widetilde{\mathsf K}}_{0,0}^*\mu)(x)
  =\int_\Sigma
  \left.\partial_\eta\right|_{\eta=0}
  \bigl[
    J_{\Sigma,\eta}(x)
    T_{\bfnu_\eta}^{x_\eta}
    \mathbf G^0(\Xi_\eta(x)-\Xi_\eta(y))
  \bigr]\mu(y)\,\dd S_y,
\end{equation}
where \(x_\eta=\Xi_\eta(x)\), and the traction differentiates the first
kernel variable.  Since
\(|V(x)-V(y)|\le C|x-y|\), differentiation in \(\eta\) does not worsen the
pseudohomogeneous order of either diagonal kernel.  Kernels coupling different
components are smooth with all these derivatives because the separation is
uniformly positive.

Apply the boundary-kernel mapping theorem used in
Appendix~\ref{app:dtn} to the kernel and its \(\eta\)-derivative.  It follows
that \(\widetilde{\mathsf S}_{0,\eta}\) and the pulled-back
\(\widetilde{\mathsf K}_{0,\eta}^*\) are \(C^1\) in the asserted operator
norms.  Lemma~\ref{lem:kernel-expansion} applies uniformly to the compact
family of pulled-back kernels.  Applying the same argument to each Taylor
coefficient and to the integral remainder gives the joint \(C^1\)-in-\(\eta\),
holomorphic-in-\(\zeta\) low-frequency expansion, including two
\(\zeta\)-derivatives of the remainder.

The static operator \(\widetilde{\mathsf S}_{0,0}\) is invertible.  Norm
continuity and openness of the invertible group give a uniform inverse for
small \(\eta\), and differentiation of the identity
\(\widetilde{\mathsf S}_{0,\eta}^{-1}
\widetilde{\mathsf S}_{0,\eta}=I\) yields
\[
  \partial_\eta\widetilde{\mathsf S}_{0,\eta}^{-1}
  =-\widetilde{\mathsf S}_{0,\eta}^{-1}
   (\partial_\eta\widetilde{\mathsf S}_{0,\eta})
   \widetilde{\mathsf S}_{0,\eta}^{-1}.
\]
The exact identity for \(\widetilde\Lambda_\eta(\zeta)\) displayed above then
gives the same \(C^1\) dependence and uniform low-frequency expansion for the
DtN map.

For a physical density \(\mu_\eta\) on \(\Sigma_\eta\), set
\(\widetilde\mu_\eta=\mathcal P_{\eta,*}\mu_\eta\).  Dual transport gives
\[
  \int_{\Sigma_\eta}\mu_\eta\,\dd S
  =\inner{\widetilde\mu_\eta}{\mathbf 1}_{\Sigma},
\]
componentwise.  Denote this fixed density-resultant functional by
\(\mathfrak r:H^{-1/2}(\Sigma)^3\to\C^3\).  For pulled-back Dirichlet data,
the total-force map used in Section~\ref{sec:radiation} is
\[
  \mathfrak q_\eta
  =\mathfrak r\widetilde{\mathsf S}_{0,\eta}^{-1}
  \in\mathcal B(H^{1/2}(\Sigma)^3,\C^3).
\]
It is therefore \(C^1\), with
\[
  \dot{\mathfrak q}_0
  =-\mathfrak r\widetilde{\mathsf S}_{0,0}^{-1}
   \dot{\widetilde{\mathsf S}}_{0,0}
   \widetilde{\mathsf S}_{0,0}^{-1}.
\]
At zero frequency, put
\(\dot{\widetilde\Lambda}_0
:=\partial_\eta\widetilde\Lambda_\eta(0)|_{\eta=0}\).  The other first
variations are
\begin{align}
  \dot{\widetilde\Lambda}_0
  & =
  \dot{\widetilde{\mathsf K}}_{0,0}^*
  \widetilde{\mathsf S}_{0,0}^{-1}
  -\left(-\frac12I+\widetilde{\mathsf K}_{0,0}^*\right)
  \widetilde{\mathsf S}_{0,0}^{-1}
  \dot{\widetilde{\mathsf S}}_{0,0}
  \widetilde{\mathsf S}_{0,0}^{-1},
  \label{eq:shape-derivative-dtn}\\
  \dot K_0
  &=-\dot B_0^*\widetilde\Lambda_0B_0
    -B_0^*\dot{\widetilde\Lambda}_0B_0
    -B_0^*\widetilde\Lambda_0\dot B_0.
  \label{eq:shape-derivative-capacitance}
\end{align}
The translated and rotated rigid fields about the moving centers are \(C^1\)
after pullback.  Their mass Gram matrix is a positive \(C^1\) matrix.
Let \(\mathsf R_\eta\) be the raw rigid trace matrix, let \(H_\eta\) be its
mass Gram matrix, and put \(Z_\eta=H_\eta^{-1/2}\).  Then
\[
  B_\eta=\mathsf R_\eta Z_\eta,
  \qquad
  \dot B_0=\dot{\mathsf R}_0Z_0+\mathsf R_0\dot Z_0,
\]
where \(\dot Z_0\) is determined by differentiating
\(Z_\eta H_\eta Z_\eta=I\):
\begin{equation}
\label{eq:inverse-square-root-derivative}
  \dot Z_0H_0Z_0+Z_0\dot H_0Z_0+Z_0H_0\dot Z_0=0.
\end{equation}
Thus the mass-normalized trace map \(B_\eta\), and consequently
\(K_\eta\) and \(\Gamma_{0,\eta}\), are \(C^1\).

The rank-three theorem applies to every member of this shape family.  Its
three positive singular values stay bounded away from zero near \(\eta=0\),
so the spectral projection onto \(\Ker\Gamma_{0,\eta}\) is \(C^1\).  The
pulled-back far-field kernels are smooth in the observation direction, and
their first two low-frequency coefficients are \(C^1\) in \(\eta\).  The
definition \eqref{eq:Gamma1-definition}, with the preceding kernel projection,
therefore gives \(C^1\) dependence of \(\Gamma_{1,\eta}\).

It remains to justify the volume Schur complement, whose domain otherwise
moves with \(\eta\).  Introduce the volume pullback
\[
  \mathsf T_\eta:H^1(D(\eta))^3\longrightarrow\cV,
  \qquad \mathsf T_\eta u=u\circ\Xi_\eta,
\]
and put \(\mathsf A_\eta=D\Xi_\eta\),
\(j_\eta=\det\mathsf A_\eta\), and
\[
  e_\eta^\#(u)=\operatorname{sym}(\nabla u\,\mathsf A_\eta^{-1}).
\]
The pulled-back elastic and mass forms are
\begin{align*}
  \mathfrak a_\eta^\#(u,v)
  &=\int_D C_De_\eta^\#(u):\overline{e_\eta^\#(v)}\,j_\eta\,\dd x,\\
  \mathfrak m_\eta^\#(u,v)
  &=\int_D\rho_Du\cdot\overline v\,j_\eta\,\dd x.
\end{align*}
Their coefficients depend \(C^1\) on \(\eta\), and the transformed elasticity
tensors remain uniformly strongly elliptic.  The apparent anisotropy of
\(\mathfrak a_\eta^\#\) is only a coordinate representation of the homogeneous
isotropic material on \(D(\eta)\).

Let \(\mathcal N(D(\eta))\) denote the physical rigid-motion space on the
moved components and set
\[
  \cN_\eta^\#=\mathsf T_\eta\mathcal N(D(\eta)),
\]
and denote by \(\Pi_\eta^\#\) the
\(\mathfrak m_\eta^\#\)-orthogonal projection onto \(\cN_\eta^\#\).  The
pulled-back translations and rotations, their mass Gram matrix, and its
inverse are \(C^1\).  Consequently
\[
  \Pi_\eta^\#,\quad Q_\eta^\#=I-\Pi_\eta^\#
  \in\mathcal B(\cV)
\]
depend \(C^1\) on \(\eta\).  Set
\(\cV_{\perp,\eta}^\#=Q_\eta^\#\cV\).  Since
\(\|Q_\eta^\#-Q_0^\#\|\to0\), the restriction
\[
  W_\eta:=Q_\eta^\#|_{\cV_{\perp,0}^\#}:
  \cV_{\perp,0}^\#\longrightarrow\cV_{\perp,\eta}^\#
\]
is a \(C^1\) family of isomorphisms for small \(|\eta|\).

Define the reduced operator on \(\cV\) by
\[
  \inner{\cA_{\delta,\eta}^\#(\zeta)u}{v}
  =\mathfrak a_\eta^\#(u,v)
   -\zeta^2\mathfrak m_\eta^\#(u,v)
   -\delta\inner{\widetilde\Lambda_\eta(\zeta)\gamma u}
                    {\gamma v}_{\Sigma},
\]
and let
\[
  G_{\delta,\eta}^\#(\zeta):
  \cV_{\perp,\eta}^\#\longrightarrow
  (\cV_{\perp,\eta}^\#)^*
\]
be its complement block.  On the fixed complement define
\[
  \widehat G_{\delta,\eta}(\zeta)
  =W_\eta^*G_{\delta,\eta}^\#(\zeta)W_\eta,
\]
where \(W_\eta^*\) denotes dual pullback.  The uniform Korn estimate follows
by a fixed-domain compactness argument.  Indeed, otherwise there would be
\(\eta_n\to0\) and \(w_n\in\cV_{\perp,\eta_n}^\#\) such that
\[
  \norm{w_n}_{H^1(D)}=1,
  \qquad
  \mathfrak a_{\eta_n}^\#(w_n,w_n)\longrightarrow0.
\]
Since \(\mathsf A_{\eta_n}\to I\) in \(C^1\), one has
\(e(w_n)\to0\) in \(L^2(D)\).  Korn's inequality and compactness give,
after passage to a subsequence, \(w_n\to r\) in \(H^1(D)^3\) for some
\(r\in\cN_0^\#\).  The moving rigid bases and mass forms converge to their
values at zero.  Passing to the limit in the defining orthogonality of
\(\cV_{\perp,\eta_n}^\#\) gives
\(r\perp_{\mathfrak m_0^\#}\cN_0^\#\), hence \(r=0\), contradicting
\(\norm{w_n}_{H^1}=1\).  Thus the Korn constant is uniform.  The lower-order
mass and DtN terms are uniformly small on the stated frequency scale.  Hence
\(\widehat G_{\delta,\eta}\) is uniformly invertible, jointly holomorphic in
\(\zeta\) and \(C^1\) in \(\eta\), and
\[
  \partial_\eta\widehat G_{\delta,\eta}^{-1}
  =-\widehat G_{\delta,\eta}^{-1}
    (\partial_\eta\widehat G_{\delta,\eta})
    \widehat G_{\delta,\eta}^{-1}.
\]
Applying the fixed-space block formulas of Section~\ref{sec:reduction} now proves the
asserted regularity and joint uniformity of
\(\cM_{\delta,\eta}(\zeta)\).
\end{proof}

\begin{proposition}
\label{prop:dimer-geometric-perturbation}
Let \(a_-\) be the exchange-odd unit eigenvector in
Proposition~\ref{prop:dimer}, with eigenvalue \(\lambda_-\).  Assume that
\(\lambda_-\) is a simple isolated eigenvalue of the full matrix \(K_0\) and
that
\[
  \inner{\Gamma_{1,0}a_-}{a_-}>0.
\]
Suppose the deformation breaks the exchange symmetry while preserving a
positive separation.  Then, for small
\(|\eta|\), the pulled-back static operators and the corresponding normalized
eigenpair \((\lambda_-(\eta),a_-(\eta))\), with phase fixed by
\(\inner{a_-(\eta)}{a_-(0)}>0\), are \(C^1\) in \(\eta\).  Moreover,
\begin{equation}
\label{eq:q1-shape-derivative}
  \mathfrak q_\eta B_\eta a_-(\eta)
  =
  \eta q_1+o(\eta),
  \qquad
  q_1
  =
  \left.
  \frac{\dd}{\dd\eta}
  \bigl(\mathfrak q_\eta B_\eta a_-(\eta)\bigr)
  \right|_{\eta=0}.
\end{equation}
Writing a dot for \(\partial_\eta|_{\eta=0}\), one has the explicit
first-variation formula
\begin{equation}
\label{eq:q1-first-variation}
  q_1
  =\dot{\mathfrak q}_0B_0a_-
  +\mathfrak q_0\dot B_0a_-
  -\mathfrak q_0B_0
  \mathcal R_-
  (I-P_-)\dot K_0a_-,
\end{equation}
where \(P_-=a_-\otimes a_-^*\) and
\[
  \mathcal R_-
  =\left((K_0-\lambda_-I)|_{a_-^\perp}\right)^{-1}:
  a_-^\perp\to a_-^\perp.
\]
Thus the transversality question is reduced to the first shape derivatives in
\eqref{eq:q1-first-variation}.  Equations
\eqref{eq:shape-derivative-single-layer},
\eqref{eq:shape-derivative-traction-kernel},
\eqref{eq:shape-derivative-dtn},
\eqref{eq:shape-derivative-capacitance}, and
\eqref{eq:inverse-square-root-derivative} express these derivatives through
boundary kernels and finite matrices.  No generic nonvanishing is asserted.
If \(q_1\ne0\), let \(P_-(\eta)\) be the orthogonal projection onto
\(a_-(\eta)\).  On
\[
  |\zeta-\delta^{1/2}\lambda_-(\eta)^{1/2}|\le c\delta
\]
there is a unique holomorphic Lyapunov--Schmidt graph vector, normalized by
\begin{equation*}
\begin{split}
  &(I-P_-(\eta))\cM_{\delta,\eta}(\zeta)
  a_{-,\delta}(\eta,\zeta)=0,
  \qquad
  \inner{a_{-,\delta}(\eta,\zeta)}{a_-(\eta)}=1,\\
  &m_{-,\delta,\eta}(\zeta)
  =
  \inner{\cM_{\delta,\eta}(\zeta)
  a_{-,\delta}(\eta,\zeta)}{a_-(\eta)}.
\end{split}
\end{equation*}
There is a unique real center in
\(|\omega-\delta^{1/2}\lambda_-(\eta)^{1/2}|\le C\delta^{3/2}\), defined
exactly by
\[
  \operatorname{Re}m_{-,\delta,\eta}
  (\widehat\omega_-(\delta,\eta))=0,
  \qquad
  \widehat\omega_-(\delta,\eta)^2
  =\delta\lambda_-(\eta)+O(\delta^2),
\]
uniformly for small \(\eta\).  At that center the graph vector satisfies
\begin{equation}
\label{eq:dimer-graph-correction}
  a_{-,\delta}(\eta,\widehat\omega_-)
  =a_-(\eta)+r_{\delta,\eta}(\widehat\omega_-),
  \qquad
  \norm{r_{\delta,\eta}(\widehat\omega_-)}
  \le C\bigl(\delta^{1/2}|\eta|+\delta\bigr).
\end{equation}
The radiated power of this corrected branch has the joint expansion
\begin{equation}
\label{eq:dimer-power-joint}
\begin{split}
  \norm{
    \cF_\eta(\widehat\omega_-)B_{\delta,\eta}(\widehat\omega_-)
    a_{-,\delta}(\eta,\widehat\omega_-)
  }_{\cZ}^2
  & =
  \gamma_{\mathrm{el}}\eta^2|q_1|^2
  +
  \widehat\omega_-^2\inner{\Gamma_{1,0}a_-}{a_-}\\
  &\quad+
  o(\eta^2+\delta),
\end{split}
\end{equation}
jointly as \((\eta,\delta)\to(0,0)\).  In particular, for some \(c>0\),
\begin{equation}
\label{eq:dimer-power-lower-bound}
  \norm{
    \cF_\eta(\widehat\omega_-)B_{\delta,\eta}(\widehat\omega_-)
    a_{-,\delta}(\eta,\widehat\omega_-)
  }_{\cZ}^2
  \ge c(\eta^2+\delta).
\end{equation}
The scalar
\(m_{-,\delta,\eta}\) has exactly one simple zero
\(\zeta_-(\delta,\eta)\) in a disk
\[
  |\zeta-\widehat\omega_-(\delta,\eta)|
  \le C\delta(\eta^2+\delta),
\]
and
\begin{equation}
\label{eq:dimer-crossover}
  \zeta_-(\delta,\eta)
  =
  \widehat\omega_-(\delta,\eta)
  -\frac{\ii\delta\eta^2}{2}
  \gamma_{\mathrm{el}}|q_1|^2
  -\frac{\ii\delta^2\lambda_-}{2}
  \inner{\Gamma_{1,0}a_-}{a_-}
  +o(\delta\eta^2+\delta^2).
\end{equation}
Here \(\lambda_-=\lambda_-(0)\), \(a_-=a_-(0)\), and
\(\Gamma_{1,0}\) is the unperturbed first higher-order radiation matrix.
Thus the symmetry-breaking and intrinsic dark widths balance when
\(|\eta|\asymp\delta^{1/2}\).
\end{proposition}

\begin{proof}
Lemma~\ref{lem:dimer-shape-regularity} places all moving operators on fixed
spaces and gives their \(C^1\) dependence.  Since \(\lambda_-\) is simple and isolated in the
full rigid space, standard finite-dimensional spectral perturbation theory
\cite{Kato1995} gives a \(C^1\) normalized eigenpair.  Exchange oddness at \(\eta=0\) gives
\(\mathfrak q_0B_0a_-(0)=0\), and Taylor expansion proves
\eqref{eq:q1-shape-derivative}.  The chosen phase gives
\(a_-^*\dot a_-=0\).  Differentiating
\(K_\eta a_-(\eta)=\lambda_-(\eta)a_-(\eta)\) and projecting onto
\(a_-^\perp\) yields
\[
  \dot a_-
  =-\mathcal R_-(I-P_-)\dot K_0a_-.
\]
Differentiation of
\(\mathfrak q_\eta B_\eta a_-(\eta)\), using the formula for
\(\dot{\mathfrak q}_0\) in Lemma~\ref{lem:dimer-shape-regularity}, gives
\eqref{eq:q1-first-variation}.

It remains to justify both the corrected branch vector and the two-parameter
remainder.  Decompose the rigid space as
\(\Span\{a_-(\eta)\}\oplus a_-(\eta)^\perp\).  On the second factor the
static part of the effective matrix is
\(\delta(K_\eta-\lambda_-(\eta)I)\), whose inverse has norm
\(O(\delta^{-1})\) uniformly for small \(\eta\), by the full-space spectral
gap.  The force factorization and \eqref{eq:q1-shape-derivative} give
\[
  \Gamma_{0,\eta}a_-(\eta)=O(|\eta|),
\]
so the off-diagonal radiative forcing of this vector is
\(O(\delta\omega|\eta|)\); the effective-matrix remainder is
\(O(\delta^2)\) on the resonant scale.  Solving the complementary equation
therefore gives
\[
  \norm{r_{\delta,\eta}(\omega)}
  \le
  C\delta^{-1}
  \bigl(\delta\omega|\eta|+\delta^2\bigr)
  \le
  C\bigl(\delta^{1/2}|\eta|+\delta\bigr),
\]
which proves \eqref{eq:dimer-graph-correction}, uniformly on the indicated
branch disk.  The same complementary inversion constructs the graph vector
holomorphically in \(\zeta\).

We record the derivative estimate needed for the joint zero count.  Relative
to \(\Span\{a_-(\eta)\}\oplus a_-(\eta)^\perp\), write
\[
  \cM_{\delta,\eta}(\zeta)
  =\begin{pmatrix}\alpha&\beta\\ \gamma&\mathsf D\end{pmatrix}.
\]
On the branch disk, Theorem~\ref{thm:effective-expansion} and
Lemma~\ref{lem:dimer-shape-regularity}, including their uniform
second-derivative estimates, give
\begin{align*}
  &\norm{\mathsf D^{-1}}=O(\delta^{-1}),\qquad
  \norm{\mathsf D'}=O(\delta^{1/2}),\qquad
  \norm{\mathsf D''}=O(1),\\
  &\norm{(\mathsf D^{-1})'}=O(\delta^{-3/2}),\qquad
  \norm{(\mathsf D^{-1})''}=O(\delta^{-2}),\\
  &\norm{\beta}+\norm{\gamma}
    =O(\delta^{3/2}|\eta|+\delta^2),\\
  &\norm{\beta'}+\norm{\gamma'}
    =O(\delta|\eta|+\delta^{3/2}),\qquad
  \norm{\beta''}+\norm{\gamma''}=O(\delta).
\end{align*}
Here the bounds are uniform in \(\eta\); they use
\(\Gamma_{0,\eta}a_-(\eta)=O(|\eta|)\).  Since
\[
  m_{-,\delta,\eta}=\alpha-\beta\mathsf D^{-1}\gamma,
\]
differentiating this Schur complement twice yields
\begin{equation}
\label{eq:dimer-scalar-derivative-bounds}
  \partial_\zeta m_{-,\delta,\eta}(\zeta)
  =-2\zeta+O(\delta),
  \qquad
  \partial_\zeta^2m_{-,\delta,\eta}(\zeta)
  =-2+O(\delta)
\end{equation}
uniformly on the branch disk.  In particular, its quadratic Taylor remainder
has a parameter-independent \(O(|\zeta-\widehat\omega_-|^2)\) bound.

If
\(\omega_{-,0}(\delta,\eta)=\delta^{1/2}\lambda_-(\eta)^{1/2}\), the associated
scalar satisfies on the real axis
\[
  \partial_\omega\operatorname{Re}m_{-,\delta,\eta}(\omega)
  =-2\omega+O(\delta)
\]
uniformly for \(|\omega-\omega_{-,0}|\le C\delta^{3/2}\), while
\(\operatorname{Re}m_{-,\delta,\eta}(\omega_{-,0})=O(\delta^2)\).
The real implicit-function theorem gives the unique center in the statement
and \(\widehat\omega_-^2=\delta\lambda_-(\eta)+O(\delta^2)\).

Lemma~\ref{lem:dimer-shape-regularity} also gives, uniformly for small \(\eta\),
\[
  \cF_\eta(\omega)
  =
  \cF_{0,\eta}
  +
  \omega\cF_{1,\eta}
  +
  O(\omega^2),
  \qquad
  B_{\delta,\eta}(\omega)=B_\eta+O(\delta).
\]
For every fixed small \(\eta\), the optical identity gives the coefficient
orthogonality
\[
  \cF_{0,\eta}^*\cF_{1,\eta}
  +
  \cF_{1,\eta}^*\cF_{0,\eta}=0.
\]
With \(x_\eta=B_\eta a_-(\eta)\), it follows that the real cross term between
\(\cF_{0,\eta}x_\eta\) and
\(\omega\cF_{1,\eta}x_\eta\) vanishes.  Moreover,
\[
  \norm{\cF_{0,\eta}x_\eta}_{\cZ}^2
  =
  \gamma_{\mathrm{el}}
  |\mathfrak q_\eta B_\eta a_-(\eta)|^2
  =
  \gamma_{\mathrm{el}}\eta^2|q_1|^2+o(\eta^2),
\]
while continuity in \(\eta\) gives
\[
  \norm{\cF_{1,\eta}x_\eta}_{\cZ}^2
  =
  \inner{\Gamma_{1,0}a_-}{a_-}+o(1).
\]
At \(\omega=\widehat\omega_-\), the correction \(r_{\delta,\eta}\)
contributes to the far-field amplitude at
most
\[
  O(\delta^{1/2}|\eta|+\delta)
  =O(\widehat\omega_-|\eta|+\widehat\omega_-^2),
\]
because \(c\delta\le\widehat\omega_-^2\le C\delta\).  Its quadratic
contribution and its cross term with the leading amplitude
\(O(|\eta|+|\widehat\omega_-|)\) are therefore
\(o(\eta^2+\delta)\).  The same is true for the terms produced by the
\(O(\omega^2)+O(\delta)\) kernel and boundary-trace remainders; for example,
\(|\eta|\omega^2\), \(\omega^3\), and
\(\delta(|\eta|+|\omega|)\) are all little-o of
\(\eta^2+\delta\) at \(\omega=\widehat\omega_-\).  For instance,
\[
  \frac{\delta|\eta|}{\eta^2+\delta}
  \le \frac12\delta^{1/2},
  \qquad
  \frac{\delta^{3/2}}{\eta^2+\delta}
  \le\delta^{1/2}.
\]
Replacing the static vector by the corrected graph vector therefore proves
\eqref{eq:dimer-power-joint}.
Since \(q_1\ne0\),
\(\inner{\Gamma_{1,0}a_-}{a_-}>0\), and
\(\widehat\omega_-^2\ge c\delta\), the same expansion gives
\eqref{eq:dimer-power-lower-bound} after decreasing the joint neighborhood.

Finally, the complementary equation and the normalization give
\[
  \cM_{\delta,\eta}(\widehat\omega_-)
  a_{-,\delta}(\eta,\widehat\omega_-)
  =m_{-,\delta,\eta}(\widehat\omega_-)a_-(\eta),
\]
and hence, with the inner-product convention of this paper,
\[
  \inner{
    \cM_{\delta,\eta}(\widehat\omega_-)
    a_{-,\delta}(\eta,\widehat\omega_-)
  }{a_{-,\delta}(\eta,\widehat\omega_-)}
  =m_{-,\delta,\eta}(\widehat\omega_-).
\]
The exact effective optical identity therefore gives its imaginary part as
\(-\delta\widehat\omega_-\) times the power in
\eqref{eq:dimer-power-joint}, whereas
\[
  \partial_\zeta m_{-,\delta,\eta}(\widehat\omega_-)
  =-2\widehat\omega_-+o(\delta^{1/2})
\]
uniformly in the same joint regime.  Put
\(R_{\delta,\eta}=C_0\delta(\eta^2+\delta)\).  The optical identity,
\eqref{eq:dimer-power-joint}, and \eqref{eq:dimer-power-lower-bound} imply
\[
  c\delta^{3/2}(\eta^2+\delta)
  \le
  |m_{-,\delta,\eta}(\widehat\omega_-)|
  \le
  C\delta^{3/2}(\eta^2+\delta).
\]
On \(|\zeta-\widehat\omega_-|=R_{\delta,\eta}\), the linear term has
size comparable to
\(C_0\delta^{3/2}(\eta^2+\delta)\), whereas the quadratic Taylor remainder
is uniformly \(O(R_{\delta,\eta}^2)\) by
\eqref{eq:dimer-scalar-derivative-bounds}.  Choose \(C_0\) larger than the ratio in the
preceding upper bound and then let \((\delta,\eta)\to(0,0)\).  Rouch\'e's
theorem, applied against the linear term, gives exactly one simple zero in the
stated disk.  The Newton quotient then yields
\[
  \zeta_--\widehat\omega_-
  =-
  \frac{m_{-,\delta,\eta}(\widehat\omega_-)}
       {\partial_\zeta m_{-,\delta,\eta}(\widehat\omega_-)}
  +o(\delta\eta^2+\delta^2).
\]
Using \(\widehat\omega_-^2=\delta\lambda_-(\eta)+O(\delta^2)\), and observing
that the \(C^1\) variation of \(\lambda_-(\eta)\), \(\Gamma_{1,\eta}\), and
the dipole coefficient contributes only to the displayed little-o remainder,
proves \eqref{eq:dimer-crossover}.  Uniformity in the last assertion follows,
for example, from
\[
  \frac{\delta^2|\eta|}{\delta\eta^2+\delta^2}
  =\frac{\delta|\eta|}{\eta^2+\delta}
  \le\frac12\delta^{1/2}.
\]
\end{proof}

\begin{remark}
A vector field \(V\) supported near one component, including a local
translation, supplies an admissible symmetry-breaking family.  Establishing
\(q_1\ne0\) for any particular family still requires evaluating
\eqref{eq:q1-first-variation}; symmetry breaking alone does not imply
transversality.  If the unperturbed eigenvalue is degenerate in the full rigid
space, one must first perturb the associated spectral projection and
rediagonalize inside that finite-dimensional cluster.
\end{remark}

\section{Concluding Remarks and Scope}\label{sec:conclusion}

The exact finite-rank reduction identifies every loss in the uniform LAP with
the effective matrix.  Its optical identity explains why a cluster has only
three leading force channels, even though its rigid space has dimension
\(6N\).  The sphere realizes both radiative scales explicitly.  For the
dimer, the first-variation criterion separates the intrinsic dark width from
the width created by symmetry breaking.

The result has two stated limits.  Modes in
\(\Ker\Gamma_0\cap\Ker\Gamma_1\) require a higher multipole expansion.  The
geometric constants also require a fixed positive component separation.  A
simultaneous closing-gap limit must track the capacitance and boundary-operator
bounds in both parameters and may exhibit stress concentration
\cite{LiXu2026,LiXuYang2026}; it is not covered here.

\appendix

\section{Low-Frequency DtN Expansion in Operator Norm}
\label{app:dtn}

\begin{lemma}
\label{lem:kernel-expansion}
For \(x\ne0\), the Kupradze tensor admits
\[
  \mathbf G^\zeta(x)
  =
  \mathbf G^0(x)
  +
  \ii\zeta\,\mathbf G_1
  +
  \zeta^2\mathbf G_2^\zeta(x),
\]
where \(\mathbf G_1\) is a constant matrix and
\(\mathbf G_2^\zeta(x)\) is holomorphic in \(\zeta\).  For
every fixed \(R>1\), every \(|\zeta|\le r_0\), and each multiindex
\(\beta\) required by the single-layer and traction traces,
\[
  |\partial_\zeta^k\partial_x^\beta\mathbf G_2^\zeta(x)|
  \le
  C_{\beta,R}
  \begin{cases}
    |x|^{-1-|\beta|},&0<|x|\le1,\\
    1,&1\le|x|\le R
  \end{cases}
\]
for \(k=0,1,2\), with a constant independent of \(\zeta\).
\end{lemma}

\begin{proof}
Use
\[
  \Phi_{k}(x)
  =
  \frac1{4\pi|x|}
  +
  \frac{\ii k}{4\pi}
  -
  \frac{k^2|x|}{8\pi}
  -
  \frac{\ii k^3|x|^2}{24\pi}
  +
  O(k^4|x|^3).
\]
In
\[
  \frac1{\rho_0\zeta^2}
  \nabla\nabla^{\mathsf T}
  (\Phi_{k_s(\zeta)}-\Phi_{k_p(\zeta)}),
\]
the constant terms disappear under two derivatives, while the
\(\zeta^2|x|\) term produces the static Kelvin tensor.  The remaining
linear term in \(\zeta\) is independent of \(x\); all higher terms have
the displayed weak singularity after differentiation.  Taylor's formula
with integral remainder gives holomorphy and the same estimates after up
to two \(\zeta\)-derivatives.
\end{proof}

\begin{lemma}
\label{lem:boundary-operator-expansion}
On a \(C^{2,\alpha}\) boundary,
\begin{align*}
  \mathsf S(\zeta)
  &=
  \mathsf S_0+\ii\zeta\mathsf S_1+
  \zeta^2\mathsf S_2(\zeta)
  &&\text{in }
  \mathcal B(H^{-1/2},H^{1/2}),\\
  \mathsf K^*(\zeta)
  &=
  \mathsf K_0^*+\zeta^2\mathsf K_2^*(\zeta)
  &&\text{in }
  \mathcal B(H^{-1/2},H^{-1/2}),
\end{align*}
where the remainders are holomorphic and uniformly bounded.
After decreasing \(r_0\), if necessary, their first two \(\zeta\)-derivatives are
uniformly bounded in the same operator norms.
\end{lemma}

\begin{proof}
The static kernels are the standard pseudohomogeneous kernels of order
\(-1\) and \(0\), respectively.  Lemma~\ref{lem:kernel-expansion} shows that
the remainder kernels have the
same or better singularity, uniformly together with their first two
\(\zeta\)-derivatives.  The mapping theorem for
pseudohomogeneous boundary kernels \cite{McLean2000} gives the asserted
Sobolev bounds.
The term linear in \(\zeta\) in the fundamental tensor is constant in
space; its traction vanishes, which explains the absence of a linear
term in \(\mathsf K^*(\zeta)\).
\end{proof}

\begin{proposition}
\label{prop:dtn-detailed}
The expansion \eqref{eq:dtn-basic-expansion} holds in
\(\mathcal B(H^{1/2}(\Sigma)^3,H^{-1/2}(\Sigma)^3)\), with a remainder
that can be differentiated twice in \(\zeta\) without loss of
uniformity.
\end{proposition}

\begin{proof}
The static single-layer operator is Fredholm of index zero.  If
\(\mathsf S_0\mu=0\), the associated single-layer potential has zero
trace.  Its interior and decaying exterior energies vanish by Green's
formula; the traction jump then gives \(\mu=0\).  Hence
\(\mathsf S_0\) is invertible.

Lemma~\ref{lem:boundary-operator-expansion} and the identity
\[
  \mathsf S(\zeta)^{-1}
  =
  \mathsf S_0^{-1}
  -
  \mathsf S_0^{-1}
  \bigl(\mathsf S(\zeta)-\mathsf S_0\bigr)
  \mathsf S_0^{-1}
  +
  O(\zeta^2)
\]
give an operator-norm Taylor expansion of the inverse.  Multiplication by
\(-\frac12I+\mathsf K^*(\zeta)\) yields the DtN expansion.  All factors
are holomorphic, so the differentiated remainder obeys the same
uniform bound.  The optical identity identifies the Hermitian
coefficient \(\Lambda_1\) and proves its positivity.
\end{proof}

\section*{Declarations}

\paragraph{Funding.}
This work was partially supported by the National Natural Science Foundation of China
(Grant No.~12371187).

\paragraph{Competing interests.}
The author has no relevant financial or non-financial interests to disclose.

\paragraph{Author contributions.}
The author conceived the study, developed the analysis, and wrote and revised
the manuscript.

\paragraph{Data availability.}
No datasets were generated or analyzed for this study.

\paragraph{AI-assisted editing.}
AI-assisted tools were used for language editing and consistency checks.  The
author verified the mathematical content and references and assumes full
responsibility for the manuscript.

\end{document}